\documentclass[12pt]{article}
\usepackage[final]{epsfig}
\usepackage{graphics}
\usepackage{amsthm}
\usepackage{amsmath}
\usepackage{amsfonts}
\usepackage{latexsym}
\usepackage{amssymb}
\usepackage{graphicx}
\usepackage{url}
\usepackage{epstopdf}

\usepackage{color}

\usepackage[matrix,arrow]{xy}

\newtheorem{lemma}{Lemma}[section]
\newtheorem{proposition}[lemma]{Proposition}
\newtheorem{theorem}{Theorem}

\newtheorem{corollary}[lemma]{Corollary}

\newtheorem{remark}[lemma]{Remark}
\newtheorem{example}[lemma]{Example}
\newtheorem{definition}[lemma]{Definition}

\makeatletter
\renewenvironment{proof}[1][\proofname] 
{\par\pushQED{\qed}\normalfont\topsep6\p@\@plus6\p@\relax\trivlist\item[\hskip\labelsep\bfseries#1\@addpunct{.}]\ignorespaces}{\popQED\endtrivlist\@endpefalse}
\makeatother

\newcommand{\proofend}{$\Box$\bigskip}

\newcommand{\C}{\mathbb{C}}

\newcommand{\R}{\mathbb {R}}
\newcommand{\Z}{\mathbb{Z}}

\newcommand{\B}{\mathcal{B}}
\newcommand{\Pev}{\mathcal{P}}
\newcommand{\Mev}{\mathcal{M}}
\newcommand{\Aev}{\mathcal{A}}
\newcommand{\Aec}{\mathcal{A}_c}
\newcommand{\Aeo}{\mathcal{A}_o}
\newcommand{\E}{\mathcal{E}}
\renewcommand{\P}{\mathbf{P}}
\newcommand{\g}{{\gamma}}
\newcommand{\G}{{\Gamma}}
\newcommand{\tr}{\operatorname{Tr}}
\newcommand{\Ao}{A_-}
\newcommand{\Ae}{A_+}
\newcommand{\lin}{\operatorname{span}}
\newcommand{\id}{\operatorname{id}}
\newcommand{\diag}{\operatorname{diag}}

\newcommand{\St}{\operatorname{St}}
\newcommand{\PS}{\operatorname{PS}}

\renewcommand{\mod}{\operatorname{mod}}

\newcommand{\rk}{\operatorname{rank}}

\newcommand{\vol}{\operatorname{vol}}
\newcommand{\dvol}{\operatorname{dvol}}
\newcommand{\Sph}{\mathbb{S}}

\graphicspath{{pictures/}}

\title{Iterating evolutes and involutes}

\author{Maxim Arnold\footnote{
Department of Mathematics, 
University of Texas, 
800 West Campbell Road,
Richardson, TX 75080;
maxim.arnold@gmail.com}
\and
Dmitry Fuchs\footnote{
Department of Mathematics, 
University of California, 
Davis, CA 95616;
 fuchs@math.ucdavis.edu}
\and
Ivan Izmestiev\footnote{
Department of Mathematics, 
University of Fribourg,
Chemin du Mus\'ee 23,
CH-1700 Fribourg;
ivan.izmestiev@unifr.ch 
}
\and
Serge Tabachnikov\footnote{
Department of Mathematics,
Pennsylvania State University,
University Park, PA 16802;
tabachni@math.psu.edu}
\and
Emmanuel Tsukerman\footnote{
Department of Mathematics, 
University of California,
Berkeley, CA 94720-3840;
e.tsukerman@berkeley.edu}
}

\date{}
\begin{document}
\maketitle

\begin{abstract}
We study iterations of two classical constructions, the evolutes and involutes of plane curves, and we describe the limiting behavior of both constructions on a class of smooth curves with singularities given by their support functions.  

Next we study two kinds of discretizations of these constructions: the curves are replaced by polygons, and the evolutes are formed by the circumcenters of the triples of consecutive vertices, or by the incenters of the triples of consecutive sides. The space of polygons is a vector bundle over the space of the side directions, and both kinds of evolutes define vector bundle morphisms. 
In both cases, we describe the linear maps of the fibers. In the first case, the induced map of the base is periodic, whereas, in the second case, it is an averaging transformation. We also study the dynamics of the related inverse constructions, the involutes of polygons.

In addition to the theoretical study, we performed numerous computer experiments; some of the observations remain unexplained.
\end{abstract}

\setcounter{tocdepth}{2}
\tableofcontents

\section{Introduction }\label{Intro}

\subsection{Classical evolutes and involutes}\label{classical}
Recall basic facts about evolutes and involutes of smooth plane curves. The reader is referred to his favorite book on elementary differential geometry or a treatise on mathematical analysis. In particular, the material outlined below is contained in \cite{FT}.

Let $\g$ be a smooth plane curve. Its {\it evolute}, $\G$, is the envelope of  the family of normal lines to $\g$. See Figure \ref{evolutedef}.
Equivalently, the evolute is the  locus of centers of osculating circles to $\g$.  Equidistant curves share the evolute. 

\begin{figure}[hbtp]
\centering
\includegraphics[height=2.5in]{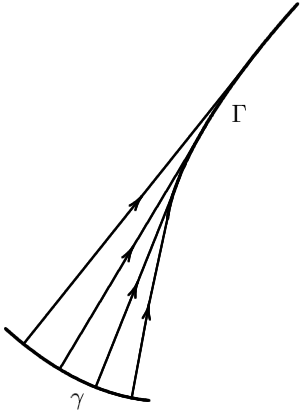}\quad\quad\quad
\includegraphics[height=2.5in]{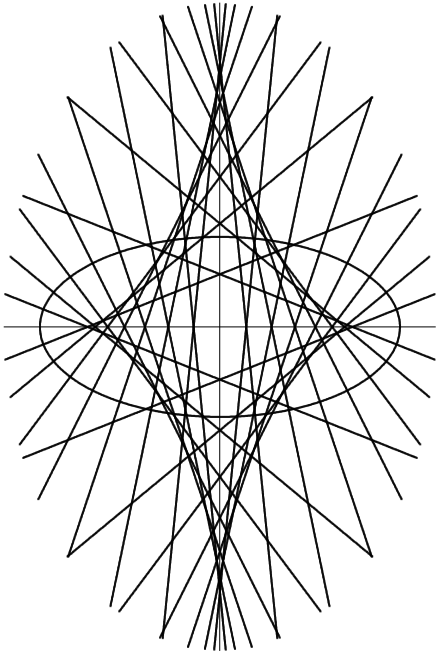}
\caption{Left: $\G$ is the evolute of $\g$, and $\g$ is an involute of $\G$. Right: the evolute of an ellipse.}
\label{evolutedef}
\end{figure}

The evolute is always locally convex: it has no inflection points. If the curve $\g$ has an inflection point then its evolute at this point escapes to infinity. We shall be mostly concerned with locally convex curves.

Typically, an evolute has cusp singularities, generically, semicubical, clear\-ly visible in Figure \ref{evolutedef}. These cusps of $\G$ correspond to the vertices of $\g$, the local extrema of its curvature. 

The evolute of a closed curve has total length zero. The length is algebraic: its sign changes each time that one passes a cusp.

We shall consider a larger class of curves with cusp singularities called {\it wave fronts}. A wave front is a piecewise smooth curve such that the tangent line is well defined at every point, including the singularities.\footnote{A more conceptual definition of a wave front: it is a projection to the plane of a smooth Legendrian curve in the space of contact elements in $\R^2$. We do not use this point of view in the present paper.} The construction of the evolute extends to wave fronts (the curvature at cusp points is infinite, and the evolute of a wave front passes through its cusps).

A coorientation of a wave front is a continuous choice of a normal direction. Unlike orientation, coorientation extends continuously through cusps. A cooriented closed wave front has an even number of cusps. A coorientation of $\g$ provides orientation, and hence coorientation, of its normal lines, and as a result, a coorientation of the evolute $\G$.

Thus one can consider a transformation on the set of cooriented wave fronts: curve $\gamma$ is mapped to its evolute $\Gamma$. We shall call this mapping the \emph{evolute transformation} and will use the notation $\G=\mathcal{E}(\g)$.

Consider the left Figure \ref{evolutedef} again. The curve $\g$ is called an {\it involute} of the curve $\G$: the involute  is orthogonal to the tangent lines of a curve. The involute $\g$ is described by the free end of a non-stretchable string whose other end is fixed on   $\G$ and which is wrapped around it (accordingly, involutes are also called evolvents). Changing the length of the string yields a 1-parameter family of involutes of a given curve.
 Having zero length of $\g$ is necessary and sufficient for the string construction to produce a closed curve.
 
 The relation between a curve and its evolute is similar to the relation between a periodic function and its derivative. A function has a unique derivative, but not every function is a derivative of a periodic function: for this, it should have zero average. And if it does, then the inverse derivative is defined up to a constant of integration. We shall see that this analogy is literally true for a certain class of plane wave fronts.

\subsection{Discretization}\label{discretization}
In this paper we  study two natural discretizations of the evolute construction to the case of polygons. Both are discretizations of the fact that the evolute of a curve is the locus of centers of its osculating circles. We consider our study as a contribution to the emerging field of discrete differential geometry \cite{BS,Ho}.

Replace the curve by a closed polygonal line $\P$. One may replace osculating circles by the circles that circumscribe the consecutive triples of vertices of $\P$. The centers of these circles form a new polygon, $\mathbf{Q}$, which may be regarded as an evolute of $\P$. To perform this construction, we need to assume that no three consecutive vertices of $\P$ are collinear. Since these centers are the intersection points of the perpendicular bisectors of the consecutive sides of $\P$, we call $\mathbf{Q}$ the  {\it perpendicular bisector evolute or $\Pev$-evolute} of $\P$. This  evolute was considered in the context of a discrete 4-vertex theorem \cite{Mo,Mu}. 

$\Pev$-evolute of a polygon shares one additional property of the classical evolutes: it borrows the coorientation from this polygon. Namely, a coorientation of a polygon consists of an arbitrary choice of coorientations (equivalently: orientations) of all sides. But the sides of $\mathbf Q$ are  perpendicular bisectors of sides of $\P$, so a choice of a coorientation of a side of $\mathbf Q$ is the same as a choice of an orientation of the corresponding side of $\P$.

Equally well, one may consider the circles that are tangent to the  consecutive triples of sides of $\P$.  In fact, for a given triple of lines, there are four such circles, so, if we make an arbitrary choice of these circles for every triple, then for an $n$-gon we will obtain as many as $4^n$ evolutes. 

To narrow this variety of choices, we notice that the center of a circle we consider is the intersection point of the angular bisectors of two consecutive angles of the polygon. Since an angle has two bisectors, we can just make a choice of one of these two bisectors at every vertex of our polygon $\P$. This gives us $2^n$ evolutes, which are all called angular bisector evolutes or $\Aev$-evolutes of $\P$. 

Thus, the sides of an $\Aev$-evolute are the chosen bisectors of the angles of $\P$. To perform this construction, we need to assume that no two consecutive vertices of $\P$ coincide and no two consecutive chosen bisectors are parallel. 

There are some choices of angle bisectors which may be regarded as natural. For example, one can take bisectors of {\it interior} angles (this choice is used in Figure \ref{hexagonexample}). Another choice is the choice of bisectors of {\it exterior} angles. Other natural choices arise when the sides of $\P$ are (co)oriented (see Figure \ref{OrBisector} below). In this approach, the $\Aev$-evolute of a polygon borrows a coorientation from this polygon.

Both versions of polygonal evolutes, along with the relative evolute of parallel polygons, are discussed in the book \cite{Pa}, see also \cite{Ta}. They are also introduced in \cite{Ho} under the names of vertex and edge evolutes.

For some (randomly chosen) hexagon $\P$ the $\Pev$- and $\Aev$-evolutes
are shown in Figure \ref{hexagonexample}.

\begin{figure}[hbtp]
	\centering
	\includegraphics[width=3.4in]{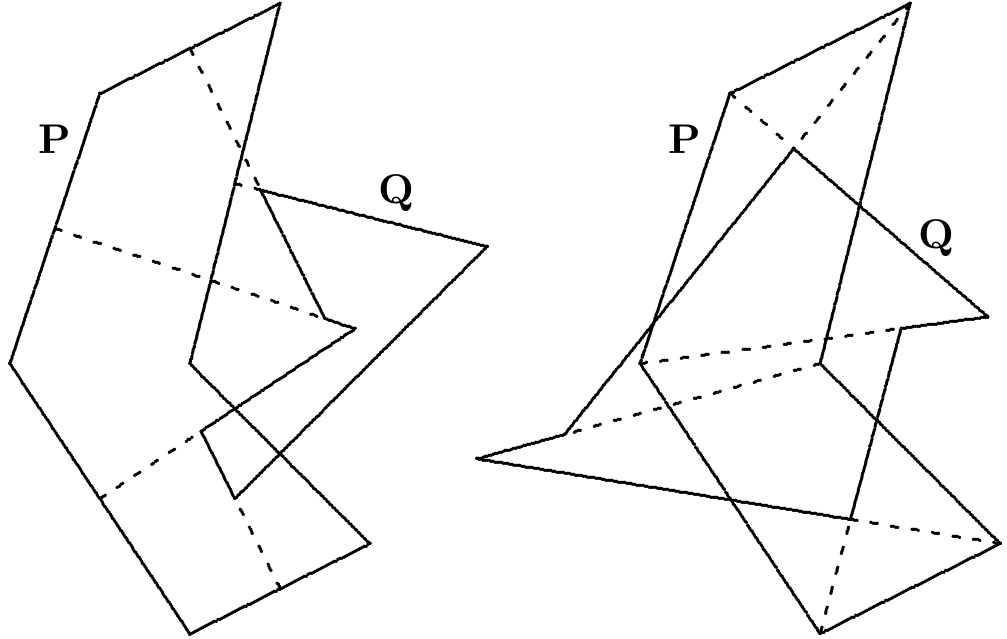}
	\caption{The two constructions of an evolute $\mathbf{Q}$ of a polygon $\P$.}
	\label{hexagonexample}
\end{figure}

It is seen that for a polygon with a small number of sides the two evolutes  may look completely different (but, certainly, for polygons approximating smooth curves, they will be close to each other).

A polygon whose evolute is $\mathbf Q$ is called an involute of $\mathbf Q$. A variety of definitions of the evolute gives rise to a variety of definitions of an involute. We will discuss this variety in Section \ref{AP-inv}. The polygonal involutes retain the most visible property of classical involutes: they exist only under an additional condition and are not unique.

\subsection{A survey of results of this paper}\label{survey}

The main subject of this paper is the dynamics of the evolute and involute transformations for smooth and polygonal curves in the plane.

Section \ref{smooth} is devoted to the smooth case. The curves in the plane which we study are closed cooriented wave fronts without inflections. The total turning angle of such curve is a multiple of $2\pi$. If it is $2\pi$, then we call the curve a {\it hedgehog}. If it is $2k\pi$ with $k>1$, then the curve is called a {\it multi-hedgehog}. 

It turns out that generically the number of cusps of consecutive evolutes grows without bounds. However, if the support function of a curve is a trigonometric polynomial of degree $n$, then the limiting shape of its iterated evolute is that of a hypocycloid of order $n$ (Theorem \ref{growth}).

A closed curve whose evolute is the given curve $\gamma$  is called an involute of $\gamma$. A curve $\gamma$ has an involute if and only if its total signed length is zero, and if this condition holds, then $\gamma$ has a one parameter family of involutes. However, for a generic curve $\gamma$ of zero signed length, precisely one involute also has  zero signed length. We call this involute an {\it evolvent}\footnote{As we mentioned earlier, in the standard terminology, ``evolvent" and ``involute" are synonymous.}, and the construction of the evolvent can be iterated. 

We observe that the iterated evolvents of a hedgehog converse to its curvature center, a.k.a. the Steiner point, and the limit shape of the iterated evolvents of a hedgehog is a hypocycloid, generically an astroid (Theorem \ref{limhypo}). We also provide analogs of these results for non-closed periodic curves, such as the classical cycloid. In particular, if such a curve differs  from its evolute by a parallel translation then it is the cycloid. (Theorem \ref{limcycl}).

Sections \ref{PolygonsHedgehogs} -- \ref{AP-inv} are devoted to polygonal versions of evolutes and involutes. It has been already pointed out in Section  \ref{discretization} that there are two competing approaches to the discretization of the notion of an evolute: for an $n$-gon $\P$, its $n$-gonal evolute can be defined as having for sides the perpendicular bisectors of the sides of $\P$, or the bisectors of the angles of $\P$. Accordingly, we speak of $\Pev$-evolutes and $\Aev$-evolutes. There arise two different theories based on $\Pev$- and $\Aev$-approaches. 

It was also noted in Section \ref{discretization} that while a $\Pev$-evolute of a given $n$-gon $\P$ is well defined and unique, the number of different $\Aev$-involutes of $\P$ is $2^n$ (or $2^{n-1}$, this depends on the details of the definition).

Section \ref{PolygonsHedgehogs} is technical: it provides a discretization of the support function (which becomes, for an $n$-gon, an $n$-tuple $(\alpha_1,p_1),\dots(\alpha_n,p_n)$ with $\alpha_i\in{\mathbb R}/2\pi{\mathbb Z},\, p_i\in\mathbb R$), their Fourier series, hypocycloids, and Steiner points (which in some cases we have to replace with geometrically less clear ``pseudo-Steiner points"). 

The $\Pev$-evolutes are studied in Section \ref{P-evo}. We notice that the $\Pev$-evolute transformation acts on  $n$-tuples $\{(\alpha_i,p_i)\}$ in the following way: $\alpha_i$ just becomes $\alpha_i+\dfrac{\pi}2$, while the components $p_i$ are transformed in a linear way (depending on $\alpha_i$): we give an explicit description of these transformations (Theorem \ref{CoordP}). 

First, we apply these results to the discrete hypocycloids (equiangular hedgehogs whose sides are tangent to a hypocycloid). In this case, we obtain a full geometric description of the $\Pev$-evolute transformation; in particular, the $\Pev$-evolute of a discrete hypocycloid is a discrete hypocycloid of the same order, magnified and rotated in an explicitly described way (Theorem \ref{EvolDHPerp}). 

For generic hedgehogs, we give  necessary and sufficient conditions for the existence of a pseudo-Steiner point, a point shared by a polygon and its $\Pev$-evolute and satisfying some natural assumptions. These necessary and sufficient conditions depend only on the directions of the sides of a polygon (Theorem \ref{thm:PsSt}).

Then we prove some partially known, partially conjectured, and partially new facts in the cases of small $n$. In particular, we prove B. Gr\"unbaum's conjecture that the third and the first $\Pev$-evolutes of a pentagon are homothetic.

After that, we pass to the main result of this section. It turns out that the spectrum of the $(n\times n)$-matrix of the transformation of $p$-components is symmetric with respect to the origin, and that this symmetry respects the sizes of the Jordan blocks of the matrix (Theorem \ref{SymSpectrum}). 

Consequently, all the eigenvalues of the square of the $\Pev$-evolute transformation have even multiplicities, generically, multiplicity 2. Since the sides of the second $\Pev$-evolute of a polygon $\P$ are parallel to the sides of $\P$, this second iteration becomes a linear transformation in the $n$-dimensional space of $n$-gons with prescribed directions of sides. 

The dynamics of this transformation is determined by the maximal (by absolute value) eigenvalue $\lambda_0$. Namely, if $\lambda_0$ is real, then the iterated $\Pev$-evolutes have two alternating limit shapes (see Figures \ref{positive} and \ref{negative}); if $\lambda_0$ is not real, then the limit behavior of the iterated $\Pev$-evolutes is more complicated, but still interesting.

Section \ref{A-evo} is devoted to $\Aev$-evolutes. There are two natural ways to specify one of the $2^n\ \Aev$-evolutes. The first  consists in choosing an $\Aev$-evolute of $\P$ using a coorientation of $\P$. We refer to the $\Aev$-evolute constructed in this way as to the $\Aev_o$-evolute. It is important that, at least generically, the $\Aev_o$-evolute of a cooriented polygon has a canonical coorientation, which makes it possible to iterate the construction of the $\Aev_o$-evolute. 

Another possibility is to take for the sides of the evolute of a polygon $\P$ the interior angle bisectors of $\P$; the evolute constructed in this way is called the $\Aev_c$-evolute.

The main properties of $\Aev_o$-evolutes are that the $\Aev_o$-evolute of a discrete hypocycloid is again a discrete hypocycloid (Theorem \ref{EvolDH}), and the iterated $\Aev_o$-evolutes of an arbitrary discrete hedgehog converge in the shape to a discrete hypocycloid (Theorem \ref{AngleEvolShapes}). By the way, Theorems \ref{EvolDHPerp} and \ref{EvolDH} show that the limit behaviors of $\Pev$- and $\Aev_o$-evolutes are similar, although the scaling factors are different.

We have no proven results on $\Aev_c$-evolutes, but some, rather mysterious, experimental results concerning $\Aev_c$-evolutes of pentagons may instigate future research.

Section \ref{AP-inv} is devoted to $\Aev$- and $\Pev$-involutes. It is easy to see that all kinds of discrete involutes of a polygon $\mathbf Q$ have to do with the properties of the composition $\mathcal S$ of the reflections of the plane in the consecutive sides of $\mathbf Q$. 

More precisely, $\Pev$-involutes correspond to fixed points of $\mathcal S$, while $\Aev$-involutes correspond to invariant lines of $\mathcal S$. This exhibits a big difference between the cases of even-gons and odd-gons: for an even-gon, $\mathcal S$ is either a rotation, or a parallel translation, or the identity, while for an odd-gon, it is either a reflection or a glide reflection. 

Thus, for a generic even-gon, the transformation $\mathcal S$ has a unique fixed point, but no invariant lines. So, a generic even-gon has a unique $\Pev$-involute (which we call the {\it $\Pev$-evolvent}) and no $\Aev$-involutes. On the contrary, a generic odd-gon has a unique $\Aev$-involute (this is called an {\it $\Aev_{\rm odd}$-evolvent}) and no $\Pev$-involutes. 

If, for an odd-gon $\mathbf Q$, the transformation $\mathcal S$ is a pure (not glide) reflection (analytically this means that a certain quantity, called ``quasiperimeter" and similar to the signed length of a smooth hedgehog, vanishes), then, in addition to the $\Aev_{\rm odd}$-evolvent, $\mathbf Q$ has a one-parameter family of $\Aev$-involutes, and, generically, one of them has zero quasiperimeter; this is the {\it $\Aev_{\rm even}$-evolvent}. 

The existence of a $\Pev$-involute for an odd-gon $\mathbf Q$ again requires the condition of vanishing quasiperimeter. Under this condition, $\mathbf Q$ has a one-parameter family of $\Pev$-involutes, and, for a generic $\mathbf Q$, precisely one of them has zero quasiperimeter; the corresponding $\Pev$-involute is the {\it $\Pev$-evolvent}. 

All kinds of evolvent constructions admit iterations. We have some special results concerning equiangular hedgehogs. Namely, the iterated $\Aev$-evolvents of an equiangular hedgehog of zero perimeter converge to its pseudo-Steiner point (Theorem \ref{AInvolEqui}). 
The iterated $\Pev$-evolvents of an equiangular hedgehog with zero perimeter for odd-gons, and zero sum of lengths of, separately, odd-numbered and even-numbered sides for even-gons, converge in shape to two discrete hedgehogs (Theorem \ref{nodd_Pinv_equi_limit}).
And the directions of the sides of iterated $\Aev_{\rm odd}$-evolvents with $n$ odd behave ergodically (Theorem \ref{ergodic}) (this generalizes the work of a number of authors on the dynamics of the map that sends a triangle to its pedal triangle).

The final section ends with a discussion of possible extensions of our results to hyperbolic and spherical geometries.

\subsection{Acknowledgments} We are grateful to B. Gr\"unbaum for 
sending us his papers \cite{Gr1,Gr2}, to T. Hoffmann for sharing his lecture notes  \cite{Ho} and for helpful remarks, and to D. Khavinson for providing reference \cite{PW} and for a useful discussion. We are also grateful to numerous other colleagues for their interest and encouragement.

This project was conceived and partially completed at ICERM, Brown University, during S. T.'s 2-year stay there and the other authors' shorter visits. We are grateful to ICERM for the inspiring, creative, and friendly atmosphere. 

D. F. is grateful to Max-Planck-Institut in Bonn, where part of the work was done, for its invariable hospitality.
Part of this work was done  while I. I.  held a Shapiro visiting position at the Pennsylvania State University; he is grateful to Penn State for its hospitality.  

I. I. was supported by ERC Advanced Grant, Project No.~247029 SDModels. 
S. T. was supported by NSF grants DMS-1105442 and  DMS-1510055.
E. T. was supported by NSF Graduate Research Fellowship under grant No. DGE 1106400. %Any opinion, findings, and conclusions or recommendations expressed in this material are those of the authors and do not necessarily reflect the views of the National Science Foundation.

\section{Evolutes and involutes of smooth hedgehogs} \label{smooth}

\begin{definition}
\label{hedgehog}
A cooriented closed plane wave front without inflections and with total turning $2\pi$ is called a \emph{hedgehog}.
\end{definition}

Hedgehogs are Minkowski differences of convex curves, see \cite{LLR,MM}. \smallskip

As we pointed out in Introduction, the evolute never has inflection points, and a coorientation of $\g$ provides a coorientation of the evolute $\G$. The total turning of $\G$ is the same as that of $\g$. Thus the evolute of a hedgehog is again a hedgehog and so we can speak about the mapping $\E :\g\mapsto \G$ on the space of hedgehogs. 

\subsection{Support functions} 
\label{smooth-support}

Choose an origin $O$ in the plane.
\begin{definition}
\label{def:support}
The signed distance from $O$ to the tangent line to a hedgehog $\g$ having the coorienting vector $(\cos\alpha,\sin\alpha)$ is called \emph{support function} of $\g$. We will denote it by $p_\g(\alpha)$. 
\end{definition}

The sign of $p_\g$ is positive if the perpendicular from $O$ to the tangent line has the same orientation as the coorienting vector, and negative otherwise. 

It is clear from definition that the support function of a hedgehog is $2\pi$-periodic function. 

Given a support function $p_\g(\alpha)$, the corresponding point of the curve is given by the formula (see, e.g., \cite{Sch})
\begin{equation}
\label{coordsup}
(x,y)=(p_\g(\alpha) \cos \alpha - p_\g'(\alpha) \sin \alpha,\  p_\g(\alpha) \sin \alpha + p_\g'(\alpha) \cos \alpha).
\end{equation}

Let us describe the support function of the evolute.  

\begin{lemma}
\label{evolsupp}
Let $p_\g(\alpha)$ be the support function of a curve $\g$. Then the support function of the evolute  is given by the formula $$p_{\E(\g)}(\alpha) = p_\g'\left(\alpha-\frac{\pi}{2}\right).$$
\end{lemma}

\begin{figure}[hbtp]
\centering
\includegraphics[height=1.5in]{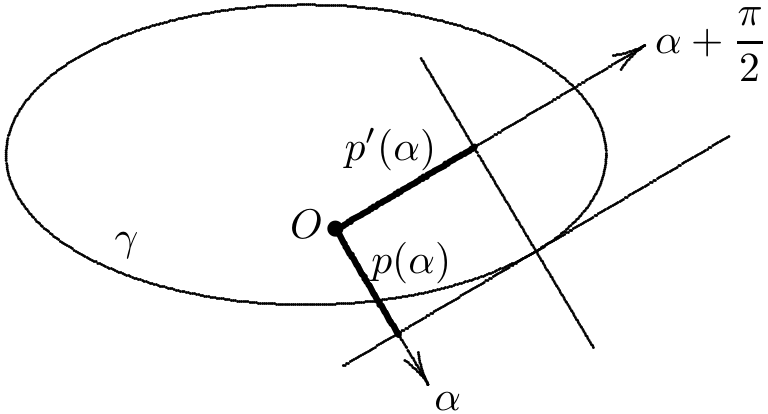}
\caption{The support function of the evolute.}
\label{support}
\end{figure}

\begin{proof}
It follows from formula (\ref{coordsup}) that the coorienting vector of the normal line to $\g$ at point $(x,y)$ is $(-\sin\alpha,\cos\alpha)$, and the signed distance from the origin to this normal  is $p'(\alpha)$, see Figure \ref{support}.
\end{proof}

Thus, iteration of evolute is differentiation of the support function, combined with the quarter-period shift of its argument.

The support function depends on the choice of the origin: parallel translation of the origin results in the addition to the support function a first harmonic, a linear combination of $\cos \alpha$ and $\sin \alpha$.

\begin{definition}
\label{Steiner} 
The \emph{Steiner point}  or the \emph{curvature centroid} of a curve $\g$ is the point
\begin{equation}
\label{Stptform}
\St(\g)=\frac{1}{\pi} \left(\int_0^{2\pi} p_\g(\alpha)\, (\cos \alpha,\sin\alpha)\, d\alpha\right).
\end{equation}
\end{definition}

It can be shown that $\St(\g)$ is the center of mass of the curve $\g$ with the density equal to the curvature.

Initially Steiner \cite{Ste} characterized $\St(\g)$ as the solution to the following problem. Roll the (convex) curve $\g$ over a line $l$; every point inside $\g$ describes a periodic curve; find among all points the one for which the area between its trajectory and the line $l$ is the smallest. 

The following observation was made by Kubota \cite{Ku}.
\begin{lemma}
\label{Stpt}
The support function of a hedgehog $\g$ is free of the first harmonics if and only if the origin is $\St(\g)$.
\end{lemma}

\begin{proof}
Formula (\ref{Stptform}) yields a zero vector if and only if $p_\gamma(\alpha)$ is $L_2$-orthogonal to the first harmonics.
\end{proof}

As an immediate corollary of Lemma \ref{Stpt} we get:

\begin{corollary}
\label{Stptevo}
Hedgehog and its evolute share the Steiner point: $\St(\g)=\St(\mathcal{E}(\g))$. 
\end{corollary}

\subsection{Limiting behavior of iterated evolutes} 
\label{smooth-cycloids}
Using the notation of the previous paragraph we can rephrase the object under consideration. The space of hedgehogs can now be thought of as the space of $2\pi$-periodic functions whose Fourier expansion starts with harmonics of order $2$ or higher. We  study two linear operators: the derivative $d$ and the inverse derivative $d^{-1}$ acting on this space. Since we are interested in the shape of  curves, we  consider periodic functions up to a non-zero factor (corresponding to a dilation) and shift of the argument (corresponding to a rotation).

\begin{figure}[hbtp]
\centering
\includegraphics[height=1.8in]{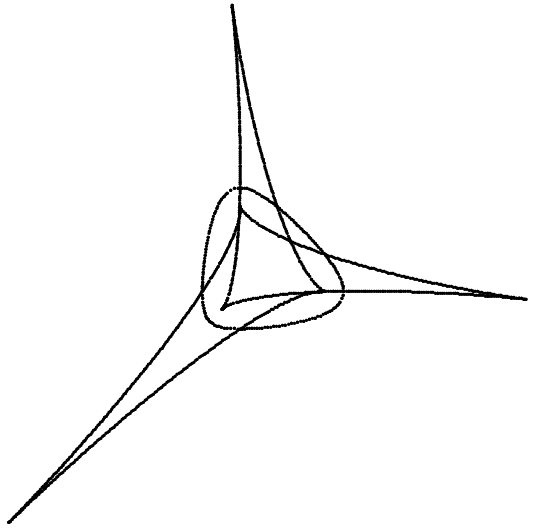}
\caption{The  curve $x=\exp(\cos t),\, y =\exp(\sin(t)$ and its evolute.}\label{6cusps}
\end{figure}

In computer experiments we observed that, typically, the number of cusps of consecutive evolutes increases without bound. Figure \ref{6cusps} shows the curve $x=\exp(\cos t),\, y =\exp(\sin t)$ and its evolute; the evolute has 6 cusps. Figure \ref{18cusps} shows the fifth evolute of the same curve. It has 18 cusps; to see them all, we had to present different magnifications of this evolute.

\begin{definition}
\label{hypocycloid}
A hedgehog whose support function is a pure harmonic of order $n$ is called a \emph{hypocycloid of order $n$} (see Figure \ref{hypo}). A hypocycloid of order 2 is also called the \emph{astroid}.
\end{definition}

\begin{figure}[hbtp]
\centering
\includegraphics[height=3.6in]{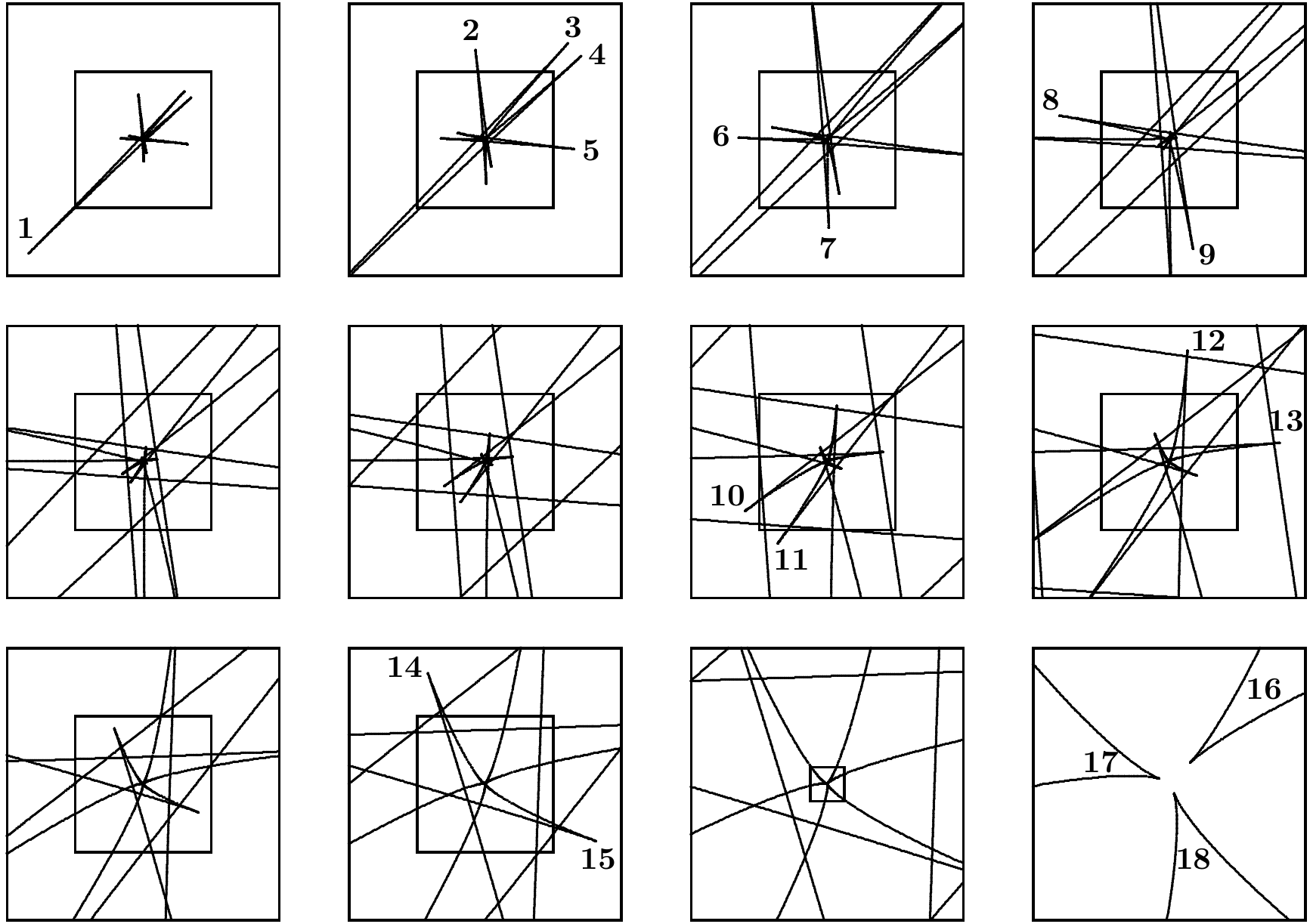}
\caption{Fifth evolute of the  curve $x=\exp(\cos t),\, y =\exp(\sin(t)$ with the cusps numerated.}\label{18cusps}
\end{figure}

The following theorem is the first  result of this work.

\begin{theorem}
\label{growth}
If the support function of a curve $\g$ is a trigonometric polynomial of degree $n$, then the limiting shape of its iterated evolute is a hypocycloid of order $n$. Otherwise, the number of cusps of the iterated evolutes increases without bound.
\end{theorem}

\begin{proof} The pure harmonics are  eigenvectors of the operator $d^2$ with the  harmonics of order $n$ having the eigenvalue  $-n^2$. Therefore the higher the harmonic, the faster it grows.  Hence if the support function is a trigonometric polynomial, we can rescale to keep the highest harmonic of a fixed size, and then all smaller harmonics will disappear in the limit. 
	
According to a theorem of Polya and Wiener \cite{PW}, if the number of zeros of all derivatives of a periodic function is uniformly bounded, then this function is a trigonometric polynomial. 
	
In terms of the support function, the radius of curvature of a curve is given by $p+p''$, and this function vanishes at cusps. Thus the number of cusps of the iterated evolutes is the number of zeros of the functions $p^{(n)} + p^{(n+2)}$. If this number is uniformly bounded then $p+p''$ is a trigonometric polynomial, and therefore so is $p$.\end{proof}

Geometric applications of the Fourier decomposition of the support 
function is a subject of the book \cite{Groe}.

\begin{figure}[hbtp]
\centering
\includegraphics[height=1.5in]{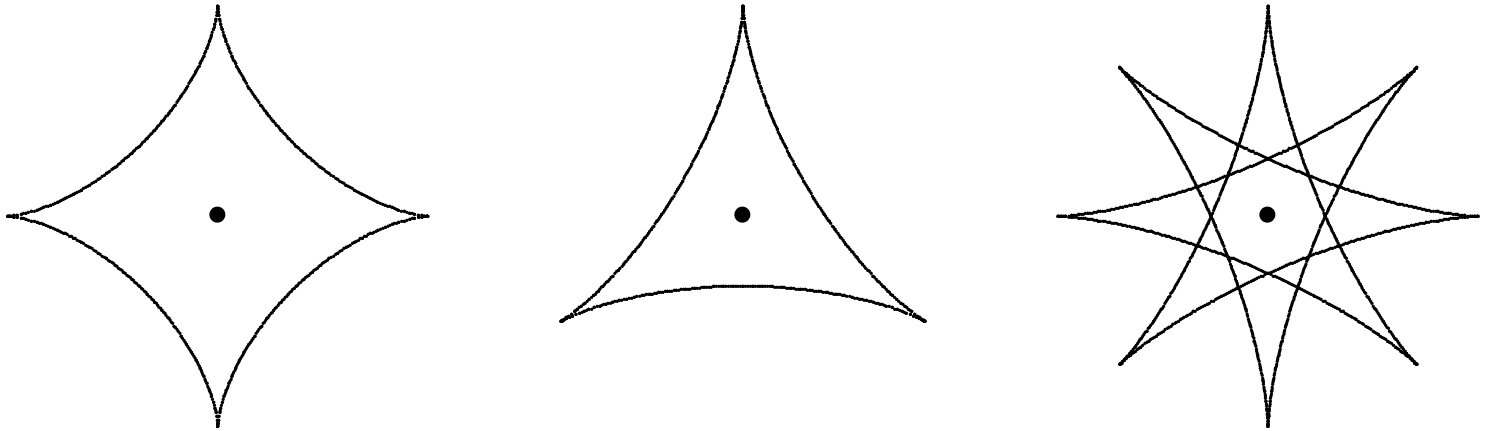}
\caption{Hypocycloids of order 2, 3, and 4 (the middle one is traversed twice).}
\label{hypo}
\end{figure}

As a corollary, we obtain a description of the curves that are similar to their 
evolutes.

\begin{corollary}
\label{similar}
A hedgehog $\g$ is similar to its evolute $\mathcal{E}(\g)$ if and only if $\g$ is a hypocycloid. 
\end{corollary} 

\begin{figure}[hbtp]
\centering
\includegraphics[height=2.5in]{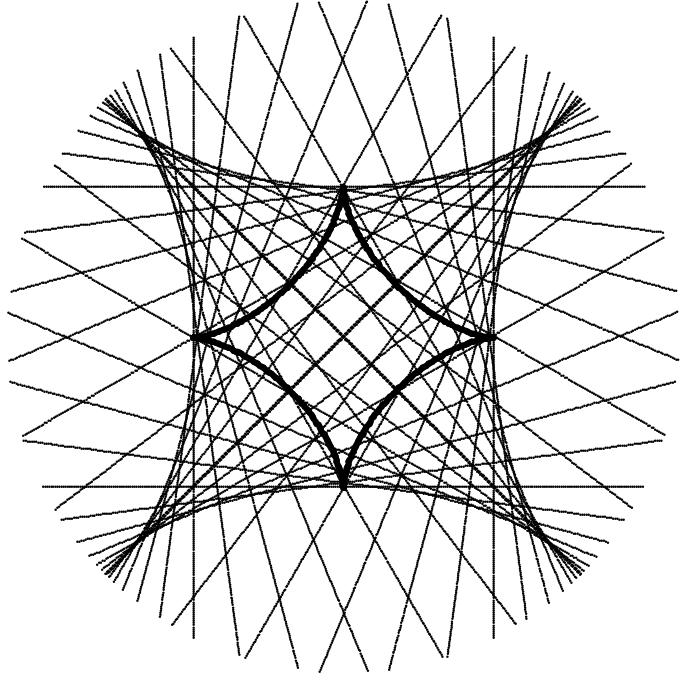}
\caption{Astroid and its evolute.}
\label{astroid}
\end{figure}

That hypocycloids are similar to their evolutes is a classical fact (C. Huygens, 1678).  See Figure \ref{astroid}. 

\subsection{Iterating involutes}\label{Involute}

The inverse derivative operator is not defined on all periodic functions; for it to be defined, the mean value of the function should vanish.  Assuming this condition, the inverse derivative is  unique. 

Obviously, the mean value of the support function of a hedgehog $\g$ does not depend on the choice of the origin: it is just the total signed length of $\g$. If we assume that this mean value is zero, then involutes exist, and the support function of precisely one of them has zero mean value. In other words, a hedgehog of zero signed length possesses a unique involute of zero signed length. We will call this unique involute the {\it evolvent}. What is important, is that we can iterate evolvents as well as evolutes.

Now we are ready to formulate the main result of this section.

\begin{theorem}
\label{limhypo}
The iterated evolvents of a hedgehog converge to its Steiner point. The limiting shape of iterated evolvents of a hedgehog is a hypocycloid, generically, of order 2.
\end{theorem}

\begin{proof} The  operator $d^{-2}$ has pure harmonics as eigenvectors; the eigenvalue on $n$-th  harmonics  is $-1/n^2$. Therefore the higher the harmonic, the faster it decays. We can rescale to keep the first non-trivial harmonic of the support function of a fixed size, and then all higher harmonics will disappear in the limit. In addition, the operator $d$ is proportional to a $90^{\circ}$ rotation on each 2-dimensional space of pure harmonics.\end{proof}

Using this result, we can obtain still another proof of the 4-vertex theorem. Recall that this theorem, in its simplest case,  asserts that a convex closed plane curve has at least four vertices, that is, extrema of curvature; see, e.g., \cite{FT}.

\begin{corollary} Any convex closed plane curve has at least four vertices. \end{corollary}  

\begin{proof} The vertices of a curve are the cusps of its evolute. Thus we want to prove that every hedgehog of zero length has at least four cusps. 
	
Since the curvature at cusps is infinite, there is at least one minimum of curvature  between any two consecutive cusps. That is, the number of cusps of the evolute is not less than that of the curve. Said differently, the number of cusps of the involute is not greater than that of the curve.
	
Now, given a hedgehog $\g$, consider its iterated evolvents. In the limit, one has a hypocycloid, which has at least four cusps. Therefore $\g$ has at least four cusps as well. \end{proof}

\subsection{Multi-hedgehogs} \label{smooth-rotation}

The above discussion extends to locally convex closed wave fronts with other rotation numbers. 

\begin{definition}\label{multi-hedgehog} A cooriented locally convex closed plane wave front without inflections and with total turning $2\pi n$ is called \emph{$n$-hedgehog}. If a locally convex closed wave front is not coorientable then its total turning is $(2n+1)\pi$ with $n\ge 1$. We shall call it \emph{$\left(n+\frac12\right)$-hedgehog}.\end{definition}
The support function $p(\alpha)$ of an $n$-hedgehog is a periodic function with the period $2\pi n$. Accordingly, its Fourier series is composed of $\cos(k\alpha/n)$ and $\sin(k\alpha/n)$ with $k\ge 1$. 
The support function of $\left(n+\frac12\right)$-hedgehog is skew-periodic: $$p(\alpha + (2n+1)\pi)=-p(\alpha),$$and its Fourier series comprises the terms $\cos(k\alpha/(2n+1))$ and $\sin(k\alpha/(2n+1))$ with odd $k$. 

As before, an iterated evolvent converges to an epicycloid, and for the same reason: only the lowest non-zero harmonics `survive' under iterations of  the operator $d^{-2}$. If such a curve is similar to its evolute then it is an epicycloid. 

\begin{figure}[hbtp]
\centering
\includegraphics[height=2.2in]{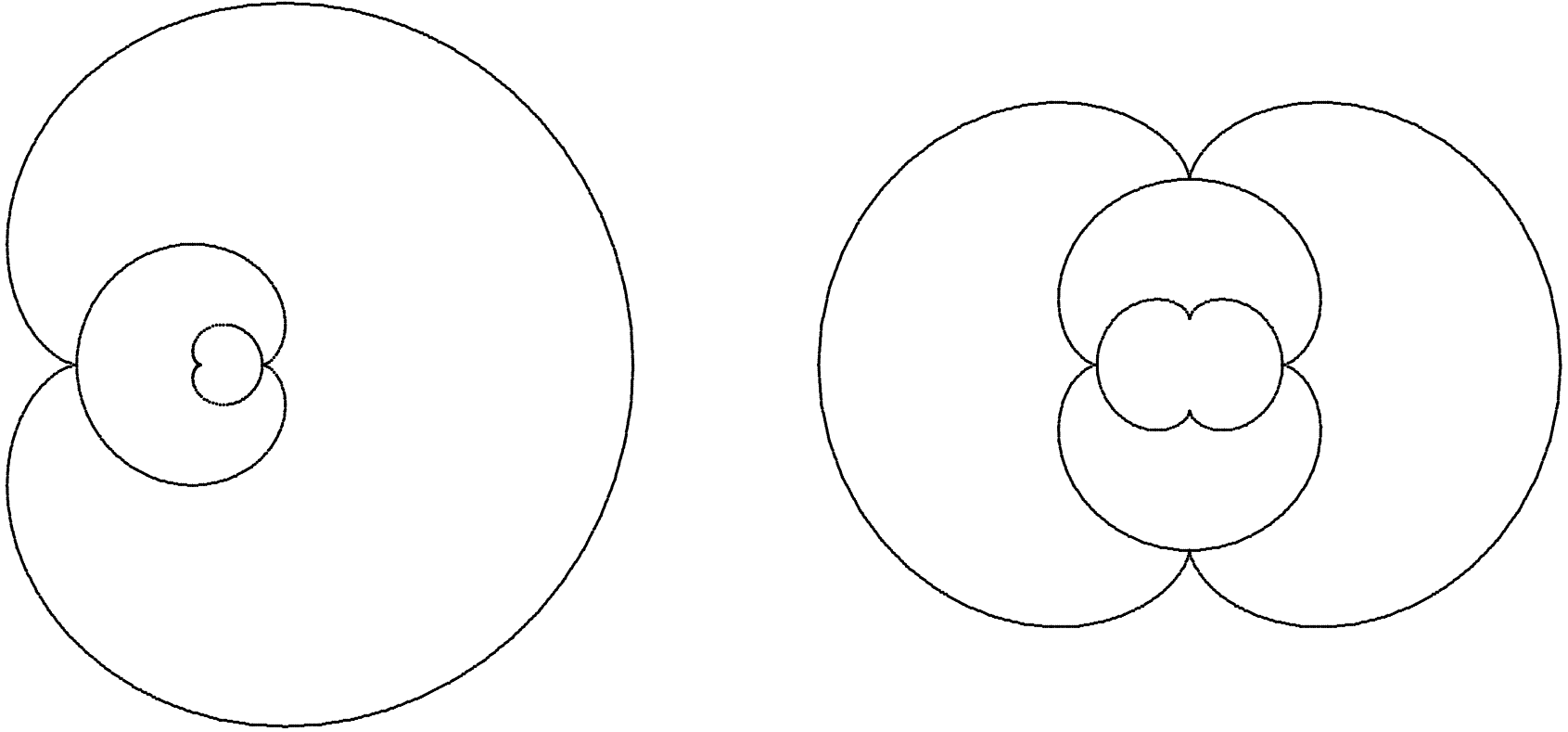}
\caption{Left: a cardioid with two consecutive evolutes. Right: a nephroid with two consecutive evolutes.}
\label{nephroid}
\end{figure}

The first examples of epicycloids are the cardioid and the nephroid; these are $\frac32$- and $2$-hedgehogs, respectively. See Figure \ref{nephroid}.

\subsection{Evolutes of non-closed smooth hedgehogs}
One can omit the closedness condition for the classical hedgehogs and consider periodic curves instead. The results stated above still apply in this case. That the cycloid coincides, up to a parallel translation, with its evolute was discovered by C. Huygens in his work on clocks. 

The cycloid is not closed, and its support function is not periodic, see Figure \ref{cycloid}.

\begin{figure}[hbtp]
	\centering
	\includegraphics[width=4in]{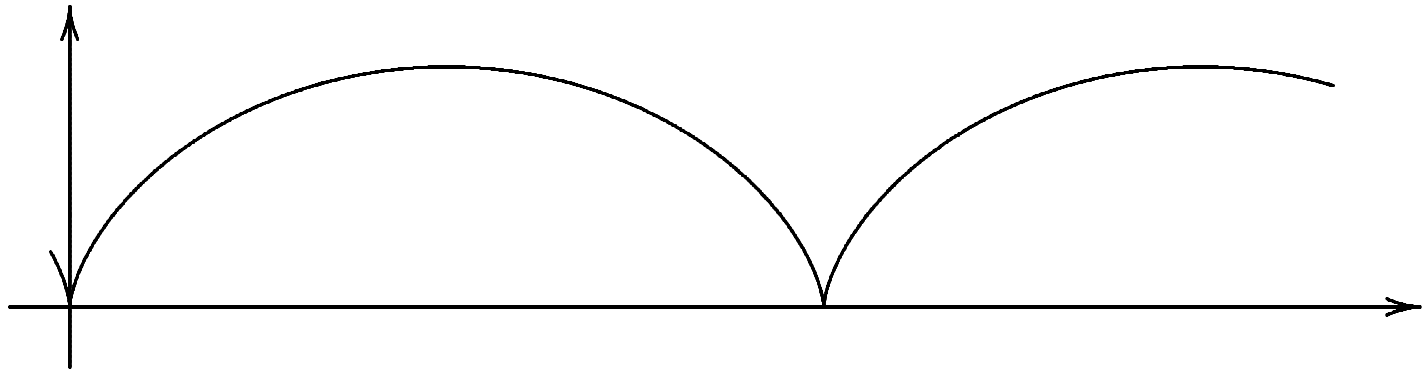}
	\caption{The support function of a cycloid is $p(\alpha)=-\alpha 
		\cos\alpha$.}
	\label{cycloid}
\end{figure}

Consider the class of locally convex curves $\gamma(\alpha)$, parameterized by angle, having the following periodicity property: $$\gamma(\alpha+2\pi)=\gamma(\alpha) -(2\pi,0)$$(the choice of this shift vector is made for convenience).

\begin{lemma}
\label{cyclshift}
The support function of such a curve has the form
\begin{equation}
\label{suppform}
p(\alpha)=-\alpha \cos\alpha + f(\alpha),
\end{equation}
where $f(\alpha)$ is $2\pi$-periodic.
\end{lemma}

\begin{proof}
Let $q(\alpha)=p(\alpha+2\pi)-p(\alpha)$. Then formulas \eqref{coordsup} imply:
$$q(\alpha) \cos\alpha - q'(\alpha) \sin\alpha = -2\pi,\ q(\alpha) \sin\alpha + q'(\alpha) \cos\alpha = 0.$$
The second differential equation has the solution $C \cos\alpha$, and the first equation determines the constant: $C=-2\pi$. Thus $q(\alpha)=-2\pi \cos\alpha$, that is, $p(\alpha+2\pi)=p(\alpha) -2\pi \cos\alpha$.
	
Let $f(\alpha)=p(\alpha) + \alpha \cos\alpha$. Then
$$f(\alpha+2\pi)=p(\alpha+2\pi) + (\alpha+2\pi) \cos\alpha = p(\alpha) + \alpha \cos\alpha=f(\alpha),
$$
and the result follows.
\end{proof}

We further restrict the class of curves assuming that the average value of the periodic part of the support function is zero. This is an analog of the assumption that the signed length of a curve vanishes. We have an analog of Theorem \ref{limhypo}. 

\begin{theorem}
\label{limcycl}
On the class of curves under consideration, the operation of taking involute is defined and unique, and under its iteration, every curve converges to the cycloid obtained from the one with the support function  $-\alpha \cos\alpha$ by a parallel translation. If a curve differs 
from its evolute by a parallel translation then the curve is the cycloid.
\end{theorem}

\begin{proof} 
As before, taking the evolute of a curve amounts to differentiating the function (\ref{suppform}) and shifting back its argument by $\pi/2$ (this operation preserves the form of the function). The effect on the function $f$ is as follows:
\begin{equation}
\label{diffshift}
f(\alpha) \mapsto f'\left(\alpha-\frac{\pi}{2}\right) + \frac{\pi}{2} \cos\alpha-\sin\alpha.
\end{equation}
Under the assumption that $f(\alpha)$ has zero average, the inverse of the operator (\ref{diffshift}) is uniquely defined. For the same reasons as in the proof of Theorem \ref{limhypo}, these iterations converge to the first harmonics. The respective limiting curve is the cycloid that differs from the one with the support function $-\alpha \cos\alpha$ by a parallel translation.
\end{proof}

\begin{remark}
{\rm Another curve that is isometric to its own evolute is a logarithmic spiral: the isometry in question is a rotation. The support function of a logarithmic spiral is $p(\alpha)=c \exp{(b\alpha)}$ where $b$ and $c$ are constants. More generally, if the support function of a curve is
$$
p(\alpha)=c_1 \exp{(b_1\alpha)}+c_2 \exp{(b_2\alpha)}\ \ {\rm with}\ \ b_1^{b_2}=b_2^{b_1},
$$ 
then the evolute is obtained from the curve by a rotation, see Figure \ref{selfrot} where $b_1=2, b_2=4$.
}
\end{remark}
\begin{figure}[hbtp]
	\centering
	\includegraphics[width=1.6in]{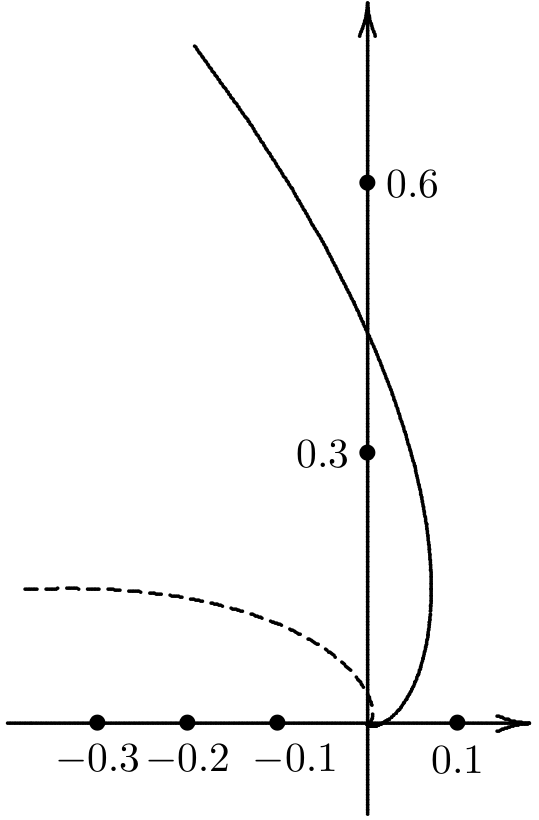}
	\caption{The evolute is a rotated copy of the curve.}
	\label{selfrot}
\end{figure}

\section{Cooriented polygons and discrete hedgehogs} \label{PolygonsHedgehogs} 

Let us summarize the basic facts about the classical evolutes and involutes,  considered in Section~\ref{smooth}, that we know so far.
\begin{itemize}
\item The evolute transformation is well defined on the space of hedgehogs/multi-hedgehogs. 
\item The necessary and sufficient condition for the existence of an involute is the zero-length condition.
\item A hedgehog and its evolute share their Steiner point.
\item The set of hedgehogs, self-similar under the evolute transformation, is the set of hypocycloids, the  hedgehogs whose support functions are pure harmonics.
\item The iterated evolvents converge to the Steiner point, while their shapes converge to a hypocycloid, generically, the astroid. 
\end{itemize}

In this section, we consider two polygonal versions of the evolute transformation which have been already described in Introduction (see Figure \ref{hexagonexample}). We will discuss what properties of the classical evolute transformations persist in either case. 

We start with the definition of the discrete version of the hedgehogs and multi-hedgehogs. As before, our analysis will be based on the notion of the support function. From now on, all index operations are assumed to be taken modulo $n$.
 
\subsection{Definitions and coordinate presentation} \label{DefCoord} Recall that the support function measures the signed distances to the tangents of a cooriented curve. The sign is positive if the perpendicular dropped from $O$ to the tangent has the same direction as the coorienting vector, and negative otherwise, see Figure \ref{lines}.

\begin{definition}
\label{coor_polyg}
A \emph{cooriented polygon} is a cyclically ordered sequence of cooriented lines $l_1, \ldots, l_n$ such that $l_j$ meets $l_{j+1}$ at a single point. We write \[l_j = (\alpha_j, p_j), \quad \alpha_j \in \R/2\pi\Z, \quad p_j \in \R,\] where $\alpha_j$ is the coorienting normal, and $p_j$ is the signed distance from the origin.\end{definition}

Reversing the coorientation transforms $(\alpha, p)$ to $(\alpha + \pi, -p)$.

\begin{figure}[hbtp]
\centering
\includegraphics[height=1.4in]{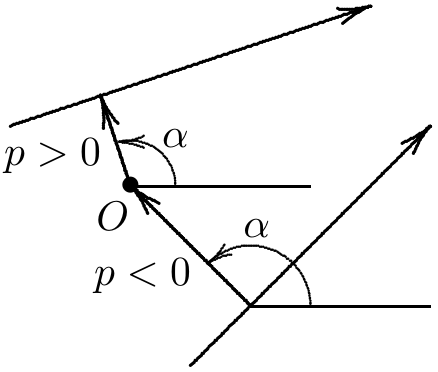}
\caption{Coordinates in the space of lines.}
\label{lines}
\end{figure}

As in the smooth case, a coorientation of a line also provides an orientation (one may turn the normal by $\pi/2$ in the counterclockwise direction).

\begin{definition} \label{polygon} For a cooriented polygon $\P$:
\begin{itemize}
\item The point of intersection of the lines $l_j$ and $l_{j+1}$ is called the \emph{$j$-th vertex} of $\P$ and is denoted $P_{j+\frac12}$.	
\item The interval $P_{j-\frac12} P_{j+\frac12}$ is called \emph{$j$-th side} of $\P$.	
\item The \emph{signed length} $\ell_j$ of the $j$-th side is defined as the distance from $P_{j-\frac12}$ to $P_{j+\frac12}$, taken with the negative sign if the order of the vertices disagrees with the orientation of $l_j$.	
\item The \emph{turning angle} at  vertex $P_{j+\frac12}$ is defined as
\begin{equation}
\label{SmallTurn}
\theta_{j+\frac12} = \alpha_{j+1} - \alpha_j (\mod{2\pi}).
\end{equation}
Since we assume that the consecutive lines $l_j$ and $l_{j+1}$ are not parallel, we have $\theta_{j+\frac12} \not\equiv 0 (\mod \pi)$. While $\theta_{j+\frac12}$ is well-defined modulo $2\pi$, we will always choose a representative between $-\pi$ and $\pi$.	
\item The integer $k$, corresponding to the sum of all turning angles, is called \emph{total turning} of $\P$:
\[
\sum_{j=1}^n \theta_{j+\frac12} = 2k\pi,
\]
\end{itemize}  
\end{definition}

\begin{remark}
\label{HalfInt}
{\rm One may also consider cooriented polygons with a half-integer turning number: take cooriented lines $l_1, \ldots, l_n$, define turning angles $\theta_{j+\frac12}$ as above, except that $\theta_{\frac12} = \alpha_1 - \alpha_n + \pi$. Informally speaking, every line changes its orientation once one traverses the polygon. Formally,  this corresponds to a cooriented $2n$-gon $\{l_1, \ldots, l_{2n}\}$, where the line $l_{n+j}$ is obtained from $l_j$ by reversing the orientation (see Figure \ref{SpecHedge} for example).
}
\end{remark}

\begin{lemma}
\label{CoordPolygon}
The vertex coordinates and the side lengths of a cooriented polygon are given by the following formulas:
\begin{equation}
\label{vertexcoord}
P_{i+\frac12}=\left(\frac{p_i \sin \alpha_{i+1} - p_{i+1} \sin \alpha_i}{\sin \theta_{i+\frac12}}, \frac{p_{i+1} \cos \alpha_i - p_i \cos \alpha_{i+1}}{\sin\theta_{i+\frac12}}\right),
\end{equation}
\begin{equation}
\label{lengths}
\ell_i = \frac{p_{i-1} - p_i\cos\theta_{i-\frac12}}{\sin\theta_{i-\frac12}} + \frac{p_{i+1} - p_i\cos\theta_{i+\frac12}}{\sin\theta_{i+\frac12}}.
\end{equation}
\end{lemma}
\begin{proof}
The equation of the line with coordinates $(\alpha_i, p_i)$ is
\[ x\cos\alpha_i+y\sin\alpha_i=p_i. \]
Solving a system of two such equations yields the first formula. The second formula is obtained from
\[ \ell_i = \langle P_{i+\frac12} - P_{i-\frac12}, v_i \rangle \]
where $v_i = (-\sin\alpha_i, \cos\alpha_i)$ is the orienting vector of the line $l_i$.
\end{proof}

\begin{definition}[Discrete analog of Definitions \ref{hedgehog} and
\ref{multi-hedgehog}]
\label{discrete_hedgehog}
If all the turning angles of a cooriented polygon $\P$ are positive, $\theta_j \in (0,\pi)$, and total turning equals $k$, the  polygon $\P$ is called a \emph{discrete $k$-hedgehog}. If $k=1$, we call it a \emph{discrete hedgehog}. 
\end{definition}

The space of parallel cooriented lines carries a natural linear structure:
\[(\alpha, p) + (\alpha, p') = (\alpha, p+p'), \quad \lambda(\alpha, p) = (\alpha, \lambda p)\]
This makes the space of cooriented $n$-gons with fixed normals an $n$-dimension\-al vector space. For discrete hedgehogs this structure is the extension of the Minkowski addition and homotheties of convex polygons; discrete hedgehogs can be viewed as formal Minkowski differences of convex polygons. Moreover, for every $\{(\alpha_j, p_j)\}$, there is a constant $C$ such that the polygon defined by $\{(\alpha_j, p_j +C)\}$ is convex, i.e., has $\ell_j > 0$.

Note that, while an arbitrary cooriented polygon is not, in general, a $k$-hedgehog for any $k$, it can be turned into a $k$-hedgehog (with integer or half-integer $k$) by an appropriate reversing of the coorientations of the sides, and this reversing is unique up to a total reversing of all the coorientations.

\subsection{Discrete Fourier analysis}
\label{Fourier}
Similarly to the smooth case, a special role in the theory of the evolute transformation of discrete hedgehogs is played by the discrete hypocycloids. To describe discrete hypocycloids we will use the technique of discrete Fourier transformation, see \cite{Sch1} and \cite{Ion}.
Let us briefly discuss it. 

For $0 < m < n/2$, define the \emph{discrete harmonics} $\mathbf{C}_m(n)$ and $\mathbf{S}_m(n)$ as the cooriented polygons formed by the lines
\begin{equation}
\label{CnSn}
\left\{\left(\frac{2\pi j}n, \cos \frac{2\pi m j}n \right)\right\}_{j = 1}^n\,\mbox{and} \quad \left\{\left(\frac{2\pi j}n, \sin \frac{2\pi m j}n \right)\right\}_{j = 1}^n
\end{equation}
respectively. The polygons $\mathbf{C}_j, \mathbf{S}_j$ have turning angles $2\pi/n$, thus they are discrete hedgehogs. Additionally, put
\[ \mathbf{C}_0(n) = \left\{\left(\frac{2\pi j}n, 1 \right)\right\}_{j=1}^n \mbox{and} \quad \mathbf{C}_{n/2}(n) = \left\{\left( \frac{2\pi j}n, (-1)^j \right)\right\}_{j=1}^n\]
(the latter is defined for $n$ even only). Geometrically, $\mathbf{C}_0(n)$ is a regular $n$-gon circumscribed about a unit circle, with the sides oriented in accordance with the counterclockwise orientation of the circle. The polygon $\mathbf{C}_{n/2}(n)$ is a regular $n/2$-gon, each side taken twice, with the opposite orientations. All lines of the polygon $\mathbf{C}_1$ go through the point $(1,0)$, and all lines of the polygon $\mathbf{S}_1$ go through $(0,1)$. See Figure \ref{SpecHedge}.

\begin{figure}[htbp]
	\centering
	\includegraphics[width=4.5in]{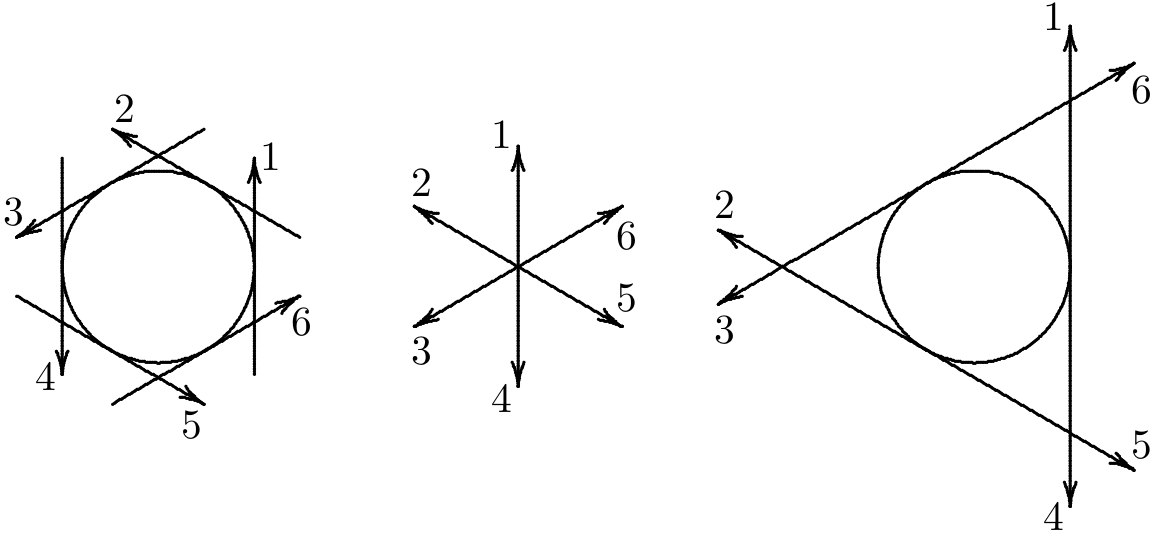}
	\caption{The polygons $\mathbf{C}_0(6)$, $\mathbf{C}_1(6)$ and 
		$\mathbf{C}_3(6)$.}
	\label{SpecHedge}
\end{figure}

It is possible to define $\mathbf{C}_m$ and $\mathbf{S}_m$ for $m > n/2$ by the same formulas \eqref{CnSn}.  However, for indices larger than $n/2$, the following reduction formulas hold:
\begin{equation}\label{reduction_formulas}
\begin{array}{ll}
\mathbf{C}_{n-m}(n) = \mathbf{C}_m(n), &\qquad \mathbf{S}_{n-m}(n) = -\mathbf{S}_m(n)\\ 
\mathbf{C}_{n+m}(n) = \mathbf{C}_m(n), &\qquad \mathbf{S}_{n+m}(n) = \mathbf{S}_m(n).
\end{array}
\end{equation}

\begin{lemma}
\label{DHypocycl_onepoint} Let $l_1, \ldots, l_n$ be oriented lines tangent to a hypocycloid of order $n+1$ or $n-1$ and such that the turning angle from $l_j$ to $l_{j+1}$ equals $2\pi/n$ for all $j$. Then all lines $l_j$ intersect at one point which lies on the circle inscribed into the hypocycloid.
\end{lemma}
\begin{proof}
This is a  consequence  of the reduction formulas \eqref{reduction_formulas}. We may assume that the coorienting normal of the line $l_j$ has direction $2\pi j/n$, and the hypocycloid is the envelope of the lines $p(\alpha) = a\cos (n\pm1)\alpha + b\sin (n\pm1)\alpha$. Then the lines $l_j$ are the sides of the polygon $a\mathbf{C}_{n\pm 1}(n) + b\mathbf{S}_{n\pm 1}(n)$. We have
\[ a\mathbf{C}_{n\pm 1}(n) + b\mathbf{S}_{n\pm 1}(n) = a\mathbf{C}_1(n) \pm b\mathbf{S}_1(n).\]
Thus all the lines $l_j$ pass through the point $(a, \pm b)$. This point lies on the inscribed circle of the hypocycloid.
\end{proof}

\begin{example}
{\rm 
Three oriented lines spaced by $2\pi/3$ and tangent to an astroid meet in a point lying on the circle inscribed in the astroid. This is a particular case of Lemma \ref{DHypocycl_onepoint}. Vice versa, for every point on this circle, the three oriented tangents to the astroid form the angles of $2\pi/3$ (see Figure \ref{AstroidCremona}). This is related to Cremona's construction: an astroid is the envelope of the chords of a circle whose  endpoints move in the opposite directions, one three times faster than the other.
}
\end{example}	

\subsection{Discrete hypocycloids} \label{DiscrHypo}
Recall that the classical hypocycloid of order $m$ is the envelope of the lines $(\alpha, a\cos m\alpha+b\sin m\alpha)$ (see Section \ref{smooth}). We will use the same definition with the  discrete harmonics $\mathbf{C}_m(n)$ and $\mathbf{S}_m(n)$ (see Section \ref{Fourier}) replacing $\cos m\alpha$ and $\sin m\alpha$. 

Every equiangular hedgehog can be represented, after an appropriate choice of the horizontal axis, as a sum of discrete harmonics
\begin{equation}
\label{DiscrFourier}
\P = \frac{a_0}2 \mathbf{C}_0 + \sum_{m=1}^{\lfloor\frac{n-1}2\rfloor} (a_m \mathbf{C}_m + b_m \mathbf{S}_m) \quad \left( + \frac{a_{n/2}}2 \mathbf{C}_{n/2} \right)
\end{equation}
(the summand with the index $n/2$ is missing for $n$ odd.)

\begin{definition}[Discrete analog of Definition \ref{hypocycloid}]\label{discrete_hypocycloids} An equiangular hedgehog with the sides tangent to a hypocycloid is called a \emph{discrete hypocycloid}. 
\end{definition}

\begin{figure} [htbp]
\centering
\includegraphics[width=3.9in]{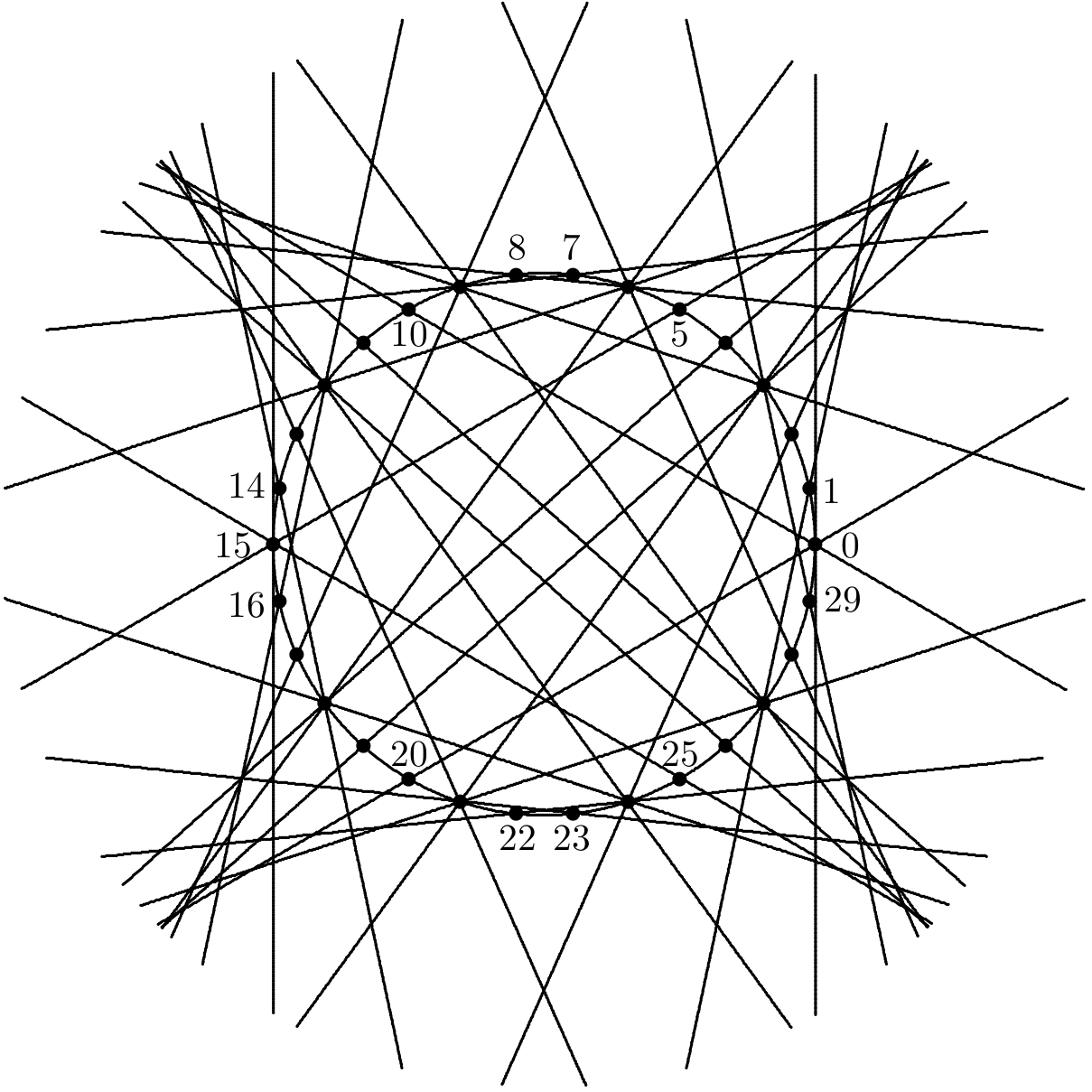}
\caption{For any $n$, the line through $n$ and $-3n$ is tangent to the astroid.}
\label{AstroidCremona}
\end{figure}

By choosing the origin at the center of the hypocycloid and directing the $x$-axis along the normal to the first side, one can represent a discrete hypocycloid in the form $a\mathbf{C}_m(n) + b\mathbf{S}_m(n)$. From the reduction formulas \eqref{reduction_formulas} it follows that all discrete hypocycloids with $n$ sides have order less than or equal to $n/2$. That is, they are tangent to a hypocycloid of some order $m\le n/2$.

\subsection{Discrete Steiner point}\label{DiscrSteiner}

Finally we have to address the discretization of the Steiner point prior to proceeding to the analysis of the discrete evolutes. 

The Steiner point can be defined for any convex closed curve in $\R^2$ by means of the formula \eqref{Stptform}.
From \eqref{Stptform} the following properties of the Steiner point are obvious:
\begin{enumerate}
\item
Minkowski linearity:
\[
\St(\gamma_1 + \gamma_2) = \St(\gamma_1) + \St(\gamma_2), \quad \St(\lambda\gamma) = \lambda\St(\gamma)
\]
where $\gamma_1 + \gamma_2$ is the boundary of the Minkowski sum of the regions bounded by $\gamma_1$ and $\gamma_2$; linear operations with curves correspond to linear operations with their support functions.
\item
Motion equivariance:
\[
\St(F(\gamma)) = F(\St(\gamma))
\]
for every proper motion $F \colon \R^2 \to \R^2$.
\item
Hausdorff continuity: the map $\gamma \mapsto \St(\gamma)$ from the space of convex closed planar curves to $\R^2$ is continuous with respect to the Hausdorff metric.
\end{enumerate}

Answering a question posed by Gr\"unbaum, Shephard \cite{Sh0} has shown that the Steiner point is uniquely characterized by the above three properties. Schneider \cite{Sch0} extended this result to higher dimensions, where the Steiner point of a convex body $K \subset \R^d$ is defined similarly to \eqref{Stptform} as
\[
\St(K) = \frac{d}{\vol(\Sph^{d-1})} \int_{\Sph^{d-1}} p(v) v\, \dvol
\]
($\dvol$ is the standard volume form on $\Sph^{d-1}$).
Schneider first showed that the first two properties imply the above formula for $\St(K)$ whenever the support function of $K$ is a linear combination of spherical harmonics (this part of the argument works only for $d \ge 3$). Since any integrable function on $\Sph^{d-1}$ can be approximated by a linear combination of spherical harmonics, a Hausdorff-continuous extension is unique.

If $\partial K$ is smooth, then $\St(K)$ is the centroid of $\partial K$ with the mass density equal to the Gauss-Kronecker curvature. The Steiner point of a convex polyhedron is the centroid of its vertices weighted by the exterior angles. In particular, for convex polygons we have
\begin{equation}
\label{eqn:StPointPolygon}
\St(\P) = \frac1{2\pi} \sum_{i=1}^n \theta_{i+\frac12} P_{i+\frac12}
\end{equation}
Hence, for equiangular polygons the Steiner point coincides with the vertex centroid. For more details on the Steiner point, see \cite{Gr0, Sch}.

The following lemma shows that the properties (i) and (ii) characterize the Steiner point for equiangular polygons uniquely. As a result, it gives an alternative proof of Shephard's theorem.
\begin{lemma}
Let $f$ be a map from the space of convex equiangular $n$-gons to $\R^2$ that is equivariant with respect to the proper motions and linear with respect to the Minkowski addition. Then $f(\P)$ is the vertex centroid of $\P$.
\end{lemma}
\begin{proof}
By linearity, the map $f$ can be uniquely extended to the space of equiangular hedgehogs. It follows that in terms of the support numbers $(p_1, \ldots, p_n)$ we have
\[
f(\P) = \sum_{j=1}^n p_j w_j,
\]
for some vectors $w_1, \ldots, w_n \in \R^2$ independent of the support numbers. Let $R$ denote the rotation of $\R^2$ about the origin by $\frac{2\pi}{n}$. Since rotation around the origin cyclically permutes the support numbers, the motion equivariance implies that
\[
\sum_{j=1}^n p_j w_{j+1} = \sum_{j=1}^n p_j R(w_j)
\]
for all values of $p_j$. Hence $w_{j+1} = R(w_j)$. 

Now consider the hedgehog $\P = \mathbf{C}_1$, that is the polygon degenerated to the point $(1,0)$, see Section \ref{Fourier}. The rotation about $(1,0)$ by $\frac{2\pi}{n}$ sends $\mathbf{C}_1$ to itself, so that by the equivariance we have $f(\mathbf{C}_1) = (1,0)$. This means that
\[
\sum_{j=1}^n \cos\frac{2\pi j}{n} R^j(w_0) = (1,0)
\]
(where $w_0 = w_n$), from which we easily compute $w_0 = (\frac2n, 0)$.

We have just shown that the linearity together with the equivariance determine the map $f$ uniquely, and that
\begin{equation}
\label{eqn:VertCentr}
f(\P) = \frac{2}{n} \sum_{j=1}^n p_jv_j,
\end{equation}
where $v_j=(\cos \frac{2\pi j}{n}, \sin \frac{2\pi j}{n})$. On the other hand, the vertex centroid is both Minkowski linear (it depends linearly on the vertex coordinates that depend linearly on the support numbers) and motion equivariant. Hence $f(\P)$ is the vertex centroid of $\P$.
\end{proof}

\begin{lemma}[Discrete analog of Lemma \ref{Stpt}]
\label{PseudoSteiner2Origin}
An equiangular hedgehog has the vertex centroid at the origin if and only if its discrete Fourier transform \eqref{DiscrFourier} is free from the first harmonics $\mathbf{C}_1$ and $\mathbf{S}_1$.
\end{lemma}

\begin{proof}
One can easily see (for example with the help of the formula \eqref{eqn:VertCentr}) that the vertex centroids of the hedgehogs $\mathbf{C}_0$ and $\mathbf{C}_m, \mathbf{S}_m$ for $m > 1$ lie at the origin. Thus the vertex centroid of an equiangular hedgehog coincides with the vertex cenroid of its first harmonic components. The hedgehog $a\mathbf{C}_1 + b\mathbf{S}_1$ has all of its vertices (and as a consequence, the vertex centroid) at the point $(a,b)$. The lemma follows.
\end{proof}

That the vertex centroid of an equiangular polygon is preserved by our constructions of angular bisector evolute and perpendicular bisector evolute will be proved later (see Propositions \ref{PsStPres} and \ref{PSPres_A}).

For non-equiangular hedgehogs the classical Steiner point \eqref{eqn:StPointPolygon} is generally not preserved by our discrete evolute constructions. Therefore we define the {\it pseudo-Steiner point} in the following axiomatic way.

\begin{definition}
\label{dfn:PsSt}
A map $\PS$ from a set of discrete hedgehogs invariant under a discrete evolute construction to 
$\R^2$  is called the \emph{pseudo-Steiner point} if it satisfies the following conditions.
\begin{itemize}
\item $\PS$ is linear with respect to the support numbers $p \in \R^n$:
\[ \PS(\alpha, \lambda p + \mu q) = \lambda \PS(\alpha, p) + \mu \PS(\alpha, q); \]
\item $\PS$ is equivariant with respect to the proper isometries of $\R^2$:
\[ \PS(F(\alpha,p)) = F(\PS(\alpha,p)) \]
for every proper isometry $F \colon \R^2 \to \R^2$;
\item $\PS$ takes the same value on a polygon and on its discrete evolute.
\end{itemize}
\end{definition}

We will see in Section \ref{sec:PsSt} that the pseudo-Steiner point is well-defined and unique for the perpendicular bisector evolute, at least in some generic situations. As for the angular bisector evolute, we have no reasonable results. This is due to the fact that the turning angles change in quite a complicated way under the angular bisector evolute transformation.

Now we are ready to describe two types of the discrete evolute transformations. Our construction is based on the definition of the evolute as the locus of the centers of the osculating circles, and we consider two different discretizations of the osculating circles of a polygon.

\section{$\Pev$-evolutes}\label{P-evo}
\subsection{Definition and coordinate presentation}\label{DefP-evo}
Let $\P=(l_1, \ldots, l_n)$ be a cooriented polygon with vertices $P_{1+\frac12}, \ldots, P_{n+\frac12}$. Denote by $l_j^*$ the line that passes through the midpoint of the segment $P_{j-\frac12}P_{j+\frac12}$, and whose direction is obtained from the direction of $l_j$ by a counterclockwise quarter-turn. The lines $l^*_j$ and $l^*_{j+1}$ are not parallel since the lines $l_j$ and $l_{j+1}$ are not parallel. By construction, their intersection point $P^*_{j+\frac12}$ is the circumcenter of the triangle $P_{j-\frac12}P_{j+\frac12}P_{j+\frac32}$ (if, say, $P_{j-\frac12} = P_{j+\frac12}$, then $P^*_{j+\frac12}$ is the center of the circle passing through $P_{j+\frac32}$ and tangent to $l_j$ at $P_{j-\frac12} = P_{j+\frac12}$). 

\begin{definition}
\label{Pevo}
The cooriented polygon $\P^*=\{l^*_1, \ldots, l^*_n\}$ is called the \emph{perpendicular bisector evolute} (or  $\Pev$-evolute) of $\P$. We will denote it by $\Pev(\P)$.
\end{definition}

The choice of (co)orientation of the polygon plays no significant role: reversing the orientation of $l_j$ simply reverses the orientation of $l^*_j$.

As before, let $(\alpha_i, p_i)$ be the coordinates of the line $l_i$.
\begin{lemma}
\label{ParamPBisect}
The line $l^*_i$ has the coordinates
\[
\alpha_i^*=\alpha_i+\frac{\pi}{2}, \quad p_i^*=\frac{p_{i+1} \sin \theta_{i-1/2} - p_{i-1} \sin \theta_{i+1/2} + p_i \sin (\theta_{i+1/2} - \theta_{i-1/2})}{2 \sin \theta_{i-1/2}\, \sin \theta_{i+1/2}}.
\]
\end{lemma}
\begin{proof}
Since $l^*_i$ is obtained from $l_i$ by a counterclockwise quarter-turn, its angle coordinate is as stated in Lemma.

The coorienting normal $\nu^*_i$ of $l^*_i$ has the coordinates $(-\sin \alpha_i, \cos \alpha_i)$. Therefore the signed distance from $O$ to $l^*_i$ equals $\langle X, \nu^*_i \rangle$, where $X$ is an arbitrary point on $l^*_i$. Use as $X$ the midpoint of the segment $P_{i-1/2}P_{i+1/2}$. The coordinates of $P_{i\pm 1/2}$ were computed in Lemma \ref{CoordPolygon}. Substitution and simplification yields the formula for $p^*_i$.
\end{proof}
Lemma \ref{ParamPBisect} implies the following coordinate description of the transformation $\Pev$.

\begin{theorem}
\label{CoordP}
Let $\P=\{(\alpha_j,p_j)\}$ be a cooriented polygon. Then  $\Pev(\P)=\{(\alpha^*_j,p^*_j)\}$ where $\alpha^*_j=\alpha_j+\pi/2$ and the vectors $p=(p_1,\dots,p_n)$ and $p^\ast=(p^\ast_1,\dots,p^\ast_n)$ are related by the formula $p^\ast=\Pev_\theta p$, in which $\Pev_\theta$ is the cyclically tridiagonal matrix
\begin{equation}
\label{eqnA}
\Pev_\theta = \frac12
\begin{pmatrix} 
a_1-a_2& b_2  & 0&\cdots 0&-b_1\\ -b_2& a_2-a_3 & b_3&\ddots & 0\\ \vdots&\ddots&\ddots &\ddots& 0\\ 0 &  \ddots     &   -b_{n-1} &a_{n-1}-a_n& b_n\\ b_1&0&\cdots  0&-b_n& a_n-a_1
\end{pmatrix}
\end{equation}
with
\[ a_j = \cot\theta_{j-\frac12},\  b_j = \csc\theta_{j-\frac12}. \]
\end{theorem}

Since the turning angles $\theta_{j+\frac12}$ remain unchanged, the dynamics of the iterated $\Pev$-evolute transformation depend solely on the spectral properties of the matrix $\Pev_{\theta}$ (for a generic polygon, on the  eigenvalues with the maximal absolute values: see details in Section \ref{limitbehavior}).

\begin{remark}
{\rm 
The coordinates of the vertices of the $\Pev$-evolute can be computed from those of the polygon as follows. Let $P_{j+1/2} = (x_j, y_j)$ and $P^*_{j+1/2} = (x^*_j, y^*_j)$. Then we have
\begin{equation*}
\begin{split}
x_j^* = \frac{\det\left|\begin{array}{ccc} x_{j-1}^2 & x_j^2 & x_{j+1}^2 \\ y_{j-1} & y_j & y_{j+1} \\ 1 & 1 & 1\end{array}\right| + \det\left|\begin{array}{ccc} y_{j-1}^2 & y_j^2 & y_{j+1}^2 \\ y_{j-1} & y_j & y_{j+1} \\ 1 & 1 & 1 \end{array}\right|} {2 \det\left|\begin{array}{ccc} x_{j-1} & x_j & x_{j+1} \\ y_{j-1} & y_j & y_{j+1} \\ 1 & 1 & 1 \end{array}\right|}, \\
y_j^* = \frac{\det\left|\begin{array}{ccc} x_{j-1} & x_j & x_{j+1} \\ x_{j-1}^2 & x_j^2 & x_{j+1}^2 \\ 1 & 1 & 1 \end{array}\right| + \det\left|\begin{array}{ccc} x_{j-1} & x_j & x_{j+1} \\ y_{j-1}^2 & y_j^2 & y_{j+1}^2 \\ 1 & 1 & 1\end{array}\right|} {2 \det\left|\begin{array}{ccc} x_{j-1} & x_j & x_{j+1} \\ y_{j-1} & y_j & y_{j+1} \\ 1 & 1 & 1\end{array}\right|}.
\end{split}
\end{equation*}
}
\end{remark}

We conclude this section with a simple geometric observation concerning the rank and the kernel of the map $p\mapsto p^\ast$. We will use the following notation. For a set $\alpha=\{\alpha_1,\dots,\alpha_n\}$, let$$\B(\alpha)=\sum_{i=1}^n(-1)^{i-1}\alpha_i.$$If the numbers $\alpha_i$ are defined modulo $2\pi$, then $\B(\alpha)$ is also defined modulo $2\pi$. If $n$ is even, then $\B$ is invariant, up to a sign change, with respect to cyclic permutations of $\alpha$. This is not so if $n$ is odd and, specifically for this case, we introduce a new notation:
$$\B_j(\alpha)=\sum_{i=1}^{n-1}(-1)^i\alpha_{j+i}$$
(we assume that $\alpha_{n+i}=\alpha_i$); obviously, in this case, $\B(\alpha)=\alpha_1+\B_1(\alpha)$. It is also clear that $\B_j(\alpha)$ can be expressed in terms of the turning angles $\theta_{i+\frac12}=\alpha_{i+1}-\alpha_i$:$$2\B_j=\theta_{j+\frac12}-\theta_{j+\frac32}+\theta_{j+\frac52}-\dots+\theta_{j+n-\frac12}.$$

\begin{proposition}\label{KernelRank}
{\rm(i)} If $n$ is odd, then the kernel of the linear map $p\mapsto p^\ast$ has dimension $1$, so its rank equals $n-1$. This kernel is generated by the vector $(p_1^\ast,\dots,p_n^\ast)$ where $$p_j^\ast=\cos\B_j(\alpha).$$
{\rm(ii)} If $n$ is even, and $\B(\alpha)\ne0$, then the map $p\mapsto p^\ast$ has zero kernel, so its rank equals $n$. If, for $n$ even, $\B(\alpha)=0$, then the rank equals $n-2$.
\end{proposition}
\begin{proof}
Obviously, a polygon belongs to the kernel of the map $p\mapsto p^\ast$ if and only if it is inscribed in a circle centered at $O$ (which means that all the perpendicular bisectors pass through $O$). For a given $\alpha$, let us try to construct such a polygon (see Figure \ref{inscribed}).

\begin{figure} [htbp]
\centering
\includegraphics[width=2.6in]{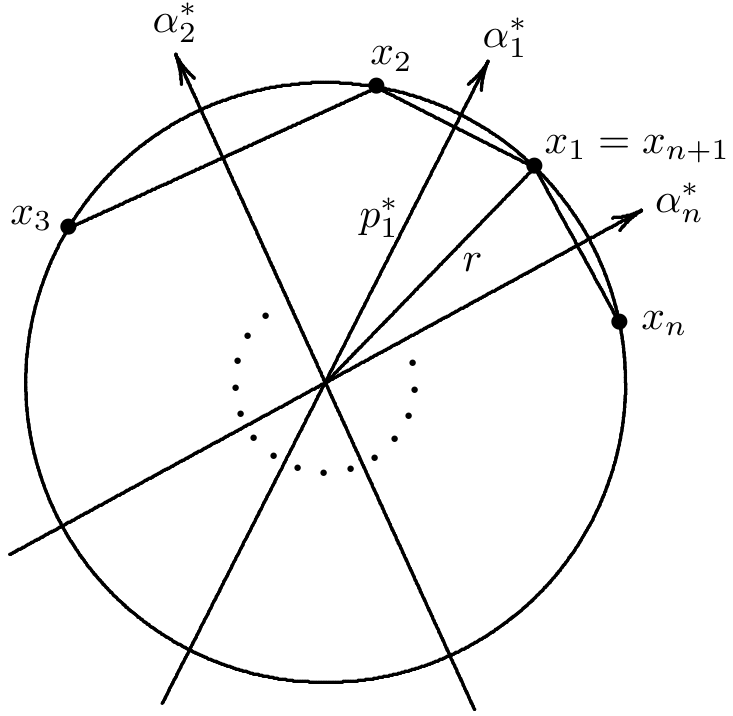}
\caption{To Proposition \ref{KernelRank}.}
\label{inscribed}
\end{figure}

For the first vertex choose an arbitrary position, $x_1$. Then the second vertex, $x_2$, is obtained from $x_1$ by a reflection in the radius $\alpha_1^\ast$, which makes it $2\alpha_1^\ast-x_1$. Similarly, $x_3=2\alpha_2^\ast-2\alpha_1^\ast+x_1,\dots,x_{n+1}=2\alpha^\ast_n-2\alpha^\ast_2+\dots-(-1)^n2\alpha^\ast_1=(-1)^nx_1.$ Since $x_{n+1}=x_1$, we obtain $$(1-(-1)^n)x_1=2(\alpha^\ast_n-\alpha^\ast_{n-1}+\dots-(-1)^n\alpha_1^\ast)=(-1)^{n-1}2\B(\alpha^\ast).$$Here comes the difference between the cases of odd and even $n$. If $n$ is odd, then $x_1=\B(\alpha^\ast)$, hence $x_1$, as well as all the other $x_i$'s, is uniquely defined, and hence the kernel of the map $p\mapsto p^\ast$ is one-dimensional. Moreover, in this case,$$p_1^\ast=r\cos(\B(\alpha^\ast)-\alpha^\ast_1)=r\cos\B_1(\alpha^\ast)=r\cos\B_1(\alpha),$$and similar formulas hold for all $p^\ast_i$. For $n$ even, $\B(\alpha^\ast)=\B(\alpha)$, and if $\B(\alpha)\ne0$, then our problem of constructing an inscribed polygon has no solutions; so, the map $p\mapsto p^\ast$ has no kernel. Finally, if $n$ is even and $\B(\alpha)=0$, then any $x_1$ will work, so we have (for any $r$) a one-parameter family  of inscribed polygons, and hence our kernel has dimension 2.
\end{proof}

\subsection{The equiangular case} \label{PevEquiang}
Let us first consider the case of equiangular hedgehogs, which are polygons with turning angles $\theta_j = 2\pi/n$ for all $j$. The transformation matrix for the support numbers takes the form
\[ \Pev_\theta = \frac1{2\sin(2\pi/n)}(\mathcal{Z} - \mathcal{Z}^\top), \]
where $\mathcal{Z}$ is the matrix corresponding to the index shift, 
\begin{equation}
\label{eqnZ}
\mathcal{Z}=\begin{pmatrix}
0 & 1  & 0\cdots & 0\\ \vdots&\ddots&\ddots&\vdots \\ 0&\cdots&0&1\\ 1& 0\cdots& 0&0
\end{pmatrix}
\end{equation}

The following lemma is classical, see for example \cite{Da}.
\begin{lemma}
\label{spectrum}
The eigenvalues of the matrix $\Pev_\theta$ are $$\lambda_m=i\, \frac{\sin(2\pi m/n)}{\sin(2\pi/n))},\ m=0,1,\dots,n-1.$$
\end{lemma}

\begin{proof}
Let $a$ be an $n$-th root of unity, $a=\cos(2\pi m/n)+i\sin(2\pi m/n)$. The vector $v=(1,a,\dots,a^{n-1})^\top$ is an eigenvector of both $\mathcal Z$ and $\mathcal Z^\top$: $\mathcal Zv=av,\, \mathcal Z^\top v=a^{-1}v=\overline av$. Hence, $\Pev_\theta v=\lambda_mv$ where $\lambda_m$ is as above.
\end{proof}

% The matrix $\Pev_\theta$ has  purely imaginary spectrum $\{\lambda_m\}$ (see Lemma \ref{spectrum}).
One can immediately see that the eigenvalues $\lambda_m$ and $\lambda_{n-m}$ are complex conjugate: $\overline{\lambda_m} = \lambda_{n-m}$. The hedgehogs in the corresponding invariant real subspace are the discrete hypocycloids of order $m$ (see Definition \ref{discrete_hypocycloids}). The only zero eigenvalues are $\lambda_0$ and  $\lambda_{n/2}$ (for even $n$); the corresponding hedgehogs are $\mathbf{C}_0(n)$ and $\mathbf{C}_{n/2}(n)$.

This calculation of the spectrum of $\Pev_\theta$ leads to the following results.

\begin{theorem}
\label{EvolDHPerp}
Let $\P$ be an equiangular hedgehog with $n$ sides tangent to a hypocycloid $h$ of order $m$. Then $\Pev(\P)$ is tangent to $\E(h)$,  scaled by the factor $\dfrac{\sin(2\pi m/n)}{m\sin(2\pi/n)}$ with respect to its center. 
\end{theorem}

\begin{proof}
Without loss of generality, assume $\alpha_j = 2\pi j/n$. Compute the evolutes of the discrete hedgehogs $\mathbf{C}_m(n)$ and $\mathbf{S}_m(n)$:
\begin{gather*}
\left( \frac{2\pi j}n, \cos \frac{2\pi m j}n \right)_{j=1}^n \mapsto \frac{\sin(2\pi m/n)}{\sin(2\pi/n)} \left( \frac{\pi}2 + \frac{2\pi j}n, -\sin \frac{2\pi m j}n \right)_{j=1}^n,\\ \left( \frac{2\pi j}n, \sin \frac{2\pi m j}n \right)_{j=1}^n \mapsto \frac{\sin(2\pi m/n)}{\sin(2\pi/n)} \left( \frac{\pi}2 + \frac{2\pi j}n, \cos \frac{2\pi m j}n \right)_{j=1}^n.
\end{gather*}
Hence the evolute of a hedgehog, tangent to the hypocycloid $h$ given by $p(\alpha) = a\cos m\alpha + b\sin m\alpha$, is tangent to the hypocycloid
\[ p(\alpha) = \frac{\sin(2\pi m/n)}{\sin(2\pi/n)} \left(-a\sin m\left(\alpha - \frac{\pi}2\right) + b\cos m\left(\alpha - \frac{\pi}2\right)\right). \]
By Lemma \ref{evolsupp}, the evolute of $h$ has equation $$p(\alpha) = m \left(-a\sin m\left(\alpha - \dfrac{\pi}2\right) + b\cos m\left(\alpha - \dfrac{\pi}2\right)\right),$$
and the theorem follows.
\end{proof}

The second $\Pev$-evolute of $\P$ is tangent to an appropriately scaled second evolute of $h$. Since $\E^2(h)$ is homothetic to $h$ with the factor $-m^2$, and the angle coordinates of $\Pev^2(\P)$ are $\alpha_j + \pi$, we have the following.
\begin{corollary} \label{corEvolDHPerp}
The second evolute $\Pev^2(\P)$ is homothetic to $\P$ with respect to the center of $h$ with the coefficient $-\dfrac{\sin^2(2\pi m/n)}{\sin^2(2\pi/n)}$.
\end{corollary}

We see that the $\Pev$-evolute of a discrete hypocycloid is ``smaller'' than the evolute of the corresponding smooth hypocycloid. As the ratio $n/m$ tends to infinity, the $\Pev$-evolute tends to the smooth one. In  Figure \ref{9and5},  we present the $\Pev$-evolutes of discrete astroids with $5$ and $9$ sides.

\begin{figure}[htbp]
	\centering
	\includegraphics[width=2.6in]{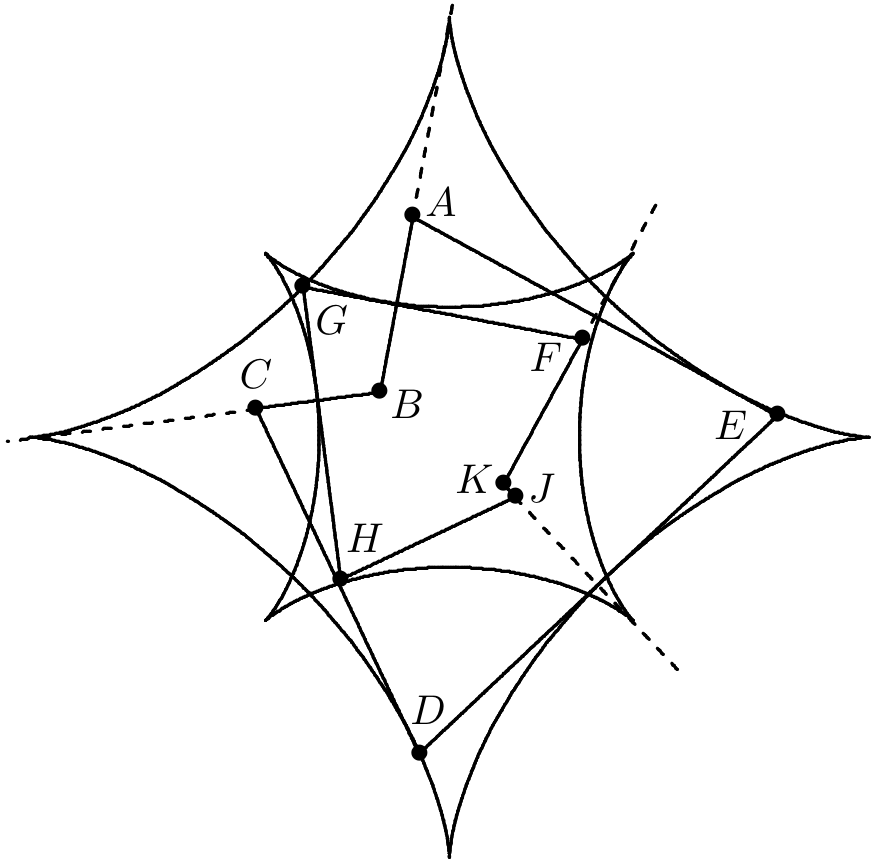}
	\includegraphics[width=2.4in]{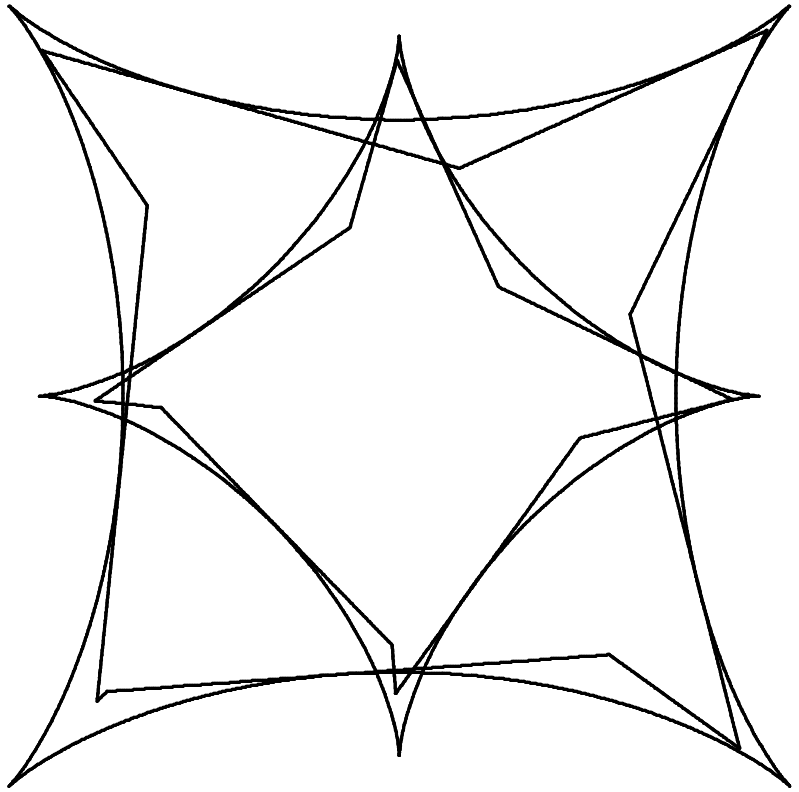}
	\caption{Discrete astroids and their $\Pev$-evolutes.}
	\label{9and5}
\end{figure}

Note that, for $n=5$, one gets 
$$\dfrac{\sin4\pi/5}{\sin2\pi/5} = 2 \cos\dfrac{2\pi}5 = 0,618...<1.$$  
Therefore the $\Pev$-evolutes of the discrete astroids with $5$ sides shrink to a point. The same happens with the hypocycloids of order $m$ with $2m+1$ sides. 

Let us take a closer look at the absolute values of the spectrum of  $\Pev_\theta$. Put $\mu_m = |\lambda_m| = |\lambda_{n-m}|$ for $0 \le m \le n/2$. Then $\mu_0 = 0$ and $\mu_1 = 1$.
\begin{itemize}
\item For $n$ odd, we have $\mu_0 < \mu_{\frac{n-1}2} < \mu_1 < \mu_{\frac{n-3}2} < \mu_2 <\cdots$ with the maximum of $\mu_m$ at \[ m =
\begin{cases}
\dfrac{n-1}4, &\text{for } n \equiv 1 \pmod 4\\
\dfrac{n+1}4, &\text{for } n \equiv 3 \pmod 4.
\end{cases} \]
\item For $n$ even, we have $\mu_0 = \mu_{\frac{n}2} < \mu_1 = \mu_{\frac{n}2-1} < \cdots$ with the maximum of $\mu_m$ at \[ m =
\begin{cases}
\dfrac{n}4, &\text{for } n \equiv 0\pmod 4\\
\dfrac{n\pm 2}4, &\text{for } n \equiv 2 \pmod 4.
\end{cases} \]
\end{itemize}

\begin{proposition}
\label{PsStPres}
The vertex centroids of an equiangular hedgehog and of its $\Pev$-evolute coincide.
\end{proposition}
\begin{proof}
The vertex centroid of an equiangular hedgehog coincides with the vertex centroid of its first harmonic component, see the proof of Lemma \ref{PseudoSteiner2Origin}. The $\Pev$-evolute sends discrete hypocycloids (which are given by pure harmonics) to discrete hypocycloids of the same order, see the proof of Theorem \ref{EvolDHPerp}. The vertices of a hypocycloid of order $1$ coincide; the vertices of its evolute also lie at the same point. Hence the evolute construction preserves the vertex centroid.
\end{proof}

Now we can address the limiting behavior of the  iterated $\Pev$-evolutes of equiangular hedgehogs.
\begin{enumerate}
\item [$n=3,\,4$] The $\Pev$-evolutes degenerate to a point after the first step.
\item[$n=5\phantom{,\,4}$] An equiangular pentagon becomes a discrete astroid after the first step, and then shrinks to its vertex centroid in such a way that $\Pev^{(k+2)}(\P)$ is homothetic to  $\Pev^{(k)}(\P)$.
\item[$n=6\phantom{,\,4}$]   After the first iteration of $\Pev$, the  $\mathbf{C}_0$ and $\mathbf{C}_3$ terms in the decomposition \eqref{DiscrFourier} disappear. After that the sequence $\Pev^k(\P)$ becomes periodic. Namely, it alternates between two hexagons as shown in Figure \ref{RegHexEvol}. (In fact, the line $l_j$ becomes the line $l_{j+3}$ after two iterations, so the sequence is alternating only if we ignore the marking of the sides.)

 \begin{figure}[htbp]
 \centering
 \includegraphics[width=4.8in]{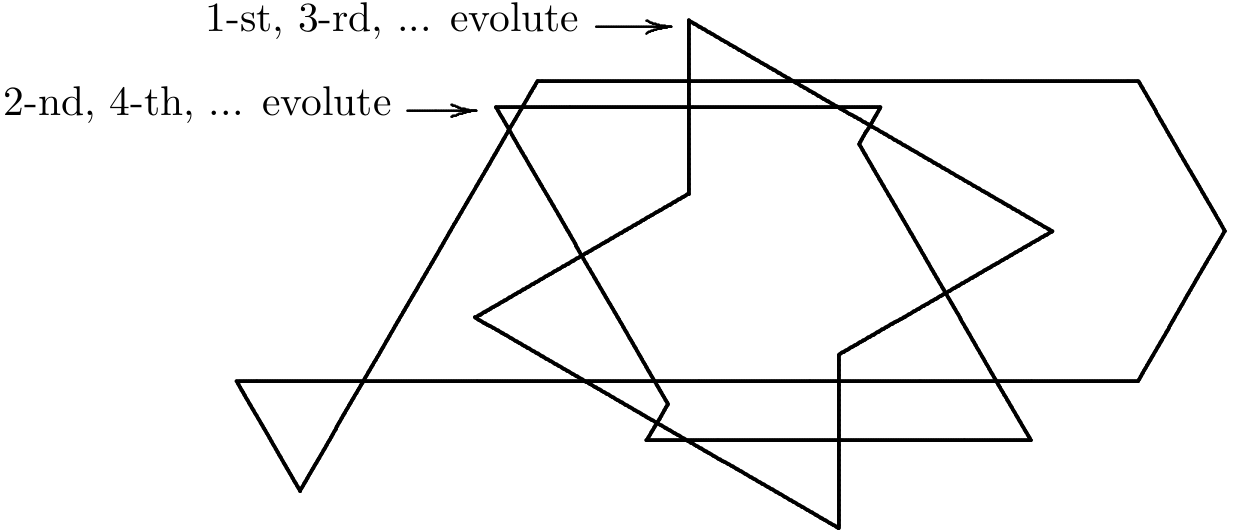}
 \caption{Iterated $\Pev$-evolutes of an equiangular hexagon become periodic after the first step.}
 \label{RegHexEvol}
 \end{figure}
 
 \item[$n\ge 7\phantom{,\,4}$]	The iterated $\Pev$-evolutes of a generic equiangular  hedgehog expand; their shapes tend to the discrete hypocycloids, generically of  order $m$, given by the formulas above. \big(If $n \equiv 2\pmod 4$, then the limiting shapes are linear combinations of the hypocycloids of order $\dfrac{n-2}4$ and $\dfrac{n+2}4$\big).
 
\end{enumerate}  

\subsection{The generic case: a pseudo-Steiner point}
\label{sec:PsSt}
Let us now move on to the case of polygons with generic turning angles. In the generic situation, the classical Steiner point is not preserved by the $\Pev$-transform. The existence of a pseudo-Steiner point (see Definition \ref{dfn:PsSt}) is governed by the spectral properties of the matrix $\Pev_\theta$.

Introduce the following vectors of support numbers:
\[ C=(\cos\alpha_1,\cos\alpha_2,\ldots,\cos\alpha_n),\ S=(\sin\alpha_1,\sin\alpha_2,\ldots,\sin\alpha_n). \]

The polygon $(\alpha, aC+bS)$ consists of lines passing through the point $(a,b)$ (see Lemma \ref{DHypocycl_onepoint}). The parallel translation through vector $(a,b)$ results in the following change of the coordinates of a hedgehog:
\[ \mathcal{T}_{(a,b)}: (\alpha,p) \mapsto (\alpha, p + aC + bS). \]

\begin{proposition}
\label{Pinvariant}
The linear subspace $\lin\{C, S\}$ is $\Pev_\theta$-invariant; the complexification of $\Pev_\theta$ restricted to $\lin\{C, S\}$ has $\pm i$ as eigenvalues.
\end{proposition}
\begin{proof}
This follows from the fact that  $\Pev$  commutes with the parallel translation $\mathcal{T}_{(a,b)}$. By a straightforward computation one gets
\begin{equation}
\label{invplane}
\Pev_{\theta} (C) = -S,\ \Pev_{\theta} (S) = C.
\end{equation}
This completes the proof.
\end{proof}

The following theorem provides a necessary and sufficient condition for the existence and uniqueness of a pseudo-Steiner point for $\Pev$-evolute.

\begin{theorem}
\label{thm:PsSt}
A pseudo-Steiner point for $\P$  exists if and only if the $\Pev_\theta$~-~invariant subspace $\lin\{C, S\} \subset \R^n$ has a $\Pev_\theta$-invariant complement:
\[ \R^n = \lin\{C, S\} \oplus W, \quad \Pev_\theta(W) \subset W. \]
If, in addition, the complexified restriction of $\Pev_\theta$ to $W$ doesn't have $\pm i$ as eigenvalues, then the pseudo-Steiner point is unique.
\end{theorem}

\begin{proof}
Assume that there is a map $\PS$ with the properties stated in Definition \ref{dfn:PsSt}. By  linearity, we can write
\[ \PS(\alpha, p) = \Pi_\alpha p \]
where $\Pi_\alpha \colon \R^n \to \R^2$ is a linear homomorphism. The translation equivariance implies
\[ \Pi_\alpha(aC+ bS) = (a,b). \]
In particular, $\Pi_\alpha$ is an epimorphism. Denote the kernel of $\Pi_\alpha$ by $W_\alpha$. We have $\R^n = \lin\{C, S\} \oplus W_\alpha$, and we claim that $W_\alpha$ is $\Pev_\theta$-invariant. 

Indeed, denoting by ${R}_{\pi/2}$ the rotation by $\pi/2$, we have
\[ \Pi_\alpha p = \PS(\alpha, p) = \PS(\alpha^*, p^*) = {R}_{\pi/2} (\PS(\alpha, \Pev_\theta p)) = {R}_{\pi/2} \Pi_\alpha \Pev_\theta p. \]
Hence if $p \in W_\alpha$, then ${R}_{\pi/2} \Pi_\alpha \Pev_\theta p = 0$, which implies $\Pev_\theta p \in W_\alpha$.
	
The above argument shows that if a pseudo-Steiner point exists, then it is obtained by projecting $p \in \R^n$ to $\lin\{C, S\}$ along a $\Pev_\theta$-invariant complement, and then identifying $\lin\{C, S\}$ with $\R^2$ by $C \mapsto (1,0)$, $S \mapsto (0,1)$. Vice versa, this construction always yields a pseudo-Steiner point. This implies the existence and the uniqueness statements of the theorem.
\end{proof}

\begin{remark}
Compare this to the definition of the classical Steiner point of convex $d$-dimensional bodies: it is obtained by projecting the support function to the space of the spherical harmonics of order $1$, and then taking the corresponding point in $\R^d$.
\end{remark}

\begin{corollary}
The set of polygons which does not have pseudo-Steiner point has measure zero.
\end{corollary}

\begin{corollary}
For odd-gons close to equiangular ones, the pseudo-Steiner point exists and is unique.
\end{corollary}
\begin{proof}
The spectrum of $\Pev_\theta$ depends continuously on $\theta$. Since for odd $n$ and $\theta_j =2\pi/n$ the matrix $\Pev_\theta$ is diagonalizable (over $\C$) with different eigenvalues, the eigenvalues remain different for ``almost equiangular'' odd-gons. In particular, the space $\lin\{C, S\}$ has a unique $\Pev_\theta$-invariant complement.
\end{proof}

The following proposition shows that if a pseudo-Steiner point is not defined, then there is a polygon that ``drifts'' under the evolute transformation.

\begin{proposition}
\label{drifting}
Denote by  $\P_\theta$ the set of the cooriented polygons  with fixed turning angles $\theta$. Then either there exists a polygon $\P\in \P_\theta$ such that its second $\Pev$-evolute is, up to the orientation 
reversing of all the lines,  a translate $\mathcal{T}_{a,b}(\P)$ of $\P$ by a non-zero vector, or there exists a pseudo-Steiner point for all $\P\in \P_\theta$. 
\end{proposition}

\begin{proof}
If there is a drifting polygon, then there is no pseudo-Steiner point, since the translation equivariance contradicts  $\Pev$-invariance.
	
If there is no pseudo-Steiner point then, by Theorem \ref{thm:PsSt}, $\lin\{C, S\}$ has no $\Pev_\theta$-invariant complement. This means that $\lin\{C, S\}$ is at the head of a Jordan cell corresponding to the complex eigenvalues $\pm i$. It follows that there are vectors $p, q \in \R^n$ such that
\[ \Pev_\theta p = q + S, \quad \Pev_\theta q = -p + C. \]
But then
\[ \Pev_\theta^2 p = \Pev_\theta(q+S) = (-p+C) + C = -p + 2C.\]
Hence the second evolute of $(\alpha, p)$ is $(\alpha + \pi, -p + 2C)$ which is the polygon $(\alpha, p)$ with reversed orientations of the sides, translated by the vector $(-2, 0)$.
\end{proof}

At the moment, we do not have examples of polygons with no pseudo-Steiner point.

We have seen that, for equiangular even-gons, the eigenvalues $\pm i$ are double. By the above arguments, the eigenvalues $\pm i$ produce polygons equal to their second $\Pev$-evolutes up to the side reorientations. The following examples exhibit non-equiangular even-gons with the same property. 

\begin{figure}[hbtp]
\centering
\includegraphics[height=1.7in]{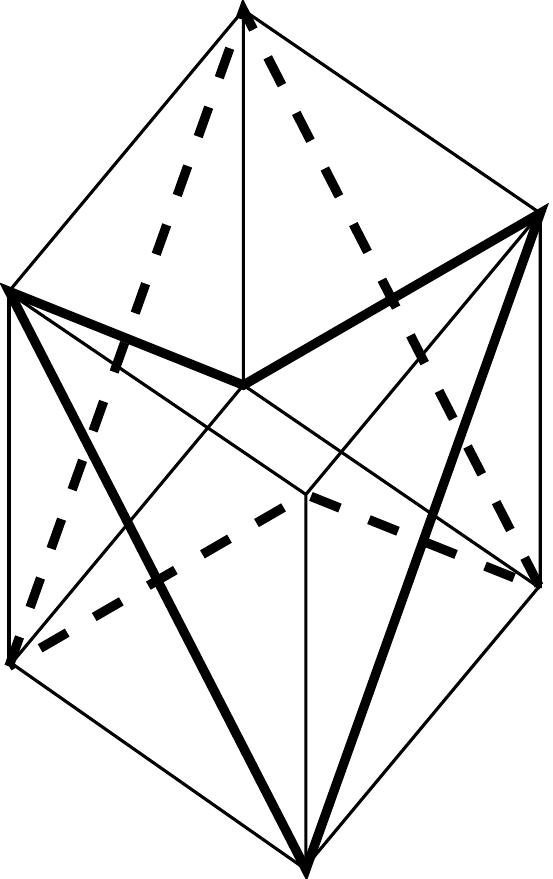} \hspace{0.2 \textwidth}
\includegraphics[height=1.7in]{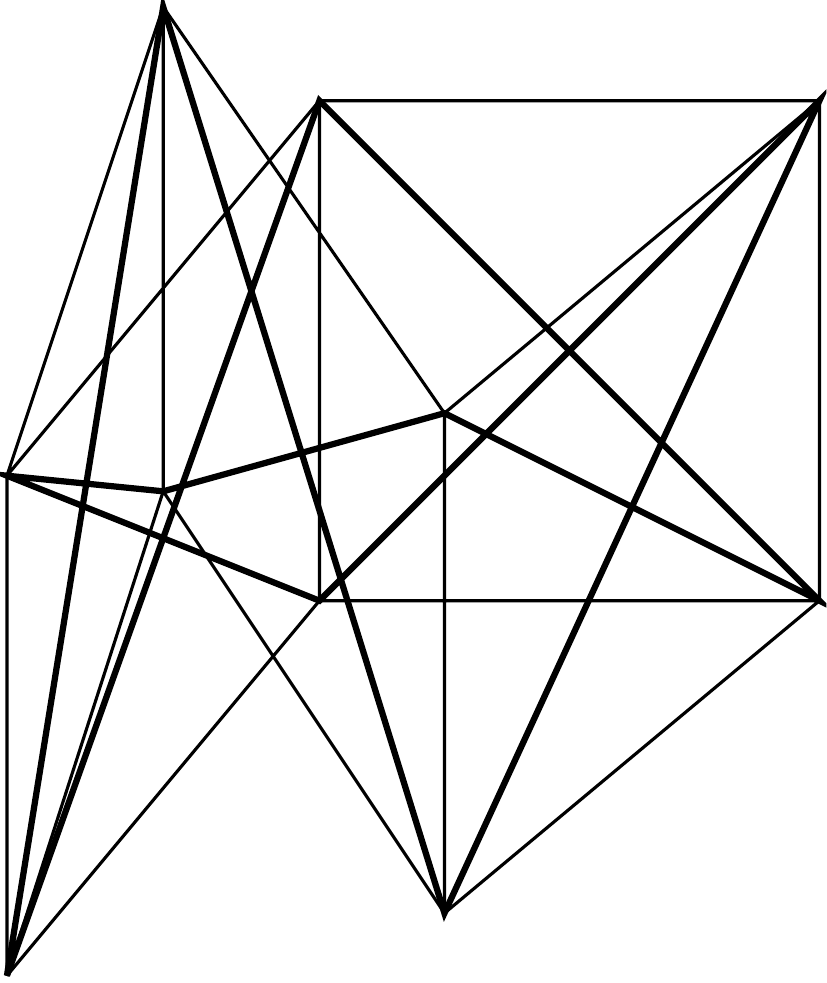}
\caption{Left: two polygons, one in solid and another in dashed line, that are each other's $\Pev$-evolutes. Right: the diagonals of the rhombi form a decagon that is its own $\Pev$-evolute.}
\label{self}
\end{figure}

\begin{example}
\label{pairs}
{\rm One can construct a pair of polygons such that each one is the $\Pev$-evolute of the other one. One can also construct a polygon that is its own $\Pev$-evolute. For this, one needs to construct a closed chain of rhombi in which every two consecutive rhombi share a side, and all these sides are parallel. If the number of rhombi in the chain is even then their diagonals form two polygons that are each other's $\Pev$-evolutes, and if the number of rhombi is odd then the diagonals form a polygon that is its own $\Pev$-evolute. See Figure \ref{self}.}
\end{example}

\subsection{Small values of $n$}
\label{small}

Let us describe the behavior of the $\Pev$-evolutes for  polygons with small number of vertices.

\subsubsection{The case $n=3$}\label{small3}  This case is trivial. The first $\Pev$-evolute degenerates to a point, the circumcenter of the triangle, which is at the same time the pseudo-Steiner point of the triangle. Note that the classical Steiner point is usually different from the circumcenter.

\subsubsection{The case $n=4$}\label{small4}This case is well understood: if a quadrilateral $\P$ is inscribed then its $\Pev$-evolute is a point, the center of the circle; otherwise $\Pev^2(\P)$ is homothetic to $\P$, with the similarity coefficient depending on the angles, but not on the lengths of the sides of the quadrilateral; see \cite{Be,Gr1,Gr2,Ki,La,RT,Sh} for a detailed study of this problem. The center of homothety is the pseudo-Steiner point of the quadrilateral, and again it is in general different from the Steiner point. Also see \cite{Tsu} for higher-dimensional and non-Euclidean versions of these results. 

We give a proof using above developed machinery.

\begin{proposition}
\label{Tsukerman} Let $\P$ be a quadrilateral. Then $\Pev^2(\P)$ is homothetic to $\P$. 
\end{proposition}
\begin{proof}
Let $C$ and $S$ be as in  Section \ref{sec:PsSt}. Let $U$ be the 4-dimensional $p$-space, and $W \subset U$ be the $\Pev_\theta$-invariant subspace spanned by the vectors $C$ and $S$. 
Denote by $\overline \Pev$ the induced linear map on $U/W$. We want to  prove that $\overline \Pev^2$ is a dilation; then $\Pev_\theta$ will be a composition of a dilation and a parallel translation, as needed.

Note that $\overline \Pev$ is a $2\times 2$ matrix. If Tr\,$\overline\Pev=0$ then $\overline \Pev^2$ is a dilation. Thus it suffices to show that Tr\,$\overline \Pev=0$.

Since the trace is linear, 
$${\rm Tr}\,\overline \Pev= {\rm Tr}\, \Pev - {\rm Tr}\, \Pev_W.$$
The latter trace is zero, since the map is a rotation by $\pi/2$, and we need to compute Tr\,$\Pev$. According to Lemma \ref{ParamPBisect}, this trace equals
\begin{equation}
\label{zerotrace}
\sum \frac{\sin (\theta_{j+1/2} - \theta_{j-1/2})}{2 \sin \theta_{j-1/2}\, \sin \theta_{j+1/2}} = \sum \frac{\cot \theta_{j-1/2} - \cot  \theta_{j+1/2}}{2} =0,
\end{equation}
as needed. 
\end{proof}

\subsubsection{The case $n=5$} \label{small5} B. Gr\"unbaum \cite{Gr1,Gr2} observed in computer experiments that the third $\Pev$-evolute of a pentagon is homothetic to the first one. We prove this below. 

\begin{proposition}
\label{Grun}
Let $\P$ be a pentagon. Then $\Pev^3(\P)$ is homothetic to $\Pev(\P)$ (see Figure \ref{pentastab} for the illustration).
\end{proposition}

\begin{figure}[hbtp]
\centering
\includegraphics[height=3.2in]{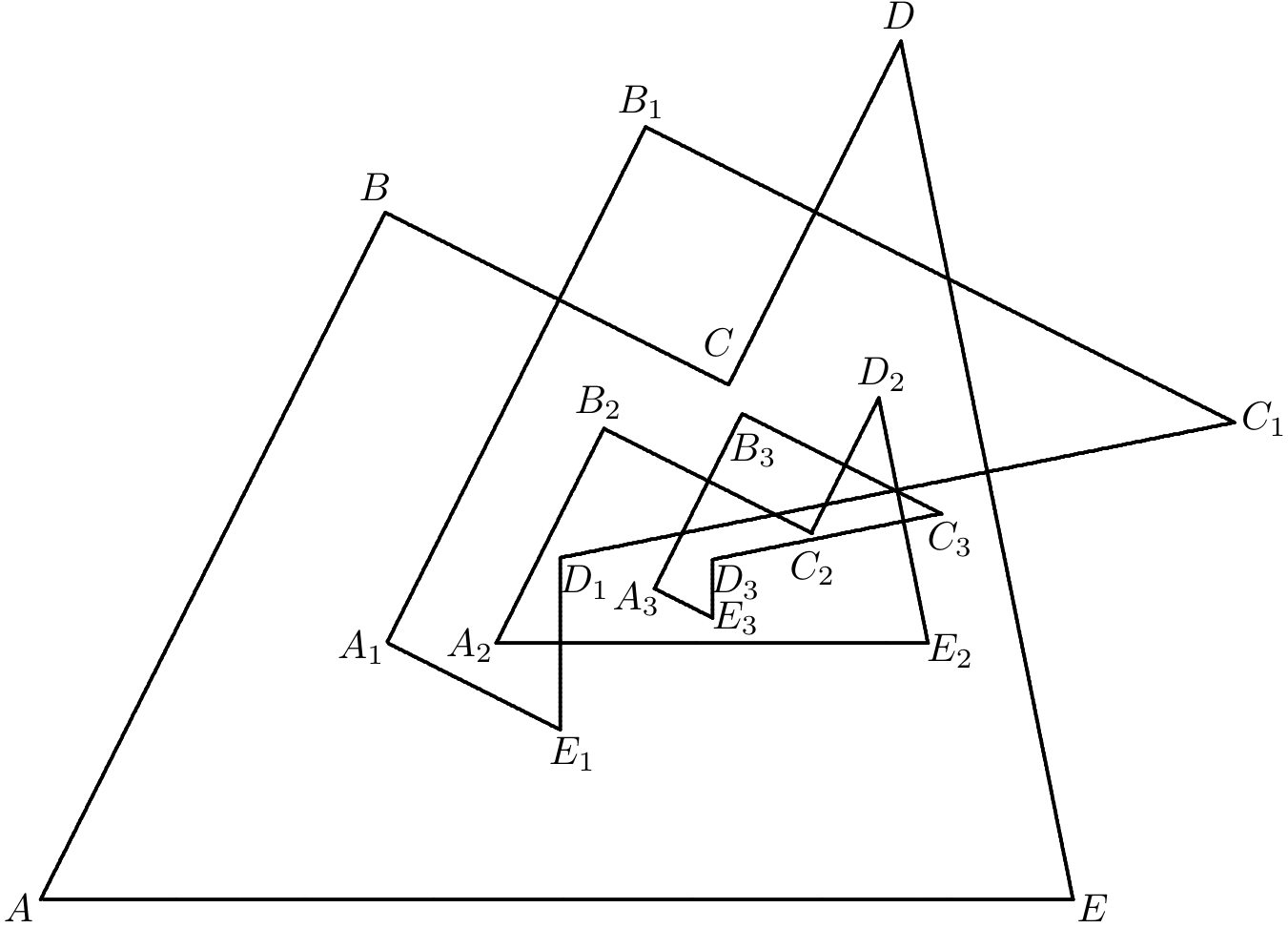}
\caption{Consecutive $\Pev$-evolutes of a pentagon. The third evolute $A_3B_3C_3D_3E_3$ is homothetic to the first evolute $A_1B_1C_1D_1E_1$.}
\label{pentastab}
\end{figure}

\begin{proof} The proof is similar to the one for quadrilaterals, given above. We want to show that if $\P$ is in the image of the map $\Pev$, then $\Pev^2(\P)$ is homothetic to $\P$.
	
Let $V$ be the 5-dimensional $p$-space and $U\subset V$ be the image of the linear map $\Pev$. Abusing notation, we denote the restriction of $\Pev$ to $U$ by the same letter. Let $W \subset U$ be the $\Pev$-invariant subspace spanned by the vectors $C$ and $S$. Denote by $\overline \Pev$ the induced linear map on $U/W$. We want to  prove that $\overline \Pev^2$ is a homothety.
	
As before, it suffices to show that Tr\,$\overline \Pev =0$, and as before, $\mathrm{Tr}\, \overline\Pev = {\rm Tr}\,\Pev_U = {\rm Tr}\,\Pev_V$. But the latter vanishes, since (\ref{zerotrace}) holds for all values of $n$.
\end{proof}

Again, the center of homothety is the pseudo-Steiner point of the pentagon.

Certainly, in some particular situation, the $\Pev$-evolute of some order can also degenerate to a point. For the case of pentagons we were able to give a necessary and sufficient condition of the third evolute collapsing to a point. We will need  the notations $\B$ and $\B_j$ introduced in Section \ref{DefP-evo}. 

\begin{proposition}\label{n5Pdegenrate}
Let $\P$ be a pentagon. Then $\Pev^3(\P)$ is a point if and only if  $\sum_{j=1}^5\sin2\B_j(\alpha)= 0$. Equivalently, the area of the inscribed pentagon with sides parallel to those of $\P$ must be zero.
\end{proposition}

\begin{proof} The proof is based on Proposition \ref{KernelRank} and results of Section \ref{AP-inv}. The results of Section \ref{Pinvodd} show that a cooriented odd-gon belongs to the image of the evolute transformation if and only if its {\sl quasiperimeter} is zero. The quasiperimeter of an odd-gon is defined and calculated in Section \ref{oddn}: for an odd-gon with the sides $l_j=(\alpha_j,p_j)$ it is equal to $\sum_{j=1}^np_j\sin\B_j(\alpha)$. This formula, combined with the formula in Proposition \ref{KernelRank}, shows that for the generator of the kernel of the transformation $p\mapsto p^\ast$ the quasiperimeter is$$\sum_{j=1}^5 p_j^\ast\sin\B_j(\alpha)=\sum_{j=1}^5\cos\B_j(\alpha)\sin\B_j(\alpha)=\frac12\sum_{j=1}^5\sin2\B_j(\alpha).$$ 

\begin{figure}[htbp]
\centering
\includegraphics[width=2.3in]{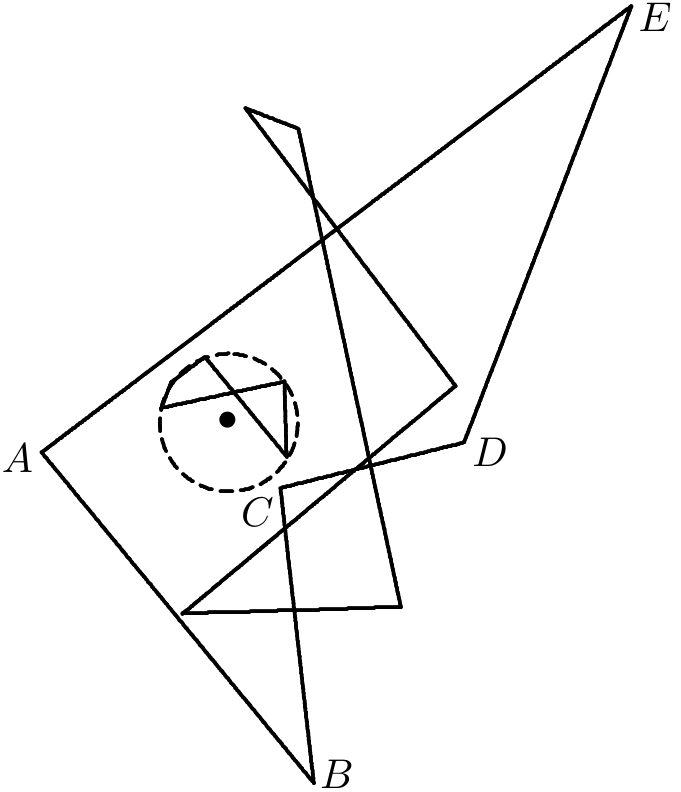}
\caption{A pentagon $ABCDE$ whose third $\Pev$-evolute is a point (so its second $\Pev$-evolute is inscribed into a circle).}
\label{NilPent}
\end{figure}

Thus, the condition $\sum_{j=1}^5\sin2\B_j(\alpha) = 0$ means precisely that the image of the map $p\mapsto p^\ast$ is contained in its kernel. This means that $\rk \Pev_\theta^2 = 3$. It follows that the homothety $\Pev_\theta^2$ of $U/W$, considered in the previous proposition, is of factor $0$.
\end{proof}

Therefore any set of turning angles $\{\theta_j\}_{j=1}^5$ providing shrinking to a point on the third iteration of $\Pev$ can be obtained as follows. Let $\{\beta_j\}_{j=1}^5$ be a solution of $\sum_j\sin\beta_j=0$ and $\sum_j \beta_j = 0\pmod{2\pi}$. Then $\theta_{j+\frac12} = (\beta_j + \beta_{j+1})/2$. See Figure \ref{NilPent} for an example.

\subsubsection{The case $n=6$} \label{small6}
In Section \ref{Dyngeneric}, we will discuss the dynamics of the $\Pev$-transformation for a generic polygon, in particular, for a generic hexagon. Below we present some particular examples of degenerate behavior.

\begin{proposition}
\label{hexapar} 
Let $\P$ be a hexagon. Then\smallskip

{\rm(i)} If the angles of $\P$ satisfy the condition $\B(\alpha) = 0\pmod \pi$ then $\Pev^3(\P)$ is homothetic to $\Pev(\P)$.\smallskip

{\rm(ii)} If the opposite sides of $\P$ are parallel then $\Pev^3(\P)=\Pev(\P)$, that is, 
$\Pev(\P)$ and $\Pev^2(\P)$ are each other's evolutes. See Figure \ref{selfdual}.
\end{proposition}

\begin{figure}[hbtp]
\centering
\includegraphics[width=2.1in]{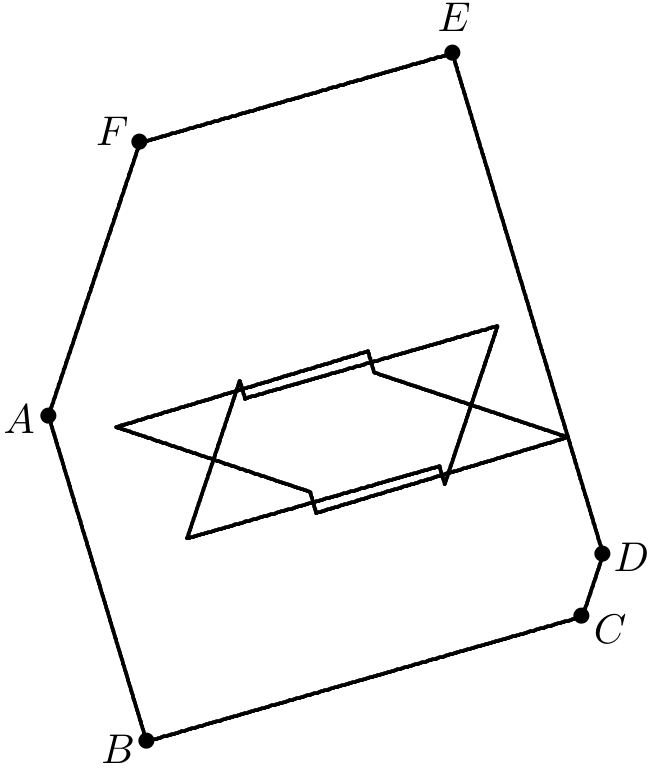} 
\caption{The $\Pev$ and $\Pev^2$ images of a hexagon $ABCDEF$ with parallel opposite sides are each other's $\Pev$-evolutes.}
\label{selfdual}
\end{figure}

The second part of this statement was already noted for equiangular 
hexagons, see Figure \ref{RegHexEvol}.

\begin{proof} The proof of (i) is similar to the proof of Proposition \ref{Grun}. Let $V$ be the 6-dimensional $p$-space, $K\subset V$ the kernel of the map $\Pev_{\theta}$, and $U\subset V$ its image. It follows from Proposition \ref{KernelRank} that  $K$ is 2-dimensional, and hence $U$ is 4-dimensional. The rest of the proof is the same as in Proposition \ref{Grun}.
	
To prove (ii), consider Figure \ref{hexaproof}. The hexagon $\Pev(\P)$ is labelled $ABCDEF$. Consider the circles circumscribed about $\triangle FAB$ and $\triangle ABC$. The line through the centers of these circles is the perpendicular bisector of the segment $AB$. We need to show that the centers are at equal distance from $AB$, that is, that the radii of the circles are equal.
	
\begin{figure}[hbtp]
\centering
\includegraphics[height=2.4in]{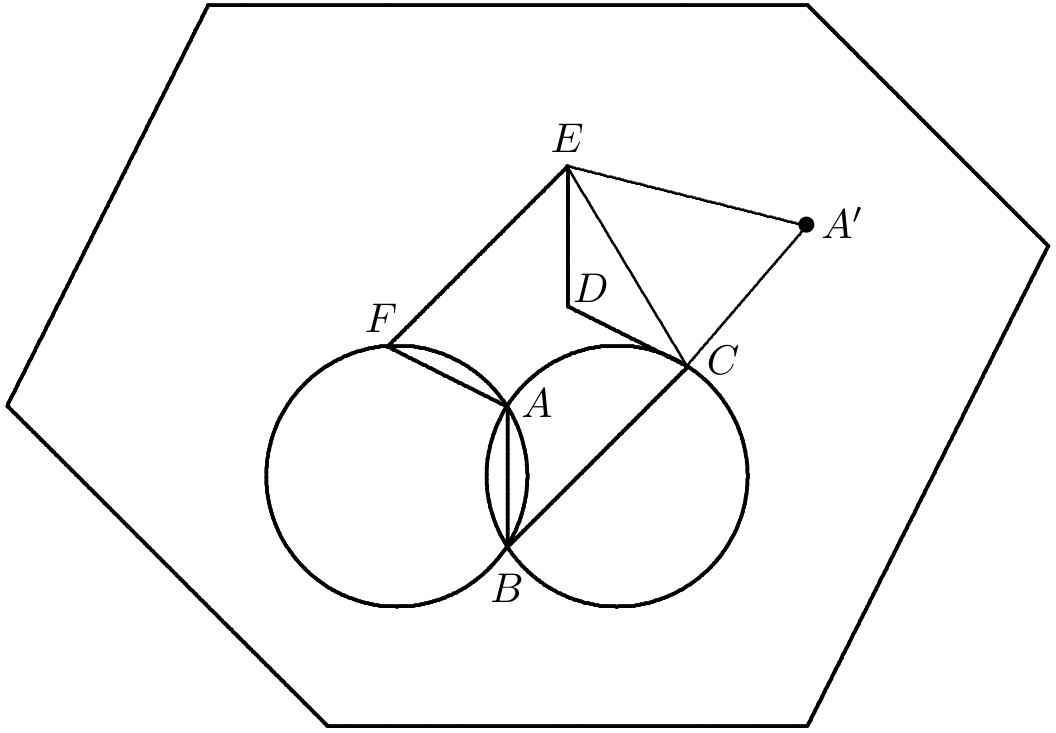} 
\caption{Proof of Proposition \ref{hexapar}(ii).}
\label{hexaproof}
\end{figure}

The composition of reflections 
$$\mathcal{S}_{EF} \circ \mathcal{S}_{DE} \circ \mathcal{S}_{CD} \circ \mathcal{S}_{BC} \circ \mathcal{S}_{AB} \circ \mathcal{S}_{FA}$$
is the identity. Applying this to point $A$, we obtain
$$A':=\mathcal{S}_{CD} \circ \mathcal{S}_{BC} (A) = \mathcal{S}_{DE}  \circ \mathcal{S}_{EF} (A).$$ 
The composition of two reflections is the rotation through the angle twice that between the axes. Thus $\angle ACA' = 2 \angle BCD$. Since the points $A$ and $A'$ are symmetric with respect to the line $CE$, one has $\angle ACE = \angle BCD$, and hence $\angle DCE = \angle BCA$. But $\angle DCE = \angle AFB$, therefore $\angle AFB = \angle BCA$.
	
Finally, the radii of the two circles, found via the Sine Rule, are:
$$\frac{|AB|}{2 \sin (\angle AFB)}\qquad{\rm and}\qquad\frac{|AB|}{2 \sin (\angle BCA)},$$
hence the radii are equal, as needed.
\end{proof}

\subsection{Dynamics of the $\Pev$-transformation in the generic case}\label{Dyngeneric}
\subsubsection{Symmetry of the spectrum of $\Pev_\theta$.}
Let us further investigate the spectral properties of the matrix $\Pev_\theta$. The above  considerations show  that the matrix $\Pev_\theta$ is not always diagonalizable: the geometric multiplicity of the zero eigenvalue can be larger than the algebraic multiplicity. 
From Proposition \ref{Pinvariant}, it follows that the spectrum of $\Pev_\theta$ always contains $\pm i$. The following theorem provides more information on the spectrum of $\Pev_\theta$.

\begin{theorem}
\label{SymSpectrum}
The spectrum of the linear map $\Pev_\theta \colon p \mapsto p^*$ is symmetric with respect to the origin. Moreover, the opposite eigenvalues have the same geometric multiplicity and the same sizes of the Jordan blocks.
\end{theorem}

Below, we give two proofs of this theorem. The first one is elementary and self-contained, but it does not cover the statement concerning the Jordan blocks. The second one covers all the statements, but it uses a reference to a recent result of one of the authors (I. I.), which may be less convenient to the reader.

\begin{proof}[First Proof (partial)] The argumentation below will show that the spectrum is symmetric, including the algebraic multiplicity of the eigenvalues only.
	
\begin{lemma}
If for an $n \times n$ matrix $A$ we have
\begin{equation}
\label{TraceOdd}
\tr A^k = 0 \text{ for all odd }k \le n,
\end{equation}
then the spectrum of $A$ is symmetric with respect to the origin, including the algebraic multiplicity.
\end{lemma}
\begin{proof}
Let $\Lambda = \{\lambda_1, \ldots, \lambda_n\}$ be the multiset of eigenvalues of $A$, listed with their algebraic multiplicity. We have $-\Lambda = \Lambda$ if and only if the characteristic polynomial has the form $\lambda^m P(\lambda^2)$, that is, if and only if the elementary symmetric polynomials $\sigma_k$ in $\lambda$ vanish for all odd $k$. We have $\tr A^k = \tau_k$, where
$\tau_k = \lambda_1^k + \cdots + \lambda_n^k$.
Newton's formulas relate $\sigma_k$ and $\tau_k$:
\[ k\sigma_k = \sum_{j=1}^k (-1)^{j-1} \sigma_{k-j} \tau_j. \]
If $k$ is odd, then either $j$ or $k-j$ is odd. Therefore the vanishing of $\tau_j$ and $\sigma_{j'}$ for all odd $j, j'$ up to $k$ implies the vanishing of $\sigma_k$. The lemma follows by induction on odd $k$. 
\end{proof}
	
It now suffices to prove \eqref{TraceOdd} for $A=\Pev_\theta$ (see \eqref{eqnA}). We have
\[ \Pev_\theta = A_1 + \cdots + A_n \]
where $A_j$ is the matrix with zero entries except for the $(j-1,j)\times(j-1,j)$ block
\[
\begin{pmatrix}
-a_j & b_j\\ -b_j & a_j
\end{pmatrix}.
\]
Let us consider the case of an even $n$ first. Put
\[ \Ao = A_1 + A_3 + \ldots + A_{n-1}, \quad \Ae = A_2 + A_4 + \ldots + A_n, \]
so that $\Pev_\theta = \Ao + \Ae$. Observe that
\begin{equation}
\label{SquareId}
\Ao^2 = -\id,\ \Ae^2=-\id.
\end{equation}
We have
\begin{equation}
\label{eqnAk}
\Pev_\theta^k = \sum\limits_{\{i_1,\dots i_m\},{ \{j_1,\dots, j_m\}}} \Ao^{i_1} \Ae^{j_1} \cdots \Ao^{i_m} \Ae^{j_m},
\end{equation}
where the sum is taken over all integer partitions $k = i_1 + j_1 + \cdots + i_m + j_m$ with $i_1$ and $j_m$ non-negative and other summands positive. We claim that, for odd $k$, every summand has zero trace.
	
Indeed, due to \eqref{SquareId}, we can decrease any $i_s$ or $j_t$ by two at the price of a minus sign. In the reduced form, every summand will become an alternating product
\[ \pm \Ao \Ae \cdots \Ao \quad \text{or} \quad \pm \Ae \Ao \cdots \Ae \] 
Due to the cyclic property of the trace $\tr(XY) = \tr(YX)$, we can make further cancellation and obtain
\[ \tr(\pm \Ao \Ae \cdots \Ao) = \pm \tr A_\pm = 0, \quad \tr(\pm \Ae \Ao \cdots \Ae) = \pm \tr A_\pm = 0.
\]
Thus all summands in \eqref{eqnAk} have zero trace, hence $\tr A^k = 0$ for odd $k$ and even $n$.
	
If $n$ is odd, then consider the $2n \times 2n$ matrix $\widetilde{\Pev}_\theta$ with $a_{n+j} = a_j$ and $b_{n+j} = b_j$. By the above argument, $\tr\widetilde{\Pev}_\theta^k = 0$ for all odd $k$. On the other hand, for all $k < n$, we have $\tr \widetilde{\Pev}_\theta^k = 2 \tr \Pev_\theta^k$. Thus we have \eqref{TraceOdd} also in the case of  odd $n$.
\end{proof}
\begin{proof}[Second Proof (complete).]
Consider the cyclically tridiagonal matrix
\[
\Mev_\theta = 
\begin{pmatrix} 
-(a_1+a_2)& b_2  & 0&\cdots 0&b_1\\ b_2& -(a_2+a_3) & b_3&\ddots & 0\\ \vdots&\ddots&\ddots &\ddots& 0\\ 0 &  \ddots     &   b_{n-1} &-(a_{n-1}+a_n)& b_n\\ b_1&0&\cdots  0&b_n& -(a_n+a_1)
\end{pmatrix}.
\]
A direct computation shows that the matrix $\mathcal{M}_\theta\Pev_\theta$ is antisymmetric. Due to $\mathcal{M}_\theta^\top = \mathcal{M}_\theta$, this can be written as
\begin{equation}
\label{eqn:AM}
\Pev_\theta^\top \mathcal{M}_\theta= -\mathcal{M}_\theta \Pev_\theta.
\end{equation}
By \cite[Theorem 3]{Izm}, we have $\det \mathcal{M}_\theta \ne 0$ whenever $\sum_{j=1}^n \theta_{j-1/2} \not\equiv 0(\mathrm{mod}\ 2\pi)$. If this is the case, then
\[ {\cal M}^{-1}_\theta \Pev_\theta^\top {\cal M}_\theta = -\Pev_\theta, \]
and hence the matrices $\Pev_\theta^\top$ and $-\Pev_\theta$ have the same Jordan normal form. This implies Theorem \ref{SymSpectrum} under the assumption $\sum_{j=1}^n \theta_{j-1/2} \not\equiv 0(\mathrm{mod}\ 2\pi)$. The general case follows by continuity.
\end{proof}

\begin{remark}
{\rm Theorem \ref{SymSpectrum} holds for any matrix of the form \eqref{eqnA} with $b_j^2-a_j^2=1$ for all $j$. }
\end{remark}
\begin{remark}
{\rm According to formula \eqref{lengths} in Lemma \ref{CoordPolygon}, the matrix $\Mev_\theta$ computes the side lengths of a polygon from its support numbers:
$\ell = \Mev_\theta p.$ The symmetry of $\Mev_\theta$ is explained by the fact that $\ell_i = \dfrac{\partial\, \mathrm{area}(p)}{\partial p_i}$, where $\mathrm{area}(p)$ is the area of the polygon, so that
\[
\Mev_\theta = \left( \frac{\partial^2\mathrm{area}}{\partial p_i \partial p_j} \right)
\]
The operator $p \mapsto\Mev_\theta p$ is a discrete analog of the operator $p \mapsto p+p''$ (the radius of curvature in terms of the support function). The spectrum of $\Mev_\theta$, its relation to the discrete Wirtinger inequality and to the isoperimetric problem are discussed in \cite{Izm}.}
\end{remark}
\begin{remark}
{\rm It is easy to show that the matrix $-\Pev_\theta^\top$ transforms the side lengths of a polygon to the side lengths of its evolute: $\ell^* = -\Pev_\theta^\top\ell$. This agrees with  equation \eqref{eqn:AM}:
\[ \Mev_\theta\Pev_\theta p = \Mev_\theta p^* = \ell^* = -\Pev_\theta^\top \ell = -\Pev_\theta^\top \Mev_\theta p. \]
The antisymmetry of $\Mev_\theta\Pev_\theta$ is equivalent to $\langle\Mev_\theta\Pev_\theta p, p \rangle = 0$, that is, to
$\sum_i p_i \ell_i^* = 0.$ 
We do not know any geometric explanation of this antisymmetry.}
\end{remark}
\begin{remark}
{\rm Theorem \ref{SymSpectrum} shows that if $n$ is odd, then the matrix $\mathcal{P}_\theta$ is degenerate. We have already known this from Proposition \ref{KernelRank}.}
\end{remark}

\subsubsection{The limit behavior of iterated $\Pev$-evolutes}
\label{limitbehavior}

Theorem \ref{SymSpectrum} states that the spectrum of the matrix $\Pev_\theta$ is symmetric with respect to 0. Let $\lambda$ be the maximum module of eigenvalues of $\Pev_\theta$. Generically, there are three options: (i) $\Pev_\theta$ has precisely two eigenvalues of module $\lambda$: $\lambda$ and $-\lambda$; (ii) $\Pev_\theta$ has precisely two eigenvalues of module $\lambda$: $\lambda i$ and $-\lambda i$;  (iii) $\Pev_\theta$ has precisely four eigenvalues of module $\lambda$: $\pm\mu$ and $\pm\overline\mu$, where $|\mu|=\lambda$ and $\mu$ is neither real, nor purely imaginary. In Case (i), $\lambda^2$ is a multiple 2 eigenvalue of $\Pev_\theta^2$, in Case (ii), $-\lambda^2$ is a multiple 2 eigenvalue of $\Pev_\theta^2$. In both these cases, $\Pev_\theta^2$ possesses a two-dimensional eigenspace with the eigenvalue, respectively, $\lambda^2$ or $-\lambda^2$. In Case (iii), $\Pev_\theta$ has a four-dimensional invariant subspace with the eigenvalues $\pm\mu,\pm\overline\mu$. 

Take a randomly chosen generic (cooriented) $n$-gon $\P=\P_0$ (in the drawings below, $n=6$); let $\theta$ be the sequence of turning angles and $\lambda$ be the maximal module of eigenvalues of $\Pev_\theta$. Consider the sequence of $\Pev$-evolutes:  $\P_1,\P_2,\P_3,\dots$. Roughly speaking, these evolute will exponentially expand with the base $\lambda$ and drift away from the origin. To compensate these expansion and drifting, we apply to every polygon $\P_0,\P_1,\P_2,\dots$ two consecutive transformations: a translation which shifts the vertex centroid to the origin, and then a rescaling with respect to the origin which makes the maximum distance from the vertices to the origin equal to $1$. We keep the notation $\P_0,\P_1,\P_2,\dots$ for the transformed polygons.

\begin{figure}[htbp]
\centering
\includegraphics[width=4.1in]{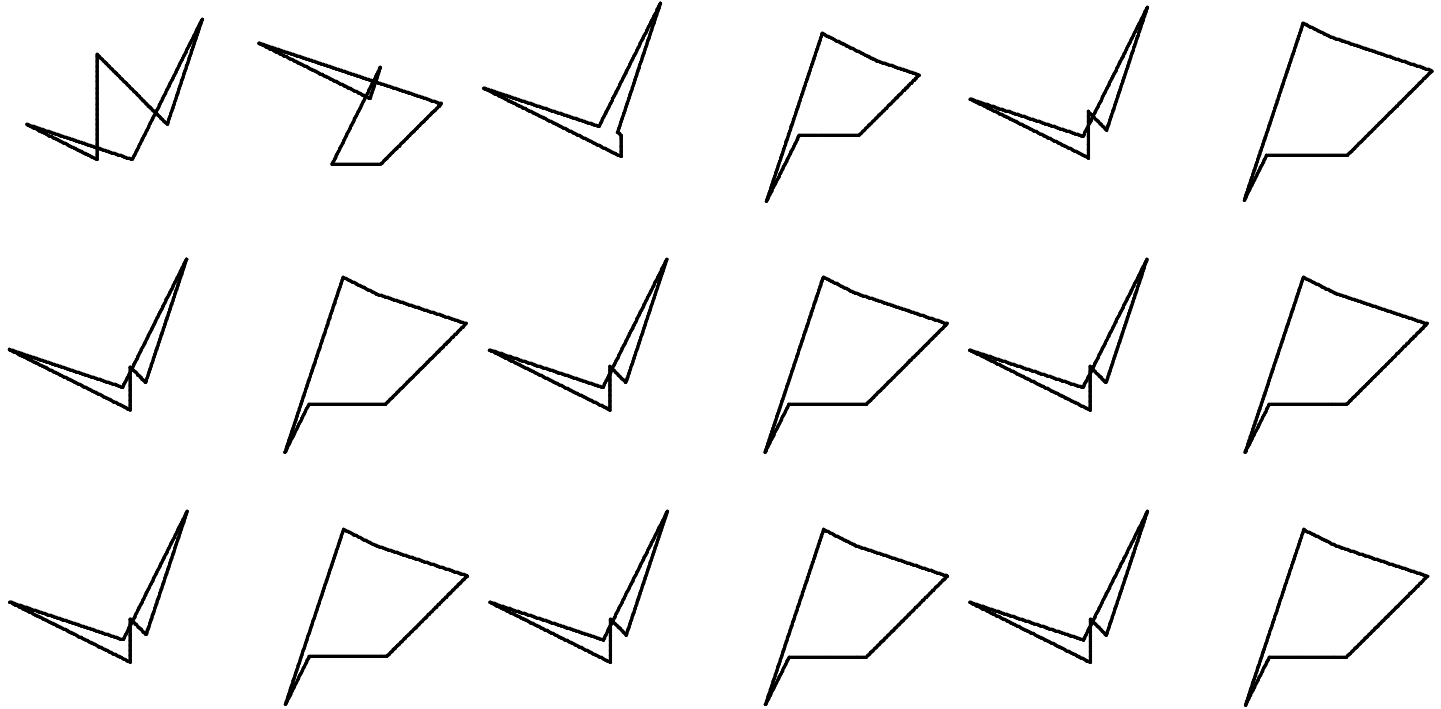}
\caption{A polygon and a sequence of $\Pev$-evolutes; Case (i).}
\label{positive}
\end{figure}

\begin{figure}[htbp]
\centering
\includegraphics[width=4.1in]{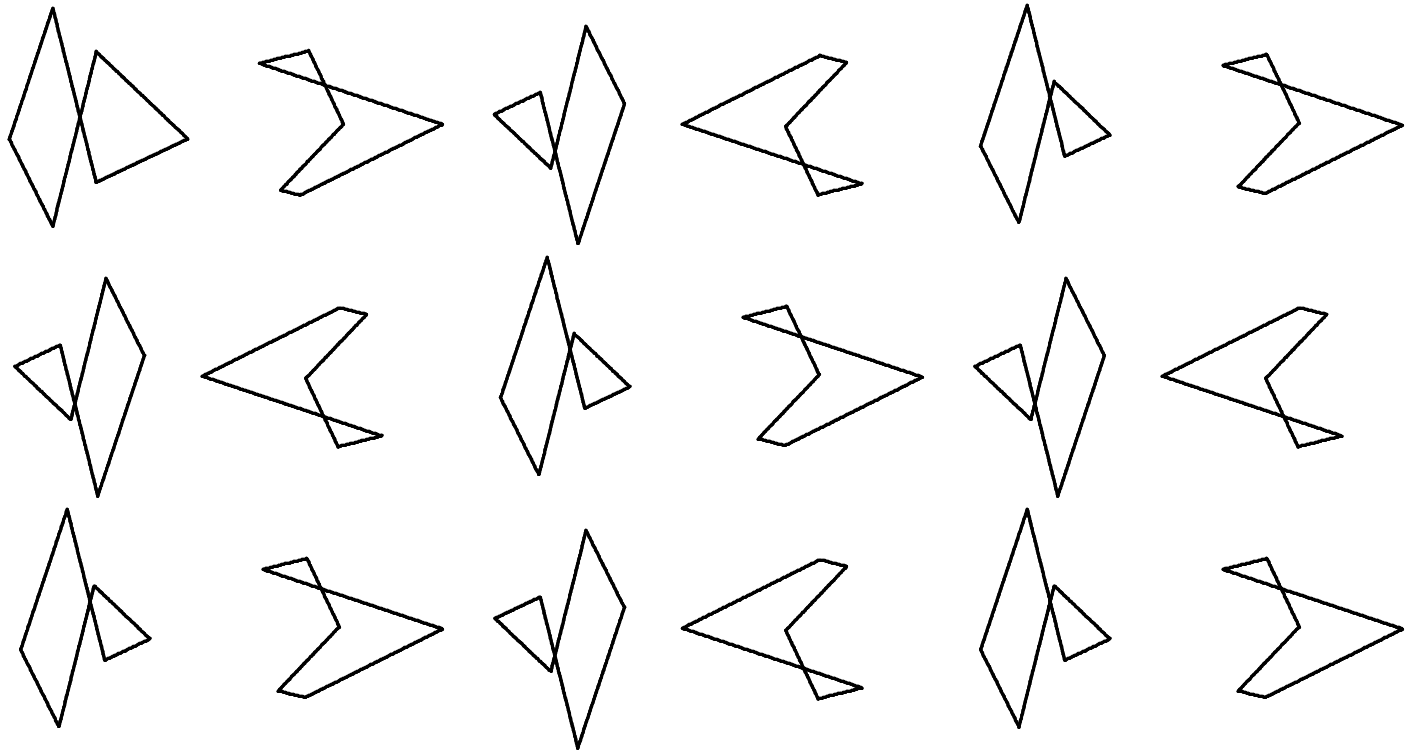}
\caption{A polygon and a sequence of $\Pev$-evolutes; Case (ii).}
\label{negative}
\end{figure}

If $\theta$ represents the first of the three cases (the maximum module eigenvalues of $\Pev_\theta$ are $\lambda$ and $-\lambda$), then, for a large $N$, the polygons $\P_N$ and $\P_{N+2}$ are almost the same, and the limit polygons $\lim\P_{2n}$ and $\lim\P_{2n+1}$ are (up to a translation and a rescaling), $\Pev$-evolutes of each other (see Figure \ref{positive}). If $\theta$ represents the second of the three cases (the maximum module eigenvalues of $\Pev_\theta$ are $\lambda i$ and $-\lambda i$), the polygons $\P_N$ and $\P_{N+2}$ are almost central symmetric to each other, the polygon $\mathbf{Q}_1=\lim\limits_{n\to\infty}\P_{4n+1}$ is the $\Pev$-evolute of $\mathbf{Q}_0=\lim\limits_{n\to\infty}\P_{4n}$, and the $\Pev$-evolute of $\mathbf{Q}_1$ is $-\mathbf{Q}_0$ (again up to a translation and a rescaling; see Figure \ref{negative}).

In the third case (the maximum module eigenvalues of $\Pev_\theta$ are $\pm\mu$ and $\pm\overline\mu$ where $\mu$ is neither real, nor purely imaginary), the behavior of the sequence of $\Pev$-evolutes appears irregular; see an example of this in Figure \ref{imaginary}. Still, in Figure \ref{imaginary}, for $N$ large, there is a certain similarity between $\P_N$ and $-\P_{N+8}$; this may be an indication of a proximity of $\mu^4$ to some negative real number.

\begin{figure}[htbp]
\centering
\includegraphics[width=3.25in]{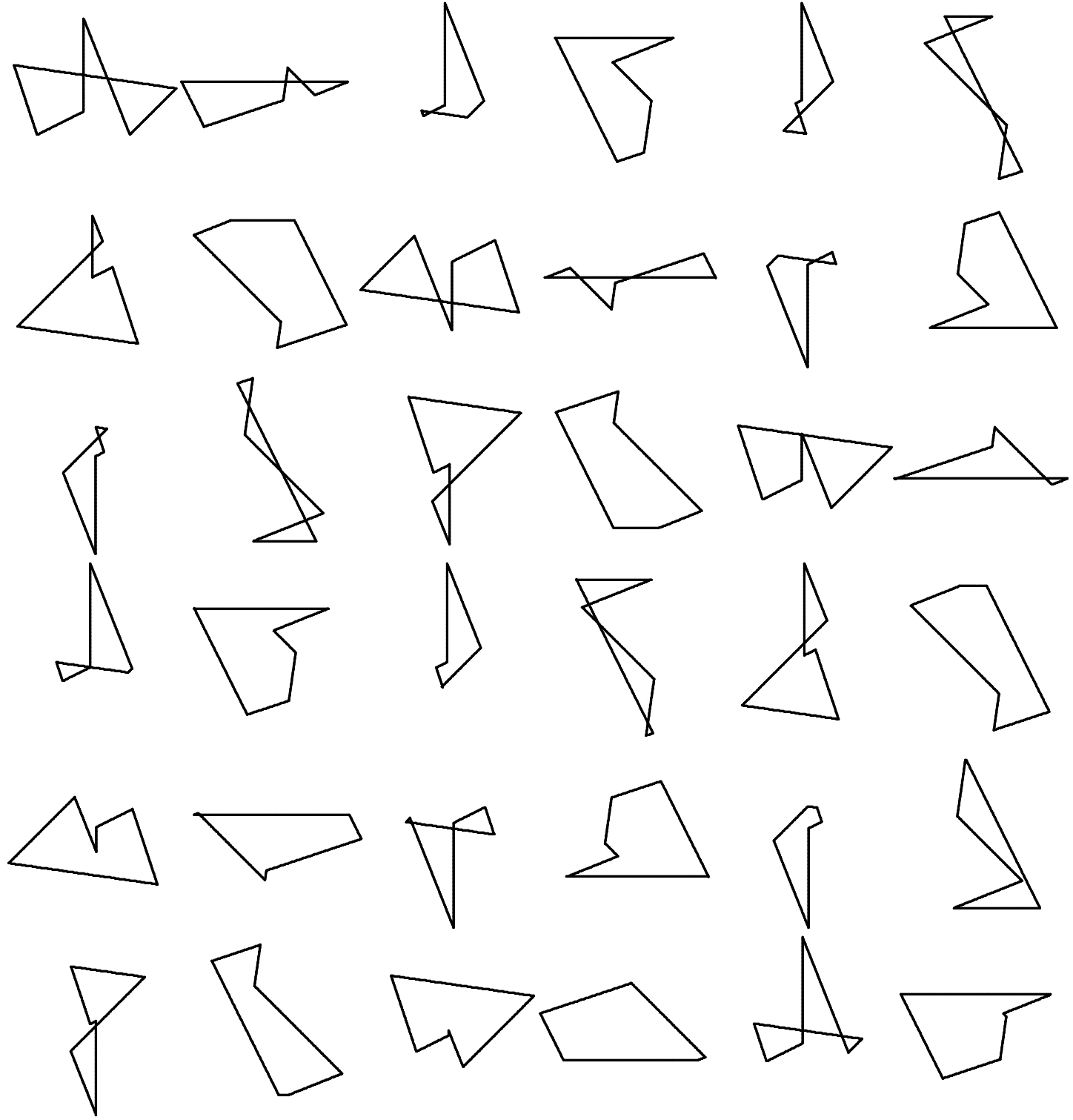}
\caption{A polygon and a sequence of $\Pev$-evolutes; Case (iii).}
\label{imaginary}
\end{figure}

\section{$\Aev$-evolutes}\label{A-evo}

As we noted in Introduction, there exists another way to discretize the evolute construction. The vertices of this evolute of a polygon are the centers of the circles tangent to the triples of the consecutive sides of the given polygon, and its sides are bisectors of angles of the given polygon. We call this evolute the angle bisector evolute, or an $\Aev$-evolute, and we denote an $\Aev$-evolute of a polygon $\P$ by $\Aev(\P)$. Since every angle has two bisectors, this construction is not single valued: an $n$-gon has as many as $2^n$ $\Aev$-evolutes.

\subsection{$\Aev_o$-evolutes}\label{Ao-evo}

\subsubsection{Definition of  $\Aeo$-evolute}\label{defAo}

To secure a choice of a preferred $\Aev$-evolute, we apply the construction to a (co)oriented polygon. If $l$ and $l'$ are two oriented crossing lines, and $e$ and $e'$ are their orienting unit vectors, then there appear two non-zero vectors, $e+e'$ and $e'-e$. The oriented lines generated by these vectors are denoted by $l^\circ$ and $l^\ast$; they are called, respectively, an {\it exterior} and {\it interior bisector}\footnote{If $l$ and $l'$ are consecutive sides of a cyclically oriented polygon, then $l^\circ$ and $l^\ast$ are the exterior and interior bisectors of the angle of this polygon; this is a geometric justification of this terminology.} of the angle $\angle(l,l')$. The definition of the lines $l^\circ$ and $l^\ast$, {\it but not of their orientatiions}, can be continuously extended to the case when the lines $l$ and $l'$ (but not necessarily their orientations) coincide, provided that the ``intersection point" $P$ is marked on $l=l'$: if the orientations of $l$ and $l'$ agree, then $l^\circ=l,\, P\ni l^\ast\perp l$; if they disagree, then $P\ni l^\circ\perp l,\, l^\ast=l$. These constructions are illustrated in Figure \ref{OrBisector}.

\begin{figure}[htbp]
\centering
\includegraphics[width=4.6in]{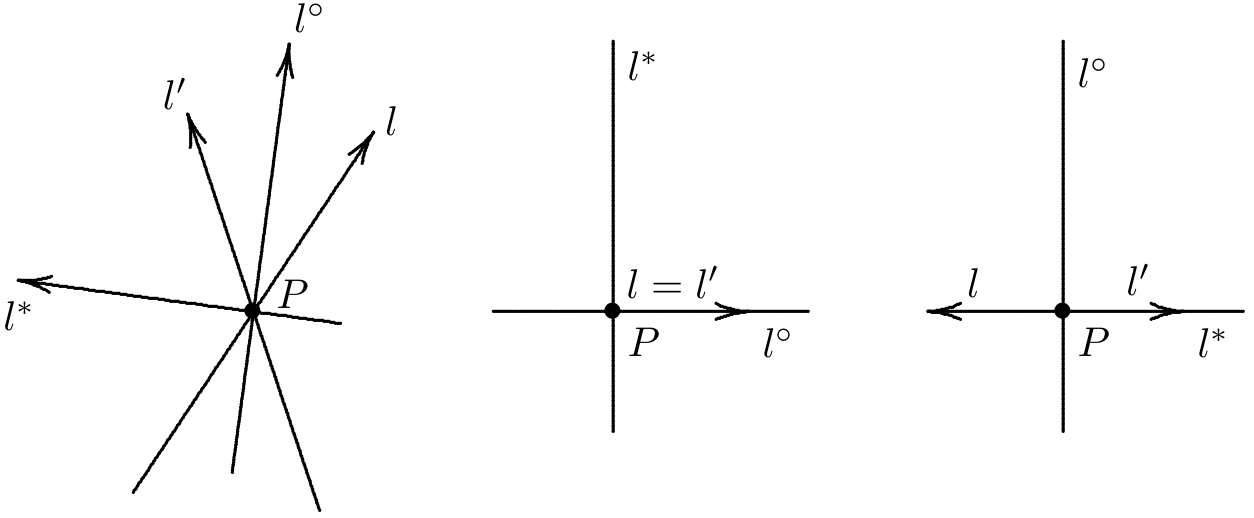}
\caption{The bisectors of oriented lines.}
\label{OrBisector}
\end{figure}

\begin{definition} \label{Aeo}
Consider a cooriented polygon $\P =(l_1, \ldots, l_n)$. The cooriented polygon $\P^*=(l^*_{1+\frac12}, \ldots, l^*_{n+\frac12})$, where $l^\ast_{i+\frac12}$ is the interior bisector of the angle $\angle(l_i,l_{i+1})$, is called the  \emph{oriented angle bisector evolute of $\P$}. We will denote it by $\Aeo(\P)$. 
\end{definition}
\begin{remark}
{\rm 
Obviously, an $n$-gon has $2^n$ different coorientations, precisely as it has $2^n$ different choices of bisectors. But the choices of coorientations and bisectors do not bijectively correspond to each other: if a coorientation is replaced by the opposite coorientation, then the interior and exterior bisectors remain interior and exterior bisectors (although their orientations reverse). Thus, only $2^{n-1}$ of $2^n$ angle bisector evolutes described in Section \ref{discretization} can be $\Aeo$-evolutes. Namely, they correspond to the choices of angle bisectors with an even number of the exterior bisectors.
}
\end{remark}

Notice that the shape of an $\Aeo$-evolute depends very essentially on the choice of the orientations of the sides. Figure \ref{Aeo_vs_Aec} shows two $\Aeo$-evolutes of the same pentagon $\P$ endowed with two different orientations: the first one, $\mathbf Q=Q_1Q_2Q_3Q_4Q_5$, corresponds to the ``cyclic" orientation of $\P$, $\overrightarrow{P_{\frac12}P_{\frac32}}$, $\overrightarrow{P_{\frac32}P_{\frac52}}$, $\overrightarrow{P_{\frac52}P_{\frac72}}$, $\overrightarrow{P_{\frac72}P_{\frac92}}$, $\overrightarrow{P_{\frac92}P_{\frac12}}$, the second one, $\mathbf R=R_1R_2R_3R_4R_5$, corresponds to the orientation of $\overrightarrow{P_{\frac32}P_{\frac12}}$, $\overrightarrow{P_{\frac32}P_{\frac52}}$, $\overrightarrow{P_{\frac52}P_{\frac72}}$, $\overrightarrow{P_{\frac92}P_{\frac72}}$, $\overrightarrow{P_{\frac12}P_{\frac92}}$.

\begin{figure}[htbp]
\centering
\includegraphics[width=4in]{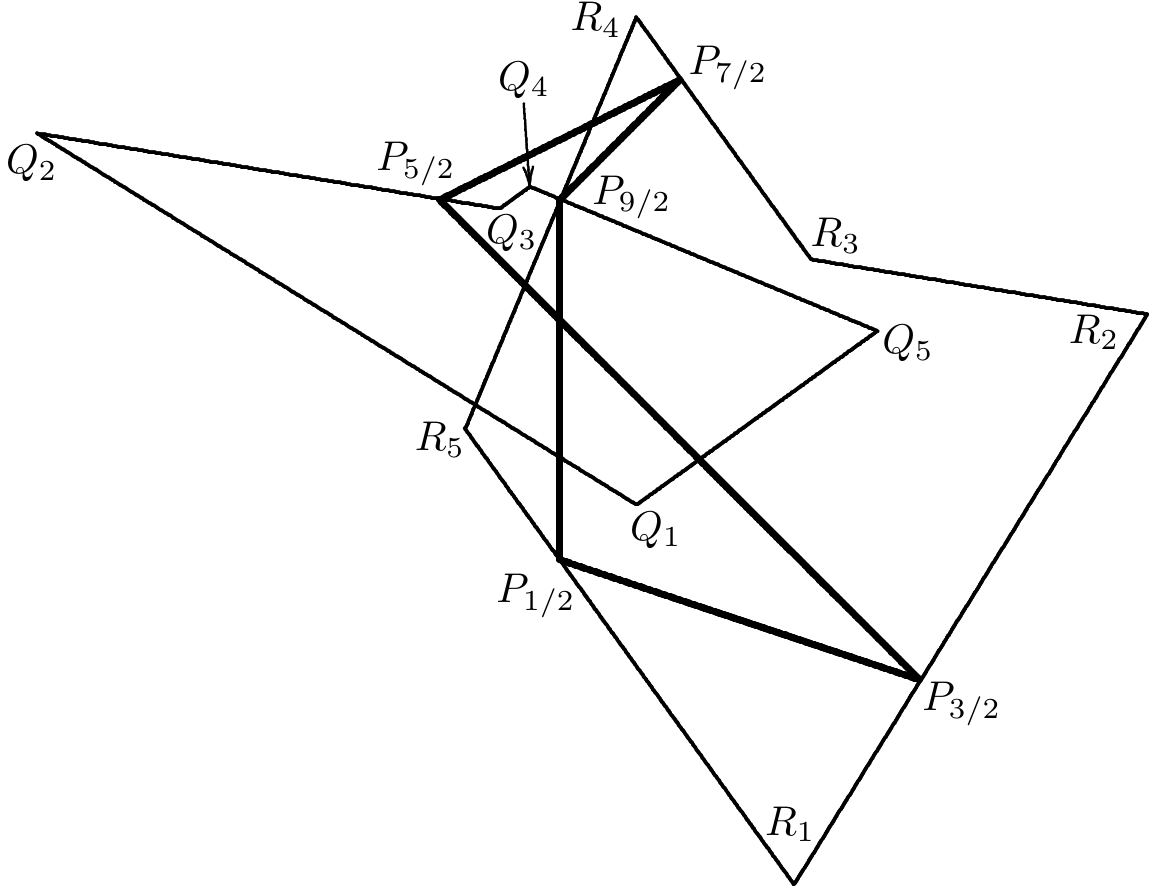}
\caption{Two $\Aeo$-evolutes of a polygon $\P=P_{\frac12}P_{\frac32}P_{\frac52}P_{\frac72}P_{\frac92}$.}
\label{Aeo_vs_Aec}
\end{figure}

\begin{remark}
{\rm 
By definition of a cooriented polygon, the consecutive lines are not parallel, hence the vertices $P_{j+1/2}$ and the bisectors $l^*_{j+1/2}$ are well-defined. However, it can happen that the $\Aeo$-evolute of a cooriented polygon has a pair of parallel consecutive sides: if the lines $l_{j-1}$ and $l_{j+1}$ are parallel, then $l^*_{j-1/2}$ and $l^*_{j+1/2}$ are parallel. This is similar to the evolute escaping to infinity at an inflection point in the classical case. 
}
\end{remark}
 
 If $\P$ is a discrete multi-hedgehog, i.e., it has only positive turning angles, then $\Aeo(\P)$ also has only positive turning angles, and so $\Aeo$ is a well-defined transformation on the space of discrete multi-hedgehogs. 

Recall our convention \eqref{SmallTurn} that the values of the turning angles are taken between $-\pi$ and $\pi$. This allows to use them in computing the coordinates of the $\Aeo$-evolute.

\begin{lemma}[$\Aeo$-analog of Lemma \ref{evolsupp}]
\label{ParamBisect}
 Let $\P=\{(\alpha_{j},p_{j})\}$ and $\Aeo(\P)=\{(\alpha_{j+\frac12}^*,p_{j+\frac12}^*)\}$. Then
\[ \alpha^*_{j+\frac12} = \alpha_j + \frac{\theta_{j+\frac12}}2 + \frac{\pi}2, \quad p^*_{j+\frac12} = \frac{p_{j+1} - p_j}{2 \sin(\theta_{j+\frac12}/2)}. \]
The vector $\theta^*=(\theta_1^*,\dots, \theta_n^*)$ of the turning angles of $\Aeo(\P)$ is computed from the vector $\theta = \left(\theta_{1+\frac12}, \ldots, \theta_{n+\frac12}\right)$ of the turning angles of $\P$ by 
\begin{equation}
\label{fAngles}
\theta^* = \frac12 (\id +\mathcal{Z}) \theta,
\end{equation}
where $\mathcal{Z}$ is given by \eqref{eqnZ}. 
\end{lemma}

\begin{proof}
The turning angle $\theta^*_j$ from $l^*_{j-\frac12}$ to $l^*_{j+\frac12}$ equals the turning angle between the exterior bisectors $l^\circ_{j-\frac12}$ and $l^\circ_{j+\frac12}$. The turning angles from $l^\circ_{j-\frac12}$ to $l_j$ and from $l_j$ to $l^\circ_{j+\frac12}$ are equal to $\dfrac{\theta_{j-\frac12}}2$ and $\dfrac{\theta_{j+\frac12}}2$, respectively, and formula \eqref{fAngles} follows.
	
The formula for $\alpha^*_{j+\frac12}$ is derived in a similar way: first rotate $l_j$ to $l^\circ_{j+\frac12}$ and then to $l^*_{j+\frac12}$. To compute the signed distance from the origin to the line $l^*_{j+\frac12}$, take the scalar product of its normal
\[ n^*_{j+\frac12} = \left( -\sin\left( \alpha_j + \frac{\theta_{j+\frac12}}2 \right), \cos\left( \alpha_j + \frac{\theta_{j+\frac12}}2 \right) \right) \]
with the position vector of the point $P_{j+\frac12}$ whose coordinates are given in Lemma \ref{CoordPolygon}.
\end{proof}

\begin{corollary}
Adding a constant to all $p_j$ does not change the $\Aeo$-evolute.
\end{corollary}
\begin{proof}
It is clear from the formulas of Lemma \ref{ParamBisect}, and also can be seen geometrically.
\end{proof}

This is similar to the smooth case: the evolute does not change if every point of the curve is moved by the same distance along the coorienting normal. 

Let us emphasize  the difference between the $\Pev$ and $\Aeo$ transformations. While the $\Pev$-evolute essentially preserves the angles of the cooriented lines and acts solely on the support numbers by a linear operator depending on the turning angles, the $\Aeo$-evolute acts on pairs $(\theta_j, p_j)$ as a skew product: the turning angles are changed by a linear operator with constant coefficients while the support numbers are changed by a linear operator whose coefficients depend on the turning angles.

\subsubsection{The equiangular case}\label{equiangularAo}

\begin{theorem}[$\Aeo$-analog of  Theorem \ref{EvolDHPerp}]
\label{EvolDH}
Let $\P$ be an equiangular hedgehog with $n$ sides tangent to a hypocycloid $h$ of order $m$. Then $\Aeo(\P)$ is tangent to $\E(h)$, scaled by  factor $\dfrac{\sin(\pi m/n)}{m\sin(\pi/n)}$ with respect to its center.
\end{theorem}
\begin{proof}
One computes the coordinates of the first $\Aeo$-evolute of $\mathbf{C}_m(n)$ and $\mathbf{S}_m(n)$ using  the formulas of Lemma \ref{ParamBisect}:
\begin{gather*}
\left(\frac{2\pi  j}n, \cos  \frac{2\pi j m}n \right)_{j=1}^n \mapsto 
\frac{\sin(\pi m/n)}{\sin(\pi/n)} \left( \frac{\pi}2 + \frac{\pi}n + \frac{2\pi j}n, -\sin m \left( \frac{\pi}n + \frac{2\pi j}n \right) \right)_{j=1}^n,\\
\left(\frac{2\pi j}n, \sin  \frac{2\pi j m}n \right)_{j=1}^n \mapsto \frac{\sin(\pi m/n)}{\sin(\pi/n)} \left( \frac{\pi}2 + \frac{\pi}n + \frac{2\pi j}n, \cos m \left( \frac{\pi}n + \frac{2\pi j}n \right) \right)_{j=1}^n.
\end{gather*}
Hence the $\Aeo$-evolute of a hedgehog, tangent to the hypocycloid $h$ given by $p(\alpha) = a\cos m\alpha + b\sin m\alpha$, is tangent to the hypocycloid
\[ p(\alpha) = \frac{\sin(2\pi m/n)}{\sin(2\pi/n)} (-a\sin m(\alpha - \frac{\pi}2) + b\cos m(\alpha - \frac{\pi}2). \]
By Lemma \ref{evolsupp}, the evolute  $\E(h)$ has the equation$$p(\alpha) = m\left(-a\sin m\left(\alpha - \dfrac{\pi}2\right) + b\cos m\left(\alpha - \dfrac{\pi}2\right)\right).$$
\end{proof}
\begin{corollary} \label{scalingfactor}
$\Aeo^2(\P)$ arises from $\P$ through reversing the orientations of all sides and scaling by the factor $\dfrac{\sin^2(\pi m/n)}{\sin^2(\pi/n)}$.
\end{corollary}

\begin{proposition}
\label{PSPres_A}
The vertex centroids of an equiangular hedgehog and its $\Aeo$-evolute coincide.
\end{proposition}
\begin{proof}
Literally follows the proof of the Proposition \ref{PsStPres} with the obvious modification.
\end{proof}

Note the difference between the scaling factors in Theorems \ref{EvolDH} and \ref{EvolDHPerp} and the different behavior of the second evolute.

For $x>0$, we have $\sin mx < m\sin x$. Therefore the $\Aeo$-evolute of a discrete hypocycloid is ``smaller'' than the evolute of a smooth hypocycloid. At the same time, as the ratio $n/m$ (the number of sides to the order of hypocycloid) tends to infinity, the $\Aeo$-evolute tends to the smooth hypocycloid in the Hausdorff metric.

Figure \ref{5and9} shows the discrete astroids with five and nine sides and their $\Aeo$-evolutes. Compare this with Figure \ref{9and5}.

\begin{figure}[htbp]
\centering
\includegraphics[width=2.5in]{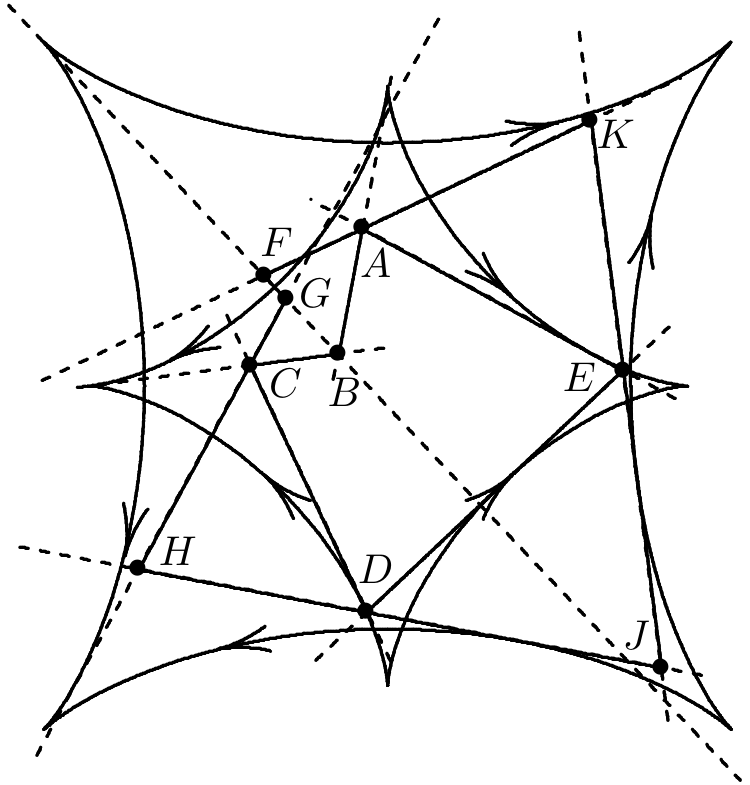}
\raise10pt\hbox{\includegraphics[width=2.5in]{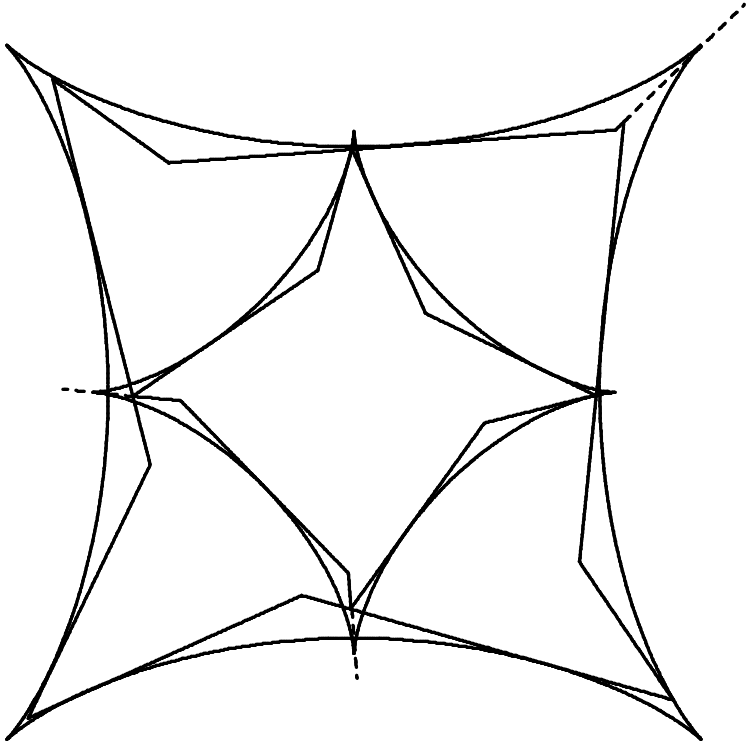}}
\caption{Discrete astroids and their $\Aeo$-evolutes.}
\label{5and9}
\end{figure}

The $\Aeo$-evolute of an equiangular hedgehog $\P$ degenerates to a point only if $\P$ is circumscribed about a circle. In this case, the discrete Fourier transform \eqref{DiscrFourier} contains no terms of order larger than $1$.

If \eqref{DiscrFourier}  contains higher order terms, then  $\Aeo^k(\P)$ expand to infinity: by Theorem \ref{EvolDH}, the $m$-th summand of $\Aeo^2(\P)$ equals that for $\P$, scaled by the factor $\dfrac{\sin^2(\pi m/n)}{\sin^2(\pi/n)}$. The highest discrete hypocycloids $\mathbf{C}_{n/2}(n)$ and $a \mathbf{C}_{(n-1)/2} + b \mathbf{S}_{(n-1)/2}$ are scaled by the largest factor. Therefore they determine the limiting shape of $\Aeo^k(\P)$.

The hedgehog $\mathbf{C}_{n/2}(n)$ is a regular $n$-gon with each side traversed twice in different directions. Its $\Aeo$-evolute is a regular $(n/2)$-gon rotated by $\frac{\pi}n$. The hedgehog $a \mathbf{C}_{(n-1)/2} + b \mathbf{S}_{(n-1)/2}$ has an irregular shape, and is not similar to its $\Aeo$-evolute, see Figure \ref{5and9}, left, for  $n=5$. Thus, for even $n$, the limiting shape is a regular $(n/2)$-gon, while for odd $n$, there are two limiting shapes which alternate.

Iterated $\Aeo$-evolutes of non-equiangular hedgehogs exhibit the same behavior.

\begin{theorem}
\label{AngleEvolShapes}
The iterated $\Aeo$-evolutes of discrete hedgehogs converge in the shape to discrete hypocycloids. Generically, these are the hypocycloids of the highest order: $n/2$ for even $n$, and $(n-1)/2$ for odd $n$.
\end{theorem}

\begin{proof}
Under the $\Aeo$-transformations, the sequence of turning angles is transformed by \eqref{fAngles}. This is a linear map with the eigenvalues
\[ \lambda_m = \frac{1+e^{\frac{2\pi i m}{n}}}2, \quad m = 0, 1, \ldots, n-1. \]
In particular, $\lambda_0 = 1$, and $|\lambda_m| < 1$ for $m > 0$. It follows that the turning angles of iterated evolutes converge to $\dfrac{2\pi k}{n}$, where $k$ is the turning number of the initial polygon. 
	
The turning angles and the support numbers are transformed according to
\[ (\theta^*, p^*) = (f(\theta), g_\theta(p)), \]
with $f$ given by \eqref{fAngles}, and
\[ g_\theta(p) = \mathcal{D}_\theta (\mathcal{Z}^\top - \id), \quad \mathcal{D}_\theta = \diag\left(\frac1{2\sin(\theta_1/2)}, \ldots, \frac1{2\sin(\theta_n/2)}\right), \]
where $p = (p_1, \ldots, p_n)$ and $p^* = (p^*_{1+\frac12}, \ldots, p^*_{n+\frac12})$. Since $f^k(\theta)$ converges to $\left( \dfrac{2\pi}n, \ldots, \dfrac{2\pi}n \right)$, the linear map $g_{f^k(\theta)}$ converges, in the operator norm, to the transformation
\[ \Aev = \frac{\mathcal{Z}^\top - \id}{2\sin(2\pi/n)}.\]
The eigenspaces of $\Aev$ correspond to the discrete hypocycloids, and those with the absolute largest eigenvalues, to the $\left\lfloor \dfrac{n}2\right\rfloor$-hypocycloids. This implies the theorem.
\end{proof}

\begin{remark}
{\rm 
When we apply the map $(f,g_\theta)$ twice,  $(\theta, p)$ goes to $(\theta^{**}, p^{**})$ with
\[ \theta^{**} = (\theta^{**}_{2+\frac12}, \ldots, \theta^{**}_{n+\frac12}, \theta^{**}_{1+\frac12}), \quad p^{**} = (p^{**}_2, \ldots, p^{**}_n, p^{**}_1). \]
Since the eigenspaces of the transformation $\Aev$ are invariant under the cyclic shift of indices, the conclusion of the Theorem doesn't change if we preserve the marking of the lines by shifting the indices of every other evolute.
}
\end{remark}
\subsection{$\Aec$-evolutes}\label{Ac-evo}

There exists another natural approach to angular bisector evolutes. Namely, we can always endow a polygon $\P=P_{\frac12}P_{\frac32}\dots P_{n-\frac12}$ with a {\it cyclic} orientation, that is, its $i$-th side is oriented from $P_{i-\frac12}$ to $P_{i+\frac12}$. To the oriented polygon obtained in this way, we apply the construction of Definition \ref{Aeo}, but with one essential modification: we endow the evolute not by the orientation provided by this definition, but again with the cyclic orientation. 

That these two orientations may be different can be derived from our example in Figure \ref{Aeo_vs_Aec}: $\mathbf Q$ is, actually, the $\Aec$-evolute of $P$, but the orientation of $\mathbf Q$ prescribed by Definition \ref{Aeo} is not cyclic (the sides $Q_1Q_2,Q_2Q_3$ and $Q_4Q_5$ are oriented according to the cyclic orientation, but the sides $Q_3Q_4$ and $Q_5Q_1$ are oriented in the opposite way).

\begin{figure}[hbtp]
	\centering
	\includegraphics[height=2.4in]{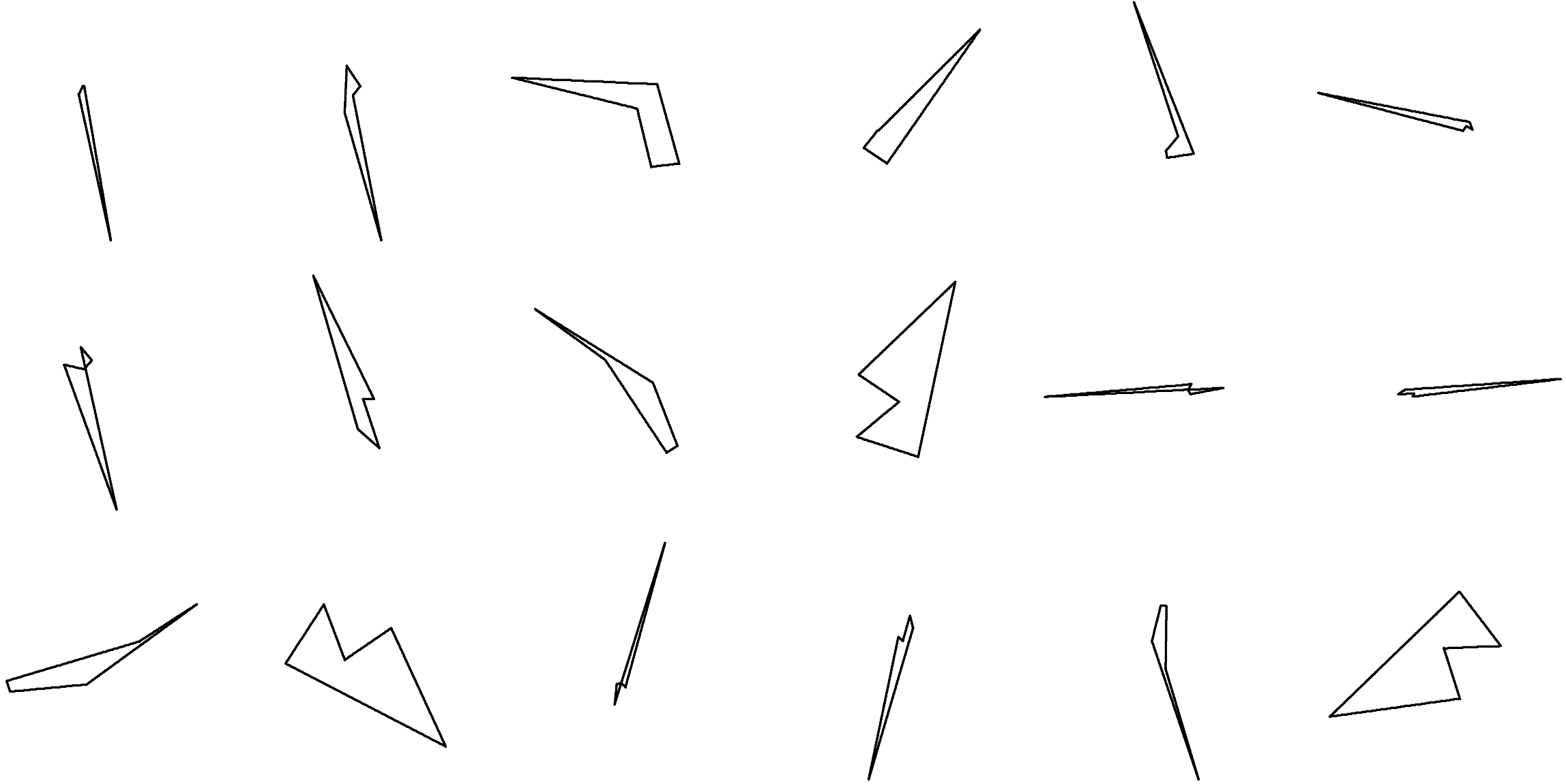} 
	\caption{The sequence of $\Aec$-evolutes of a pentagon.}
	\label{saw1}
\end{figure}

We cannot say much about $\Aec$-evolute, but we will mention some purely experimental observation concerning iterations of the $\Aec$-construction, applied to a randomly chosen pentagon. We modify the sequence of consecutive $\Aec$-evolutes precisely as we did for the sequence of consecutive $\Pev$-evolutes in Section \ref{limitbehavior}: to every term of our sequence, we first apply a translation which takes its centroid to the origin, and then a dilation which makes the maximal distance from the centroid to a vertex equal to 1.

The resulting sequence of pentagons, with some positive probability, possesses a surprising periodicity:  for a large $N$, the polygon number $N+4$ is obtained from the polygon number $N$ by a  rotation by  $3\pi/5$. Moreover, up to a rotation (by some angle) and a reflection in a line, these polygons do not depend on the initial random pentagon. Figures \ref{saw1} and \ref{saw2} demonstrate this phenomenon. 

\begin{figure}[hbtp]
	\centering
	\includegraphics[height=2.4in]{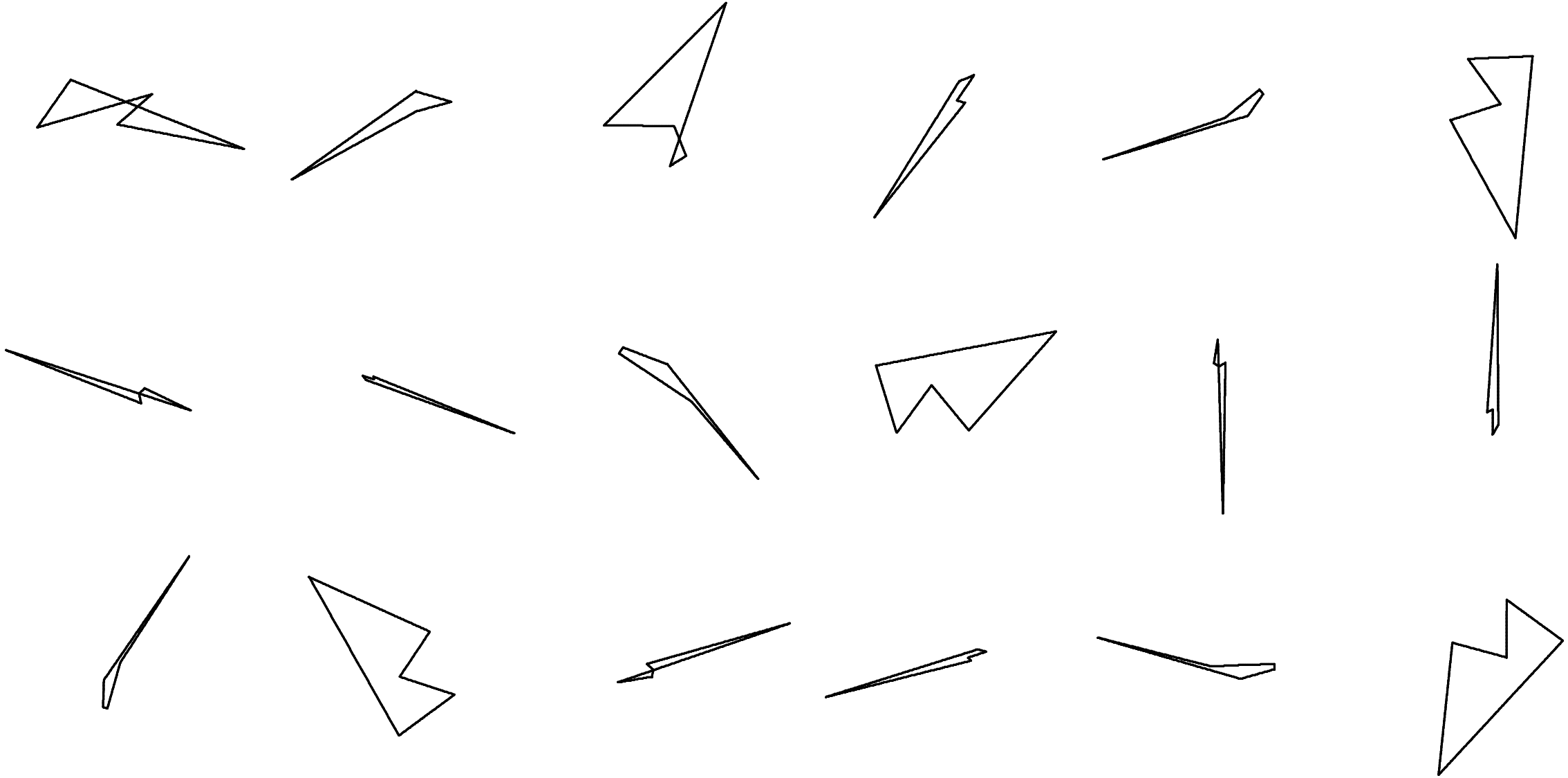} 
	\caption{One more sequence of $\Aec$-evolutes of a pentagon.}
	\label{saw2}
\end{figure}

In both sequences, starting from some place, a saw-shaped pentagon is repeated. This pentagon is shown separately in Figure \ref{saw}.  

Thus, the forth $\Aec$-evolute of the saw polygon is similar to the same polygon, rotated by $3\pi/5$. We were not able to detect any special properties of the angles and the side lengths of this polygon.

\begin{figure}[hbtp]
	\centering
	\includegraphics[height=1.2in]{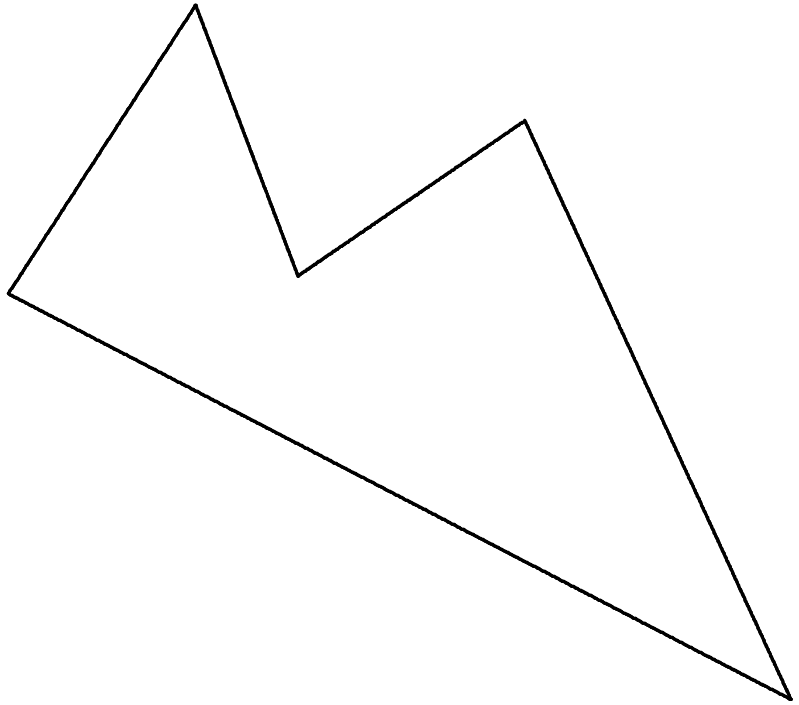} 
	\caption{The saw.}
	\label{saw}
\end{figure}

No similar property was detected in our experiments involving $n$-gons with $n>5$. (The cases $n<5$ are not interesting: the $\Aec$-evolute of a triangle is always one point, and for a quadrilateral, the limit shape of a multiple $\Aec$-evolute is always an interval traversed four times in two opposite directions, or a point.)

\section{Discrete involutes} \label{AP-inv}
\subsection{Preliminary considerations}\label{preliminary}
For a polygon $\mathbf Q$, we can consider $\Pev$-involutes and $\Aev$-involutes. Namely, $\P$ is a $\Pev$-involute of $\mathbf Q$ if $\mathbf Q$ is the $\Pev$-evolute of $\P$, and $\P$ is an $\Aev$-involute of $\mathbf Q$ if $\mathbf Q$ is an $\Aev$-evolute of $\P$ (one of the $\Aev$-evolutes of $\P$). 

Thus, if $\P$ is a $\Pev$-involute of $\mathbf Q$, then the sides of $\mathbf Q$ are perpendicular bisectors of the sides of $\P$. This means that the vertices $P_{i+\frac12},P_{i-\frac12}$ of $\P$ are symmetric to each other with respect to the side $l_i$ of $\mathbf Q$. Therefore the action of consecutive reflections in the sides $l_1,l_2,\dots,l_n$ of $\mathbf Q$ on $P_{\frac12}$ is as follows:$$P_{\frac12}\mapsto P_{\frac32}\mapsto P_{\frac52}\mapsto\dots\mapsto P_{n-\frac12}\mapsto P_{\frac12};$$in particular, $P_{\frac12}$ is fixed by the composition ${\mathcal S}={\mathcal S}_n\circ\dots\circ{\mathcal S}_1$, where ${\mathcal S}_i$ is the reflection in the line $l_i$. 

We obtain the following universal way of constructing $\Pev$-involutes: find a fixed point of the transformation $\mathcal S$ and reflect it, consecutively, in $l_1,\dots,l_n$. The $n$ points thus constructed  are the vertices of a $\Pev$-involute of $\mathbf Q$, and this construction gives all the $\Pev$-involutes of $\mathbf Q$. Thus, we obtain a bijection between the fixed points of $\mathcal S$ and the $\Pev$-involutes of $\mathbf Q$. 

Remark, in addition, that if we take not a fixed, but a periodic point of $\mathcal S$ of period $k$, and repeat our sequence of reflections $k$ times, then we will obtain a $kn$-gon whose $\Pev$-evolute is the polygon $\P$ traversed $k$ times. 

Similarly, if $\P$ is an $\Aev$-involute of $\mathbf Q$, then the side $l_i$ of $\mathbf Q$ is a (interior or exterior) bisector of the angle formed by the sides $l^\ast_{i-\frac12}$ and $l^\ast_{i+\frac12}$ of $\P$ (a small change of notation is made for the future convenience's sake). This means that the sides $l^\ast_{i+\frac12},l^\ast_{i-\frac12}$ are symmetric  with respect to the side $l_i$ of $\mathbf Q$ (this is true whether $l_i$ is an interior or exterior bisector of an angle of $\P$). Thus, the action of the reflections ${\mathcal S}_1,{\mathcal S}_2,\dots,{\mathcal S}_n$ on the line $l^\ast_{\frac12}$ is as follows: $$l^\ast_{\frac12}\mapsto l^\ast_{\frac32}\mapsto l^\ast_{\frac52}\mapsto\dots\mapsto l^\ast_{n-\frac12}\mapsto l^\ast_{\frac12};$$in particular, $l^\ast_{\frac12}$ is invariant with respect to the composition $\mathcal S$. 

We get a universal way of constructing  $\Aev$-involutes: find an invariant line of the transformation $\mathcal S$ and reflect it, consecutively, in $l_1,\dots,l_n$. The $n$ lines thus constructed are the sides of an $\Aev$-involute of $\mathbf Q$, and this construction gives all the $\Aev$-involutes of $\mathbf Q$. Thus, we obtain a bijection between the invariant lines of $\mathcal S$ and the $\Aev$-involutes of $\mathbf Q$. 

Similarly to the $\Pev$-case, an invariant line of ${\mathcal S}^k$ gives rise to a $kn$-gon which is an $\Aev$-involute of the polygon $\P$ traversed $k$ times. 

\begin{remark}
\label{rem:EvenOdd}
{\rm
If $l$ is an invariant line of $\mathcal S$, then $\mathcal S$ either preserves or reverses the orientation of $l$. How does this affect the corresponding $\Aev$-involute? 

Let $\P$ be an $\Aev$-involute of $\mathbf Q$, and let us equip the sides of $\P$ with the cyclic orientation. The sides of $\mathbf Q$ are either interior, or exterior, angle bisectors of $\P$. But the reflection in a bisector of the angle formed by consecutive sides $l,l'$ provides a map $l\to l'$ which preserves orientation if the bisector is exterior, and reverses the orientation if the bisector is interior. 

We arrive at the following conclusion: if the transformation $\mathcal S$ preserves the orientation of an invariant line, then the number of sides of $\mathbf Q$ which are the interior angle bisectors of the $\Aev$-involute corresponding to this line is even; if $\mathcal S$ reverses the orientation, then this number is odd. 

The same can be said about the invariant lines of the transformation ${\mathcal S}^k$.
}
\end{remark}

\subsection{Fixed points and invariant lines of compositions of reflections in lines}\label{fixed_points}

Let, as before, $l_1,\dots,l_n$ be $n$ lines in the plane, ${\mathcal S}_i$ the reflection in $l_i$, and $\mathcal S={\mathcal S}_n\circ\ldots\circ{\mathcal S}_1$. We will consider separately the cases of odd and even $n$.

\subsubsection{The case of odd $n$}\label{oddn} The transformation $\mathcal S$ is either a reflection, or a glide reflection. 

If $\mathcal S$ is a glide reflection, then it has no fixed or periodic points, the axis of $\mathcal S$ is a unique invariant line and all lines parallel to the axis (and no other lines) are invariant lines of ${\mathcal S}^2$. 

If $\mathcal S$ is a reflection, then the invariant lines are the axis of $\mathcal S$ and all lines perpendicular to the axis, and fixed points are points of the axis; since in this case ${\mathcal S}^2=\id$, all points are fixed points and all lines are invariant lines of ${\mathcal S}^2$. 

How to find out whether $\mathcal S$ is a reflection or a glide reflection?

We fix orientations for all lines $l_i$ (the reflections ${\mathcal S}_i$ will not depend on these orientations). Let the coordinates of the line $l_i$ be $(\alpha_i,p_i)$. We will use the following notation: the vector $(\cos\alpha,\sin\alpha)$ is denoted by $e_\alpha$ (so, $e_{\alpha+\pi}=-e_\alpha$), and the reflection in the line through the origin, with the coorienting vector $e_\alpha$, is denoted by $A_\alpha$ (so, $A_\alpha$ is a linear transformation). 

Then, obviously, $A_\alpha e_\beta=-e_{2\alpha-\beta}$. A direct computation shows that the reflection in the line with the coordinates $(\alpha,p)$ acts by the formula $$u\mapsto -A_\alpha u+2pe_\alpha.$$Now let us apply ${\mathcal S}^2$ to the origin 0. If the result is 0, then $\mathcal S$ is a reflection, if it is not 0, then it is a glide reflection. To make the computation more transparent, we will do it for $n=3$ and then formulate the result for an arbitrary odd $n$.

\[\begin{array} {rlll} 0\, \mathop{\longmapsto}\limits^{{\mathcal S}_1}&\phantom{-}2p_1e_{\alpha_1}\\ 
\mathop{\longmapsto}\limits^{{\mathcal S}_2}&-2p_1e_{2\alpha_2-\alpha_1}&+2p_2e_{\alpha_2}\\ 
\mathop{\longmapsto}\limits^{{\mathcal S}_3}&\phantom{-}2p_1e_{2\alpha_3-2\alpha_2+\alpha_1}&-2p_2e_{2\alpha_3-\alpha_2}&+2p_3e_{\alpha_3}\\ 
\mathop{\longmapsto}\limits^{{\mathcal S}_1}&-2p_1e_{-2\alpha_3+2\alpha_2+\alpha_1}&+2p_2e_{-2\alpha_3+\alpha_2+2\alpha_1}&-2p_3e_{-\alpha_3+2\alpha_1}\\
&+2p_1e_{\alpha_1}\\
\mathop{\longmapsto}\limits^{{\mathcal S}_2}&\phantom{-}2p_1e_{2\alpha_3-\alpha_1}&-2p_2e_{2\alpha_3+\alpha_2-2\alpha_1}&-2p_3e_{\alpha_3+2\alpha_2-2\alpha_1}\\ 
&-2p_1e_{2\alpha_2-\alpha_1}&+2p_2e_{\alpha_2}\\
\mathop{\longmapsto}\limits^{{\mathcal S}_3}&-2p_1e_{\alpha_1}&+2p_2e_{-\alpha_2+2\alpha_1}&-2p_3e_{\alpha_3-2\alpha_2+2\alpha_1}\\ 
&\phantom{-}2p_1e_{2\alpha_3-2\alpha_2+\alpha_1}&-2p_2e_{2\alpha_3-\alpha_2}&+2p_3e_{\alpha_3}
\end{array}\]

Obviously, $e_\beta+e_\gamma=2\cos\displaystyle\frac{\gamma-\beta}2e_{\frac{\beta+\gamma}2}$, from which$$e_\beta-e_\gamma=e_\beta+e_{\gamma+\pi}=2\cos\frac{\gamma+\pi-\beta}2e_{\scriptstyle{\frac\pi2+\frac{\beta+\gamma}2}}=2\sin\frac{\beta-\gamma}2e_{\scriptstyle{\frac\pi2+\frac{\beta+\gamma}2}}.$$Using this formula, we can reduce the last result to the following convenient form:$${\mathcal S}^2(0)=4(p_1\sin(\alpha_3-\alpha_2)+p_2\sin(\alpha_1-\alpha_3)+p_3\sin(\alpha_1-\alpha_2))e_{\frac\pi2+\alpha_1-\alpha_2+\alpha_3}.$$A similar computation for an arbitrary odd $n$ yields a similar result which we can write, using the notations $\B$ and $\B_j$ (see Section \ref{DefP-evo}) in the following, final, form:
$$
{\mathcal S}^2(0)=4 e_{\frac\pi2+\B(\alpha)} \sum_{j=1}^np_j\sin\B_j(\alpha).
$$

\begin{definition}\label{quasiperimeter}
The sum $\sum_{j=1}^np_j\sin\B_j(\alpha)$ is called the \emph{quasiperimeter} of a cooriented polygon $\{l_1,\dots,l_n\}$ (it does not depend on the coorientation).
\end{definition}
We arrive at the following result. 
\begin{proposition}
The composition of the reflections in the sides of an odd-gon is a reflection if and only if its quasiperimeter is zero. 
\end{proposition}

\begin{remark}
{\rm 
For equiangular polygons, the quasiperimeter is proportional to the perimeter. But, in general, it is a non-local quantity (as opposed to the perimeter which is the sum of the side lengths): all the cosines in the definition of the quasiperimeter are those of the angles which depend on all the angles $\alpha_j$.
}
\end{remark}

Notice, in conclusion, that whether the composition $\mathcal S$ of reflections in the lines $l_i$ is a reflection or a glide reflection, the direction of the axis of this (glide) reflection is determined by the directions of the lines $l_i$: if the angular parameters of $l_1,\dots,l_n$ are $\alpha_1,\dots,\alpha_n$, then the angular parameter of the axis is $\B(\alpha)$. This has an obvious geometric proof, but, in the case of glide reflection, follows from the last formula, and the case of reflection is reduced to the case of glide reflection by continuity. 

\subsubsection{The case of even $n$.}\label{evenn}
As before, we have $n$ oriented lines $l_i=(\alpha_i,p_i)$, and as before, the final results do not depend on the orientations; the only difference with the previous case is that $n$ is even. In this case, the transformation $\mathcal S$ is either a rotation (by some non-zero angle, about some point), or a (non-trivial) translation, or the identity. 

A rotation has just one fixed point, the center of rotation, and, except the case of the rotation by $\pi$, no invariant lines. If the angle of rotation is $\pi$ (this rotation is the reflection in a point), then all lines passing through this point are invariant, and the rotation reverses their orientations. If the angle of rotation is a rational multiple of $\pi$, then all point and all lines are periodic. 

A translation has no fixed or periodic points, and all lines parallel to the direction of the translation are invariant, their orientations are preserved by the translation. 
No comments for the identity.

The transformation $\mathcal S$ acts by the formula
$$
{\mathcal S}u=A_{\alpha_n}A_{\alpha_{n-1}}\dots A_{\alpha_1}u+\sum_{k=1}^n(-1)^kA_{\alpha_n}\dots A_{\alpha_{k+1}}e_{\alpha_k}.
$$
But the composition of two reflections is a rotation, $A_\beta A_\gamma=R_{2(\gamma-\beta)}$, and, as we noticed earlier, $A_\beta e_\gamma=-e_{2\beta-\gamma}$. Hence,
$$
{\mathcal S}u=R_{2{\mathcal B}(\alpha)}u+\sum_{k=1}^n2p_ke_{(-1)^k\alpha_k+2\sum_{j=k+1}^n(-1)^j\alpha_j}.
$$
Using the notation
$$
w=\sum_{k=1}^n2p_ke_{(-1)^k\alpha_k+2\sum_{j=k+1}^n(-1)^j\alpha_j},
$$
we can formulate the result.

\begin{proposition} {\rm (i)} If $\B(\alpha)$ is not a multiple of $\pi$, then $\mathcal S$ is a rotation by the angle $2\B(\alpha)$ about the point 
$$
\frac{w-R_{-2\B(\alpha)}w}{4 \sin^2\B(\alpha)}.
$$ 

{\rm(ii)} If $\B(\alpha)$ is a multiple of $\pi$ and $w\ne0$, then $\mathcal S$ is a translation by $w$.

{\rm(iii)} If $\B(\alpha)$ is a multiple of $\pi$ and $w=0$, then $\mathcal S=\id$.
\end{proposition}

\subsection{$\Aev$-involutes}\label{Ainv}

\subsubsection{The case of odd $n$}\label{Ainvodd}

Let $\mathbf Q$ be some $n$-gon with  odd $n$. Label the sides of $\mathbf Q$ cyclically as $l_1,\dots,l_n$. The composition $\mathcal S$ of the consecutive reflections ${\mathcal S}_i$  in these sides is a reflection or a glide reflection, and in either of these cases, $\mathcal S$ possesses a unique invariant line whose orientation is preserved by $\mathcal S$: the axis of the reflection or the glide reflection. 

\begin{figure}[htbp]
\centering
\includegraphics[width=4in]{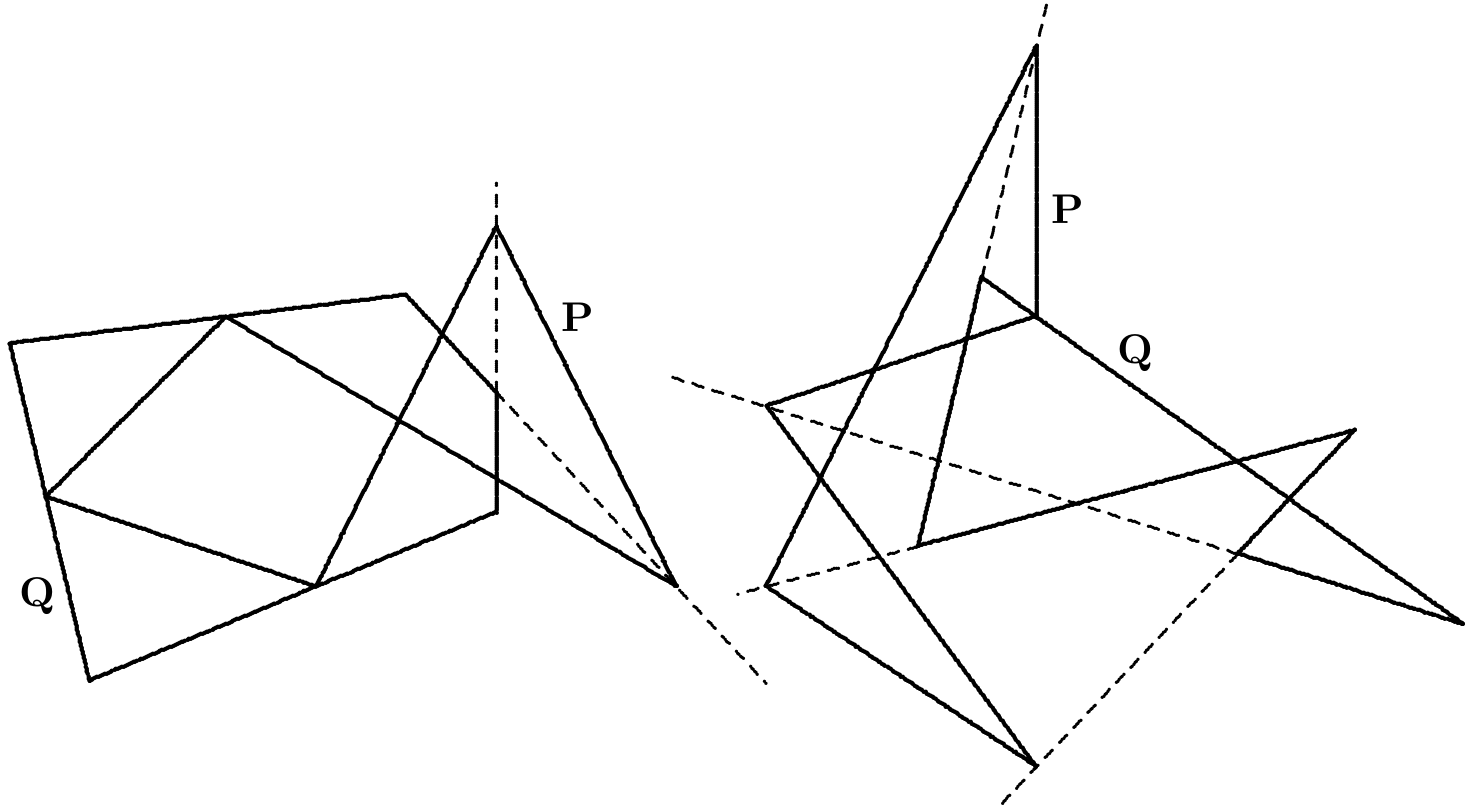}
\caption{$\Aev_{\rm odd}$-evolvents}
\label{anglevolodd}
\end{figure}

According to Section \ref{preliminary}, the consecutive reflections of this line in $l_1,\dots,l_n$ give an involute of $\mathbf Q$. This involute $\P$ does not change under the action of the dihedral group on the labels of the sides (if we replace $l_1,\dots,l_n$ by $l_n,\dots,l_1$, then $\mathcal S$ will become ${\mathcal S}^{-1}$, but the axis will stay unchanged). We call $\P$ the {\it $\Aev_{\rm odd}$-evolvent} of $\mathbf Q$. 

This case is illustrated in Figure \ref{anglevolodd}. The polygon $\P$ is the $\Aev_{\rm odd}$-evolvent of the polygon $\mathbf Q$.  The sides of $\mathbf Q$ are the bisectors of (interior or exterior) angles of $\P$, and, according to Remark \ref{rem:EvenOdd}, the number of exterior angles is odd. This number is 3 in Figure \ref{anglevolodd}, left and 1 in Figure \ref{anglevolodd}, right.

\begin{remark} \label{bill}
{\rm
One may consider the $\Aev_{\rm odd}$-evolvent $\P$ of an $n$-gon $\mathbf Q$ as a generalized $n$-periodic billiard trajectory in $\mathbf Q$: the consecutive pairs of sides of $\P$ make equal angles with the respective sides of $\mathbf Q$. For example, if $\mathbf Q$ is an acute triangle, then $\P$ is the so-called Fagnano 3-periodic billiard trajectory connecting the foot points of the altitudes of $\mathbf Q$. However, if $\mathbf Q$ is an obtuse triangle, then $\P$ does not lie inside $\mathbf Q$ anymore, see Figure \ref{pedal}.
}
\end{remark}

\begin{figure}[hbtp]
\centering
\includegraphics[width=3.8in]{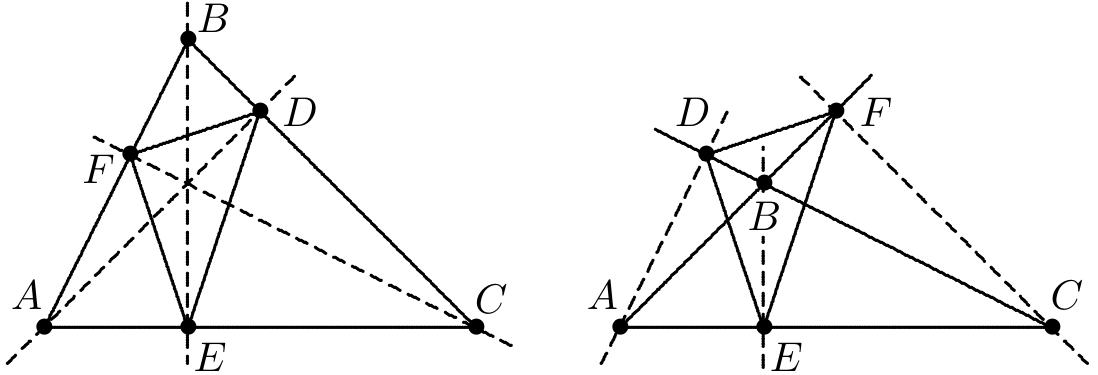} 
\caption{Triangle $DEF$ is the $\Aev_{\rm odd}$-evolvent of the triangle $ABC$.}
\label{pedal}
\end{figure}

If the quasiperimeter of $\mathbf Q$ is zero, then $\mathcal S$ is a reflection, and all lines perpendicular to the axis of this reflection are invariant with respect to $\mathcal S$; in this case, $\mathcal S$ reverses the orientation of all these lines.

This case is illustrated in Figure \ref{anglevoleven}. The polygon $\mathbf Q$ has  zero quasi\-perimeter. It has a family of parallel $\Aev$-involutes. The  $\alpha$-coordinates of the sides of these involutes are the same for all involutes in our family, and the  $p$-coordinates vary with a constant speed,  the same for all sides. This shows that the quasiperimeters of our involutes vary at a constant speed, and generically this speed is not zero. For this reason, for a generic odd-gon with  zero quasiperimeter, precisely one of these parallel involutes has a zero quasiperimeter. We can call this involute the $\Aev_{\rm even}$-evolvent. Figure \ref{anglevoleven} shows several such involutes for two polygons $\mathbf Q$. One of them is the $\Aev_{\rm even}$-evolvent; it is denoted by $\P$. 

\begin{figure}[htbp]
\centering
\includegraphics[width=4.1in]{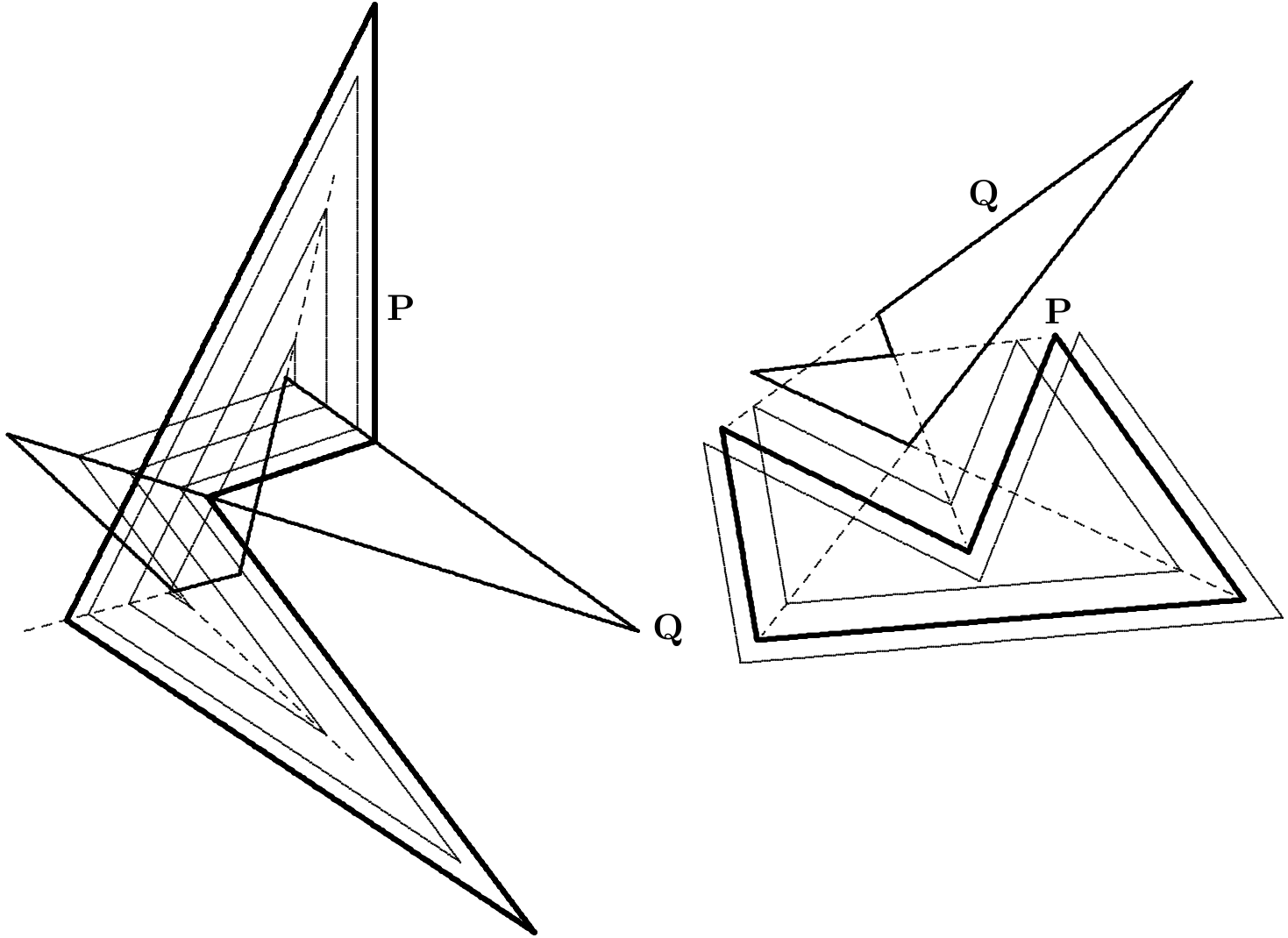}
\caption{Families of parallel $\Aev$-involutes of an odd-gon with zero quasiperimeter. The $\Aev$-involute marked as $\P$ has a zero quasiperimeter; it is the $\Aev_{\rm even}$-evolvent.}
\label{anglevoleven}
\end{figure}

The sides of $\mathbf Q$ are the bisectors of interior or exterior angles of all involutes, and the number of exterior angles is even. This number is 0 in Figure \ref{anglevoleven}, left, and is 2 in Figure \ref{anglevoleven}, right.

\subsubsection{The case of even $n$}\label{Ainveven}

For an even-gon $\mathbf Q$, the transformation $\mathcal S$ preserves  orientation, so it may be a rotation, a parallel translation, or the identity. If the angular coordinates of the sides of $\mathbf Q$ are $\alpha_1,\alpha_2,\dots,\alpha_n$, then the angle of rotation is 
$$\beta=2\B(\alpha_1,\dots,\alpha_n)= 2(\alpha_1-\alpha_2+\dots-\alpha_n);$$
 if $\beta\bmod2\pi$ is not 0 or $\pi$, then $\mathcal S$ has no invariant lines; so, in this case (which is generic), $\mathbf Q$ has no $\Aev$-involutes at all. (Still, if $\beta$ is a non-zero rational multiple of $\pi$, then, for some $k$, ${\mathcal S}^k=\id$, and every line is invariant with respect to ${\mathcal S}^k$. In this case, there are infinitely many polygons whose $\Aev$-evolute is $\mathbf Q$, traversed $k$ times.) 
 
 If $\beta\equiv\pi\bmod2\pi$, then $\mathcal S$ is a reflection in some point. In this case, every line through this point is $\mathcal S$-invariant, which creates a one-parameter family of $\Aev$-involutes (and every line in ${\mathcal S}^2$-invariant).

All these cases are shown in Figures \ref{hexaginvolone} -- \ref{hexaginvolthree} below.

\begin{figure}[hbtp]
\centering
\includegraphics[width=3in]{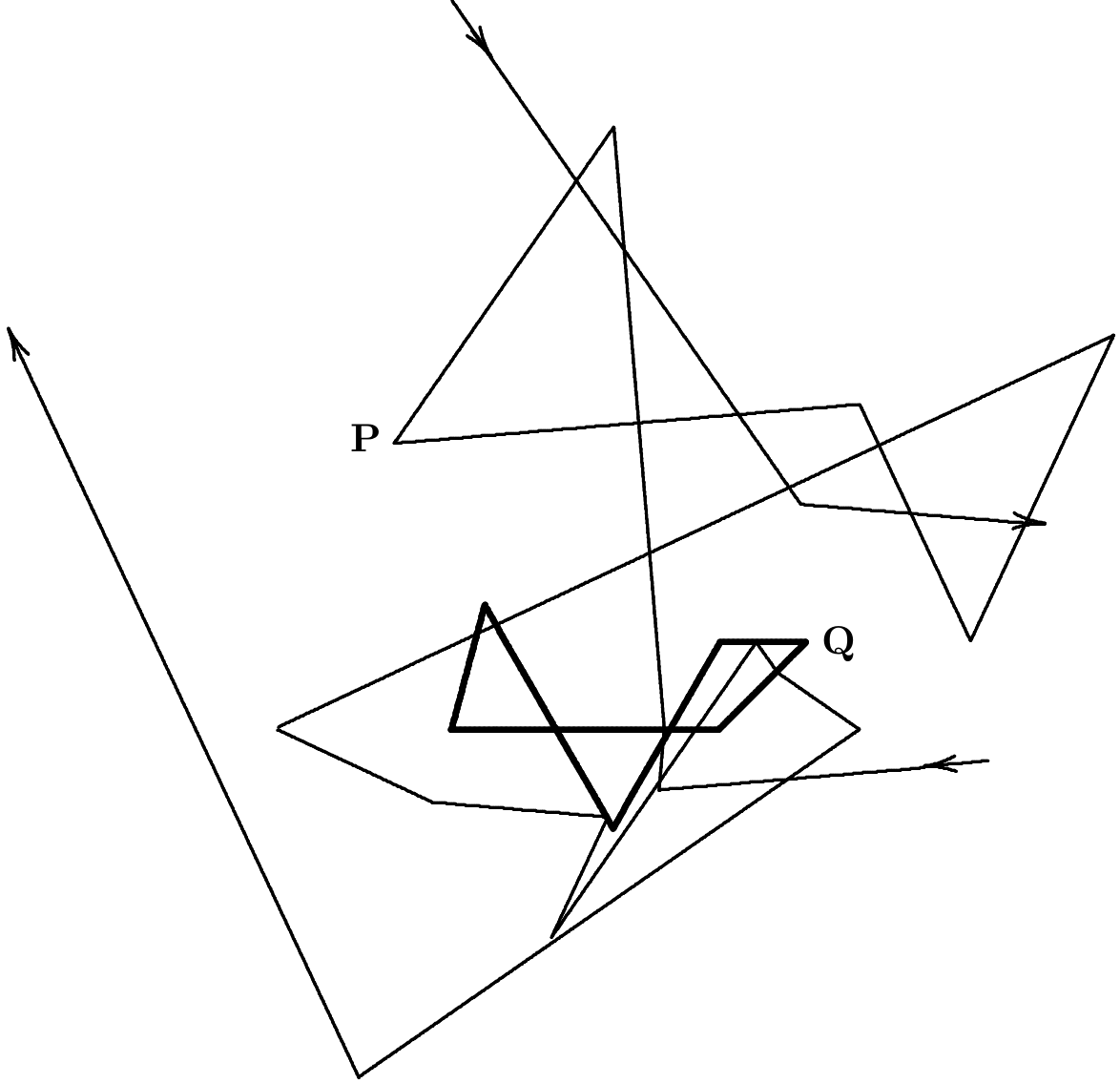}
\caption{An 18-gonal $\Aev$-involute $\P$ of the hexagon $\mathbf Q$ traversed thrice.}
\label{hexaginvolone}
\end{figure}

For the hexagon $\mathbf Q$ shown in Figure \ref{hexaginvolone}, the transformation $\mathcal S$ is a rotation by the angle $4\pi/3$. As it was noted above, this hexagon has no $\Aev$-involutes; however, there is a 2-parameter family of 18-gonal $\Aev$-involutes of the hexagon $\mathbf Q$ traversed thrice; one of them is shown in Figure \ref{hexaginvolone} (it is too large to be fully visible: two of its 18 vertices are outside of the figure). We will consider this polygon $\mathbf Q$ again in Section \ref{Pinveven}: see Figure \ref{hexagPinvol}.

For the hexagon $\mathbf Q$ shown in Figure \ref{hexaginvoltwo}, the transformation $\mathcal S$ is a rotation by the angle $\pi$. (There exists a very simple way to construct such a hexagon: the sum of the first, third, and fifth angles should be $\dfrac{\pi}2\bmod\pi$; for example, all  three of them may be right, and this is the case for the hexagon in Figure \ref{hexaginvoltwo}.) As it was noted above, such a polygon $\mathbf Q$ has a one-parameter family of $\Aev$-involutes. The vertices of these involutes slide along the sides of $\mathbf Q$ which are angle bisectors for all the involutes. There are $n$ points in the plane which lie on the sides of all involutes; one of these points is marked in Figure \ref{hexaginvoltwo}.

\begin{figure}[hbtp]
\centering
\includegraphics[width=3.7in]{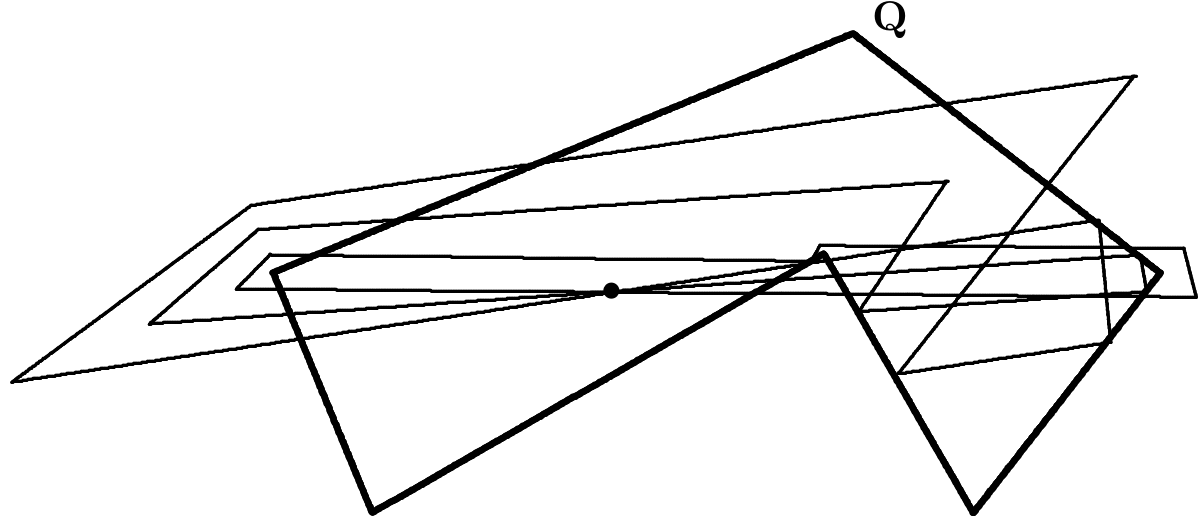}
\caption{A family of $\Aev$-involutes of a hexagon $\mathbf Q$ for which the transformation $\mathcal S$ is a reflection in a point.}
\label{hexaginvoltwo}
\end{figure}

Figure \ref{hexaginvolthree} shows two hexagons for which the transformation $\mathcal S$ is, respectively, a translation and the identity; they possess, respectively, a one-parameter and a two parameter family of $\Aev$-involutes.

\begin{figure}[hbtp]
\centering
\includegraphics[width=3.5in]{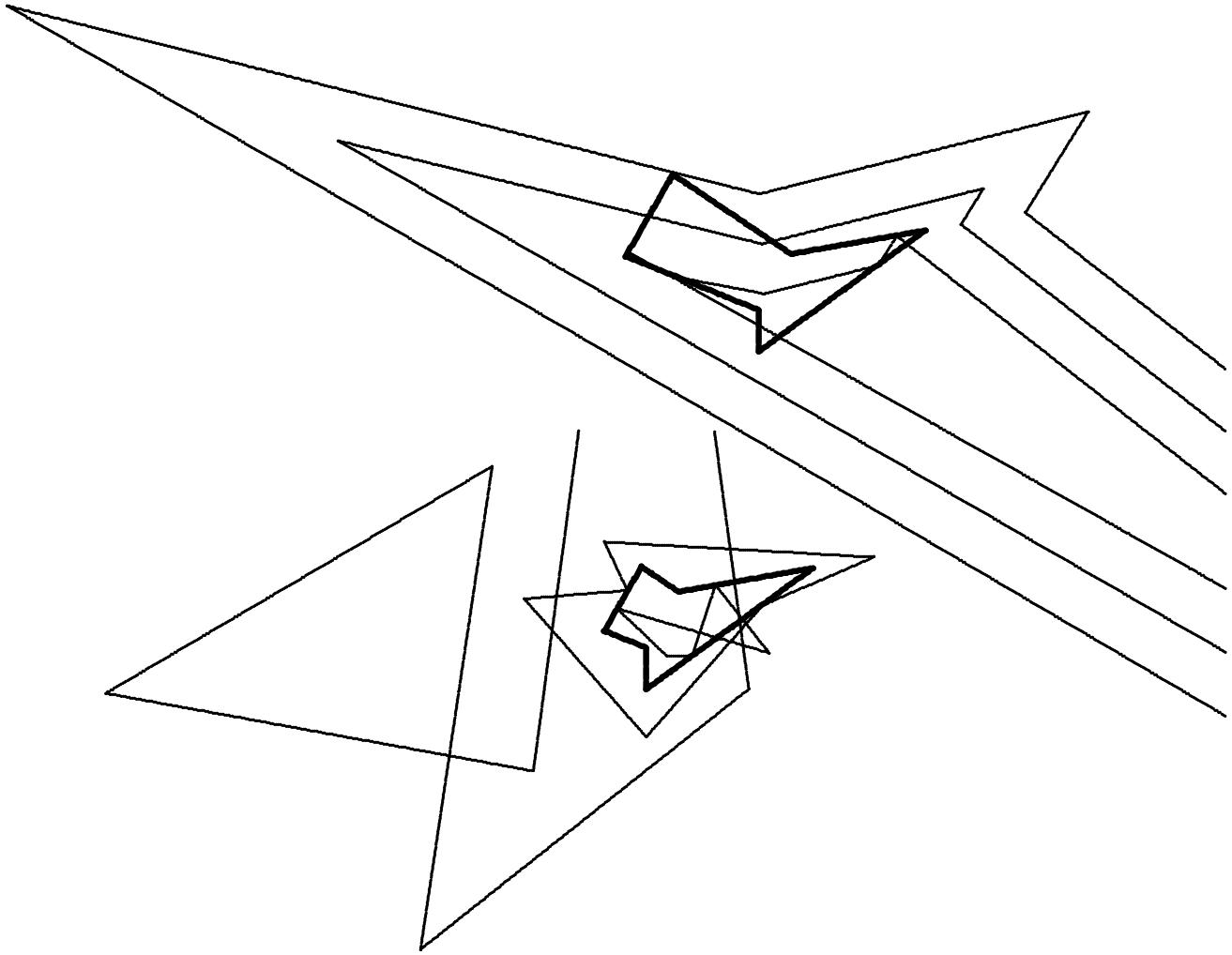}
\caption{Upper: a hexagon with a family of parallel $\Aev$-involutes. Lower: a hexagon with a two-parameter family of $\Aev$-involutes.}
\label{hexaginvolthree}
\end{figure}

\subsubsection{Limiting behavior of $\Aev$-evolvents for equiangular hedgehogs}

In the case of equiangular hedgehogs most of our constructions are simplified, and one observes an almost  straightforward similarity to the smooth case. 

\begin{proposition}
\label{WhenInvol}
For an equiangular hedgehog $\mathbf Q$, the following three conditions are equivalent:

{\rm (i)} $\mathbf Q$ posesses an $\Aeo$-involute;

{\rm (ii)}  the coefficient $a_0$ in the Fourier decomposition \eqref{DiscrFourier} vanishes;

{\rm (iii)} $\mathbf Q$ has zero perimeter.
\end{proposition}
\begin{proof}
The $\Aeo$-evolute of $\mathbf{C}_0$ is $0$, that is, a set of lines through the origin; the $\Aeo$-evolutes of higher discrete hypocycloids are discrete hypocycloids of the same order, according to Theorem \ref{EvolDH}. Therefore, if $\P$ is the $\Aeo$-evolute of some polygon, then the $a_0$ coefficient of $\P$ vanishes. By Theorem \ref{EvolDH} and Corollary \ref{scalingfactor}, every equiangular hedgehog with vanishing $a_0$ is the $\Aeo^2$-evolute of some other hedgehog; in particular it has an $\Aeo$-involute. Thus the first two conditions are equivalent.
	
The second condition is easily seen to be equivalent to $\sum_{j=1}^n p_j = 0$. Formula (\ref{lengths}) from Lemma \ref{CoordPolygon} implies that this is equivalent to $\sum_{j=1}^n \ell_j = 0$.
\end{proof}

By Proposition \ref{WhenInvol}, every equiangular hedgehog of zero perimeter generates an infinite sequence of iterated $\Aev$-evolvents: at every step one chooses the involute that is free of the $\mathbf{C}_0$ term.

\begin{theorem}\label{AInvolEqui}
The iterated $\Aev$-evolvents of an equiangular hedgehog with zero perimeter converge to its vertex centroid.
\end{theorem}
\begin{proof}
By Theorem \ref{EvolDH} and Corollary \ref{scalingfactor}, the iterated $\Aev$-evolvents of $a_m \mathbf{C}_m + b_m \mathbf{S}_m$ for $m > 1$ are inscribed into hypocycloids centered at the origin and shrinking with each step by a constant factor. Thus the distances of the sides of the hedgehog from its vertex centroid decrease exponentially, and all vertices converge to the vertex centroid.
\end{proof}

Note that the smallest non-vanishing harmonics dominate the sequence of iterated $\Aev$-evolvents. Therefore this sequence has a limiting shape or, more often, two shapes between which the members of the sequence alternate.

\subsubsection{Evolution of the angles of $\Aev_{\rm odd}$-evolvent, $n$ odd}

Consider Figure \ref{pedal} again. Given a triangle $ABC$, the triangle $DEF$ made by the foot points of its altitudes is called the {\it pedal triangle}.  As we mentioned in Remark \ref{bill}, the pedal triangle is the $\Aev_{\rm odd}$-evolvent of the initial triangle. 

A number of authors studied the dynamics of the map that sends a triangle to its pedal triangle \cite{KS,Lax,Un,Al,Ma}. The iterated pedal triangles converge to a point, and this point depends continuously, but not differentiably, on the initial triangle. Restricted to the shapes of triangles, the pedal map is ergodic  (it is modeled by the full shift on four symbols).

In this section, we extend this ergodicity result to odd $n>3$. Since we are interested in the angles only, we consider $n$ lines $l_1,\ldots,l_n$ through the origin. As before, let $\mathcal S$ be the composition of the reflections in $l_1,\ldots,l_n$. Since $n$ is odd, $\mathcal S$ is a reflection; let $l^*_1$ be its axis. Let $l^*_2,\ldots,l^*_n$ be the consecutive reflections of $l^*_1$ in $l_1, l_2,\ldots, l_{n-1}$. We are interested in the map
$$
F: (l_1,\ldots,l_n) \mapsto (l^*_1,\ldots,l^*_n).
$$

The map $F$ can be considered as a self-map of the torus $\R^n/(\pi\Z)^n$. This map commutes with the diagonal action of the circle $\R/\pi\Z$ (rotating all lines by the same angle). Let $G$ be the induced map on the quotient torus $\R^{n-1}/(\pi\Z)^{n-1}$. This map describes the evolution of the angles of $\Aev_{\rm odd}$-evolvents of an $n$-gon.

\begin{theorem} \label{ergodic}
The map $G$ is measure-preserving and ergodic.
\end{theorem}

\proof
Let $\alpha_1,\ldots,\alpha_n$ be the directions of the lines $l_1,\ldots,l_n$; the angles $\alpha_i$ are considered mod $\pi$. One can easily calculate the directions of the lines $l^*_1,\ldots,l^*_n$; these directions are
\begin{equation} \label{cyclang}
\alpha^*_{i}=\alpha_i-\alpha_{i+1}+\alpha_{i+2}-\ldots + \alpha_{n+i-1}
\end{equation}
(the indices are considered mod $n$). 

Thus the torus map $F$ is covered by the linear map (\ref{cyclang}). This linear torus map is measure-preserving: 
the preimage of every subset consists of $2^{n}$ copies of this subset, each having the volume $2^{n}$ times smaller.

A linear epimorphism of a torus is ergodic if and only if it is generated by a matrix that has no roots of unity as eigenvalues, see, e. g., \cite[Corollary 1.10.1]{Wal}. Let us find the eigenvalues of the map (\ref{cyclang}).

The matrix of (\ref{cyclang}) is circulant, and its eigenvalues are 
$$
\lambda_k = 1 - \omega^k + \omega^{2k} - \ldots + \omega^{(n-1)k},\ k=0,1\ldots,n-1,
$$
where $\omega$ is a primitive $n$th root of 1. That is, $\lambda_k = 2/(1+\omega^k)$.

In particular,  $\lambda_0 = 1$, and the respective eigenspace is spanned by the vector $(1,\ldots,1)$. The eigenvalues of the quotient map, $G$, are $\lambda_k,\ k=1\ldots,n-1$, and we have $|\lambda_k| > 1$. Therefore $G$ is ergodic.
\proofend

\subsection{$\Pev$-involutes}\label{P-inv}
\subsubsection{The case of odd $n$} \label{Pinvodd}
As before, the transfomation $\mathcal S$  is either a reflection in a line or a glide reflection. The latter has neither fixed, nor periodic points, and the former has a line consisting of fixed points. Thus a necessary and sufficient condition for $\mathbf Q$ to have a $\Pev$-involute is that the isometry $\mathcal S$ be a reflection (and not a glide reflection), in other words, that the quasiperimeter of $\mathbf Q$ be zero.

 In this case, $\mathbf Q$ has a one-parameter family of (parallel) $\Pev$-involutes; to construct any one of them, we choose a point on the axis of reflection and reflect it in the consecutive sides of $\mathbf Q$; the points obtained are the vertices of a $\Pev$-involute of $\mathbf Q$. 
 
 The quasiperimeters of the involutes of this family vary linearly. So a generic $n$-gon $\mathbf Q$ of zero quasiperimeter has precisely one $\Pev$-involute of zero quasiperimeter (compare with Section \ref{Ainvodd}). We call it the {\it $\Pev$-evolvent} of $\mathbf Q$. So a generic odd-gon of zero quasiperimeter gives rise to an infinite sequence of consecutive $\Pev$-evolvents.

If $\mathcal S$ is a reflection, then every point not on the axis is  periodic  of period 2. If we take such a point and repeat the full cycle of reflections in the sides of $\mathbf Q$ twice, we obtain a $2n$-gonal $\Pev$-involute of the polygon $\mathbf Q$, traversed twice; thus, there exists a two-parameter family of such involutes.

Figure \ref{Pinvolodd} shows all these constructions applied to a pentagon $\mathbf Q$ with zero quasiperimeter. The left diagram shows a family of $\Pev$-involutes; one of them, $\P$, is the $\Pev$-evolvent of $\mathbf Q$. The right diagram shows a 10-gonal $\Pev$-involute $\widetilde\P$ of the pentagon $\mathbf Q$, traversed twice.

\begin{figure}[hbtp]
\centering
\includegraphics[height=2.5in]{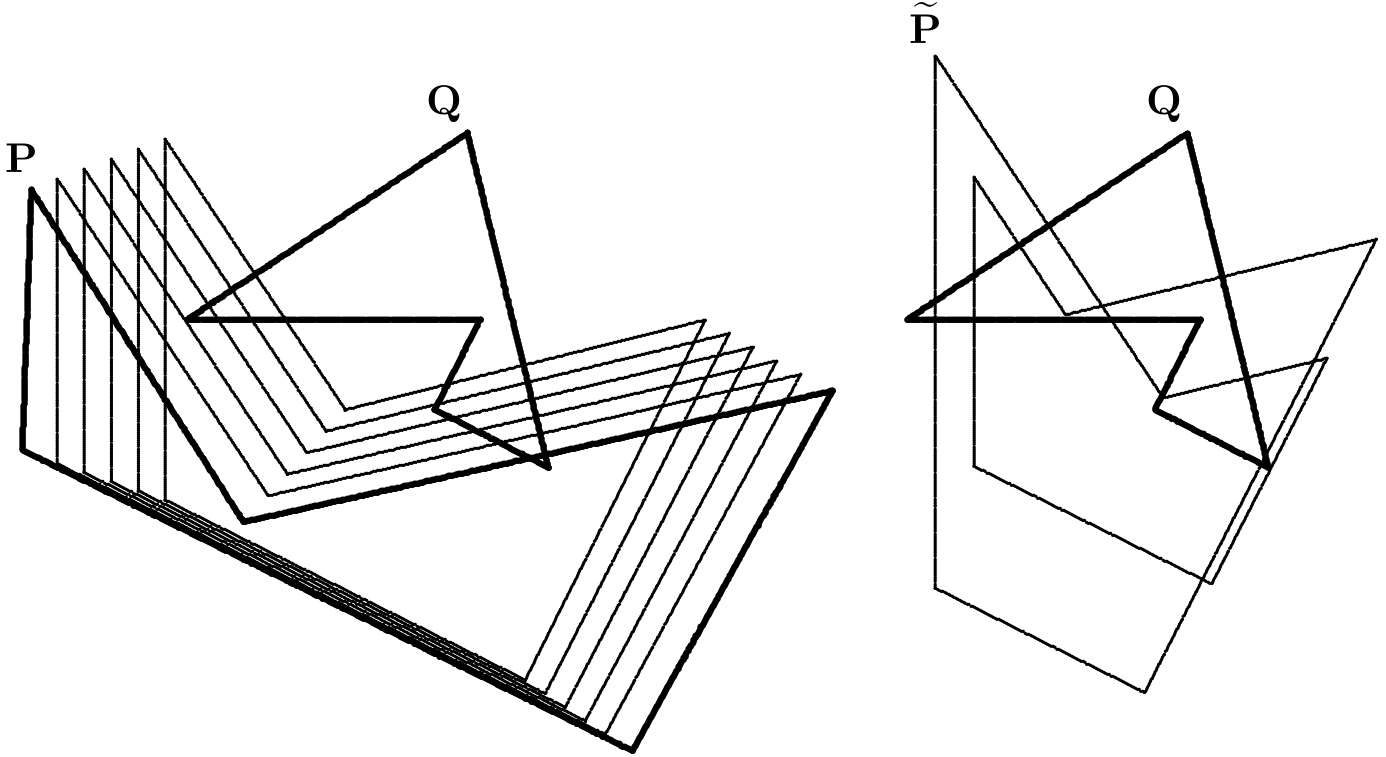}
\caption{Left: a family of $\Pev$-involutes of a pentagon $\mathbf Q$; $\P$ is the $\Pev$-evolvent. Right: a $\Pev$-involute $\widetilde\P$ of the same pentagon traversed twice.}
\label{Pinvolodd}
\end{figure}

\subsubsection{The case of even $n$}\label{Pinveven}
 The transformation $\mathcal S$ is a composition of an even number of reflections. Thus, it is either a (nontrivial) rotation, or a (nontrivial) parallel translation, or the identity. In the first case, one has a unique fixed point corresponding to the center of rotation, in the second case there are no fixed points, and in the third case, all the points are fixed. 

\begin{figure}[hbtp]
\centering
\includegraphics[width=5in]{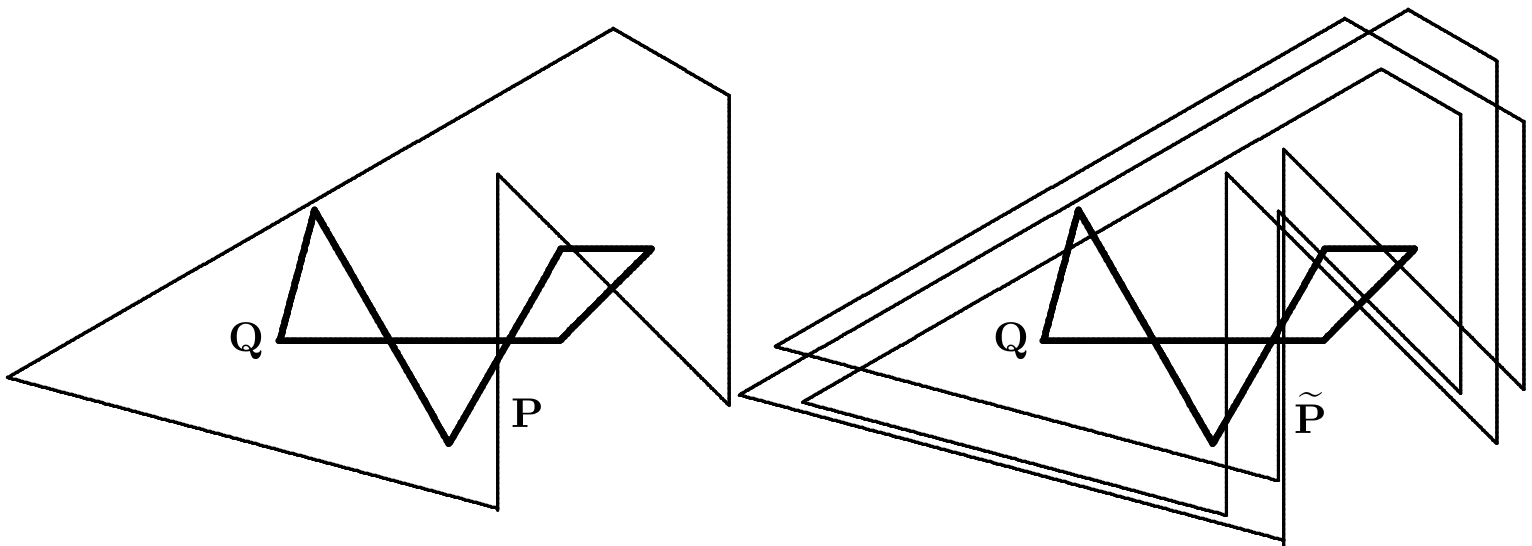} 
\caption{Left: a hexagon $\mathbf Q$ with the (hexagonal) $\Pev$-evolvent $\P$. Right: a $\Pev$-involute $\widetilde\P$ of the same hexagon traversed 3 times.}
\label{hexagPinvol}
\end{figure}

For a generic polygon $\mathbf Q$, the transformation $\mathcal S$ is a rotation, and there is a unique $\Pev$-evolvent $\P$. To construct it, we should take the center of rotation and consecutively reflect it in the sides of $\mathbf Q$. If the angle of this rotation is a rational multiple of $\pi$, then all the points of the plane, except the center of rotation, are periodic points of $\mathcal S$ of some period $k$. In this case we apply the same construction to an arbitrary point in the plane, repeating it $k$ times. We will get a closed $kn$-gon $\widetilde\P$ whose $\Pev$-evolute is the polygon $\mathbf Q$, traversed $k$ times. For this $kn$-gon, the vertices number $j,j+n,\dots,j+(k-1)n$ with $j=1,2,\dots,k$, form a regular $k$-gon whose size is the same for all $j$. These possibilities are shown in Figure \ref{hexagPinvol}.

In the case of parallel translation, no involute exists. For example, this happens if $\mathbf{Q}$ is a quadrilateral, inscribed into a circle (the opposite angles add up to $\pi$). If $\mathcal S$ is the identity, one has a 2-parameter family of $\Pev$-involutes that still have parallel sides. See Figure \ref{hexa}.

\begin{figure}[hbtp]
\centering
\includegraphics[height=2.45in]{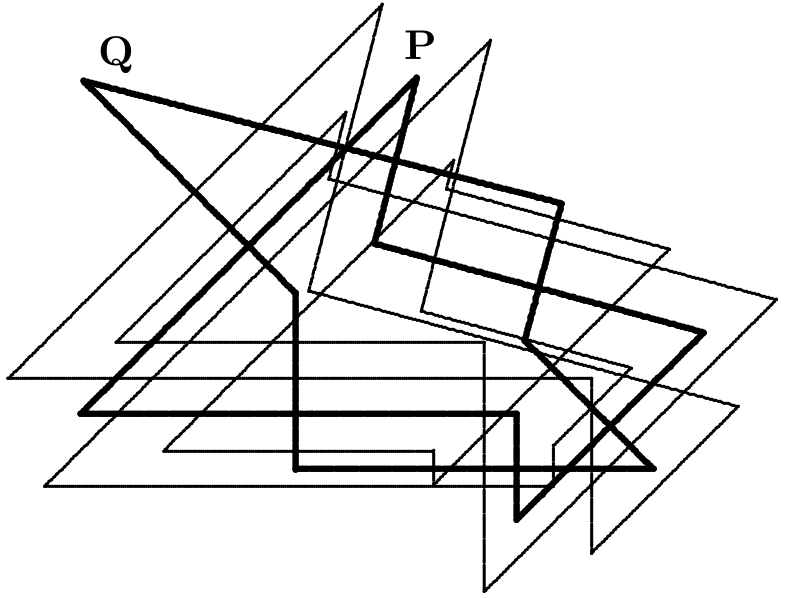} 
\caption{A hexagon $\mathbf Q$ with a two-parameter family of $\Pev$-involutes. One of them, $\P$, is the $\Pev$-evolvent of $\mathbf Q$.}
\label{hexa}
\end{figure}

Similarly to the case of odd-gons, we will describe the conditions for the existence of a sequence of consecutive $\Pev$-evolvents for an even-gon. 

Namely, if for a given even-gon, the transformation $\mathcal S$ is a non-trivial rotation, then it has a unique $\Pev$-involute, and for this $\Pev$-involute the transformation $\mathcal S$ is again a rotation (the directions of the sides of a $\Pev$-involute differ from the directions of the sides of the initial polygon by $\dfrac{\pi}2$), so $\B(\alpha)$ stays unchanged modulo $\pi$. Thus, in this (generic) case, we have an infinite sequence of consecutive $\Pev$-evolvents. 

If $\B(\alpha)\equiv0\bmod\pi$, then $\mathcal S$ is a parallel translation by the vector$$2\sum_{k=1}^n(-1)^kp_k(\cos\beta_k,\sin\beta_k),$$where$$\beta_k=(-1)^k\alpha_k+2\sum_{j=k+1}^n(-1)^j\alpha_j$$(see Section \ref{evenn}). If this vector does not vanish, then there are no $\Pev$-involutes at all. But if it does vanish, then there is a two parameter family of $\Pev$-involutes, and for each of them, the transformation $\mathcal S$ is again a non-trivial parallel translation or the identity. It is the identity, if$$\sum(-1)^kp^\ast_k(\cos\beta^\ast_k,\sin\beta^\ast_k)=0,$$where $p^\ast_k$ and $\beta_k^\ast$ are calculated for the polygons in the family of involutes. However, $\beta^\ast_k$ are the same for all the involutes, and $p_k^\ast$ are linear functions of the two parameters of the family. So, generically, the vector of translation is zero for precisely one $\Pev$-involute. We  take it for $\Pev$-evolvent and conclude that, for almost all even-gons with $\B(\alpha)\equiv0\bmod\pi$, an infinite sequence of consecutive $\Pev$-evolvents still exists.

\subsubsection{Limiting behavior of $\Pev$-evolvents for equiangular hedgehogs}

\begin{proposition}
\label{nodd_Pinv_equi}
{\rm (A).} For $n$ odd, an equiangular hedgehog $\P$ with $n$ sides possesses a $\Pev$-involute if and only if one of the following equivalent conditions is satisfied:

{\rm (i)} the coefficient $a_0$ in \eqref{DiscrFourier} vanishes;

{\rm (ii)} $\P$ has zero perimeter: $\ell_1 + \cdots + \ell_n = 0$.\smallskip

{\rm (B).} For $n$ even, an equiangular hedgehog $\P$ with $n$ sides possesses a $\Pev$-involute if and only if one of the following equivalent conditions is satisfied:

{\rm (i)} the coefficients $a_0$ and $a_{n/2}$ in \eqref{DiscrFourier} vanish;

{\rm (ii)} the lengths of the even-numbered sides, as well as the lengths of the odd-numbered sides, sum up to zero:
\begin{equation}
\label{TwoPer}
\ell_1 + \ell_3 + \cdots + \ell_{n-1} = 0 = \ell_2 + \ell_4 + \cdots + \ell_n.
\end{equation}
\end{proposition}

\begin{proof}
The proof is similar to that of Proposition \ref{WhenInvol}. In the even case, the decomposition \eqref{DiscrFourier} is free of the $\mathbf{C}_{n/2}$ term if and only if $p$ is orthogonal to the vector $(1,-1, \ldots, 1, -1)$. Combining this with the orthogonality to $(1, 1, \ldots, 1)$ and using the formulas for $\ell$ in terms of $p$, one transforms this to $\sum_j \ell_j = \B(\ell) = 0$.
\end{proof}

Any  $\Pev$-involute  can be modified by a $\mathbf{C}_0$ or $\mathbf{C}_{n/2}$ summand, remaining a $\Pev$-involute. We have to avoid these summands in order to be able to iterate  $\Pev^{-1}$. Thus the $\Pev$-evolvent of an odd-gon is such a $\Pev$-involute which has no $\mathbf{C}_0$ component.  Using our knowledge of the eigenvalues of the matrix $\Pev_\theta$, we arrive at the following theorem.

\begin{theorem}
\label{nodd_Pinv_equi_limit}
{\rm (A)}. For $n$ odd, every equiangular hedgehog of zero perimeter has a unique infinite sequence of iterated $\Pev$-evolvents. Generically, the polygons in this sequence expand, tending in their shapes to two discrete hypocycloids of order $\dfrac{n-1}2$.\smallskip

{\rm (B)} For $n$ even, every equiangular hedgehog that satisfies \eqref{TwoPer} has a unique infinite sequence of $\Pev$-evolvents. Generically, the polygons in this sequence tend to two discrete hypocycloids of order $\dfrac{n}2 - 1$. Otherwise they shrink to the vertex centroid of the initial hedgehog, tending in shape to discrete hypocycloids of some other order.
\end{theorem}

\begin{proof} In the odd case, the iterated involutes are dominated by the $a\mathbf{C}_{\frac{n-1}2} + b\mathbf{S}_{\frac{n-1}2}$ term (which is generically non-zero). Indeed, this invariant subspace of $\Pev_\theta$ has the smallest eigenvalues $\pm i \dfrac{\sin(\pi/n)}{\sin(2\pi/n)}$. The absolute value of the latter is smaller than $1$, hence the $\Pev$-involutes expand (we apply the inverse of $\Pev_\theta$ on the subspace where it is defined).
	
\begin{figure}[htbp]
\centering
\includegraphics[width=3.8in]{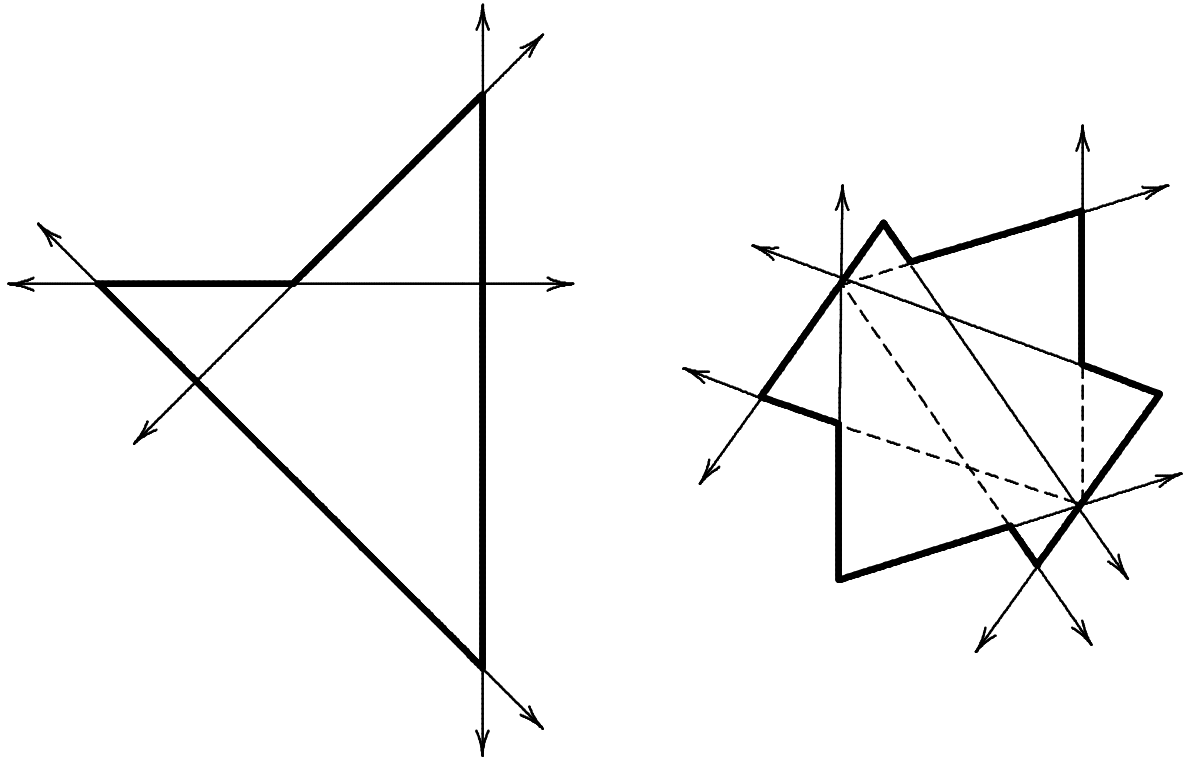}
\caption{Limiting shapes of iterated $\Pev$-evolvents of equiangular even-gons: discrete hypocycloids of order $\dfrac{n}2 - 1$ for $n=8$ (all lines occur twice with different orientations) and $n=10$.}
\label{EvenShapeInvol}
\end{figure}

Similarly, in the even case, the smallest eigenvalues are $\pm i$. They correspond to $a\mathbf{C}_1 + b\mathbf{S}_1$, which fix the center of hypocycloids, but also to $a\mathbf{C}_{\frac{n}2 - 1} + b\mathbf{S}_{\frac{n}2 - 1}$. All other terms tend to zero as the involute is iterated.
\end{proof}

The discrete hypocycloids of order $\dfrac{n}2 - 1$ are linear combinations of the hedgehogs
\begin{equation*}
\mathbf{C}_{\frac{n}2 - 1} = \left\{\left( \frac{2\pi j}n, (-1)^j\cos\frac{2\pi j}n \right) \right\} \mbox{ and }\,\mathbf{S}_{\frac{n}2 - 1} = \left\{\left( \frac{2\pi j}n, (-1)^{j-1}\sin\frac{2\pi j}n \right)\right\}.
\end{equation*}
In the hedgehog $a\mathbf{C}_{\frac{n}2 - 1} + b\mathbf{S}_{\frac{n}2 - 1}$, all lines with odd indices pass through the point $(-a,b)$, and all lines with even indices through $(a,-b)$ (see Lemma \ref{DHypocycl_onepoint}). We have already met these hedgehogs, for $n=6$, in Figure \ref{RegHexEvol}. Examples for $n=8$ and $n=10$ together with their $\Pev$-evolvents are shown in Figure \ref{EvenShapeInvol}.

\section{Future directions}\label{future_directions}

In the future, we aim to pursue similar problems in the set-up of the spherical and hyperbolic geometries. In the Euclidean case, the maps involved are, essentially, linear. For example,  a plane curve can be characterized by its support function $p(\alpha)$; then the support function of its evolute is $p'(\alpha-\pi/2)$. Thus the continuous problem reduces to the study of iterations of the derivative, or the inverse derivative, of a periodic function. In contrast, the constant curvature versions of the problem are highly non-linear.

\begin{figure}[hbtp]
\centering
\includegraphics[width=5in]{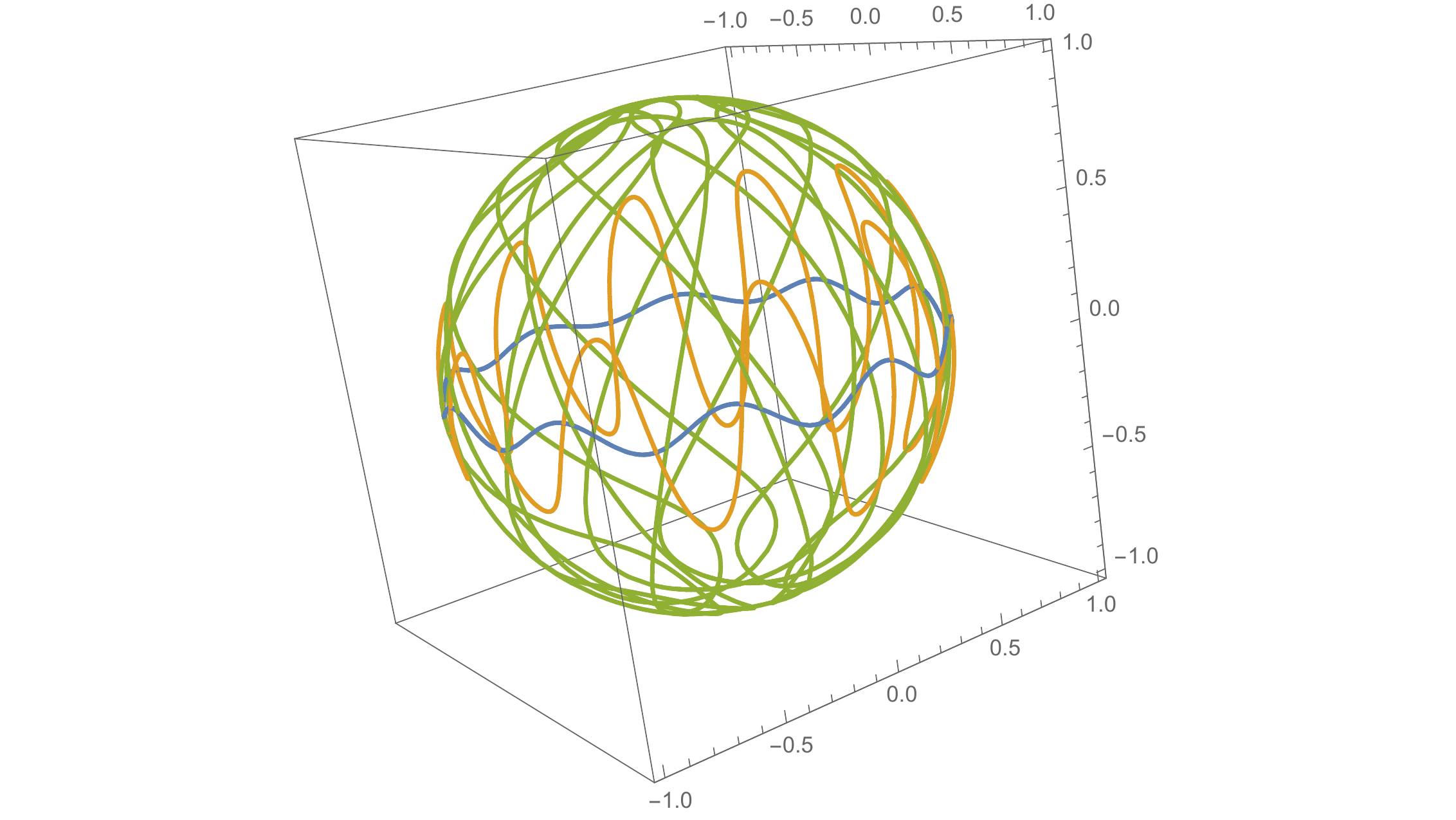}
\caption{Two iterations of the spherical derivative map.}
\label{sphere}
\end{figure}

For example, an equivalent formulation, in the spherical geometry, is as follows. Let $\gamma$ be a spherical arc-length parameterized curve. Consider the tangent indicatrix curve $\Gamma = \gamma'$. The curve $\Gamma$ bisects the area of the sphere. 
The `derivative' mapping $\gamma \mapsto \Gamma$ is a spherical analog of the evolute transformation in the plane, see Figure \ref{sphere}. 
One can also consider the `inverse derivative'  mapping $\Gamma \mapsto \gamma$ where $\gamma$ is chosen so that it bisects the area of the sphere. 
Concerning these mappings, it would be of interest to ascertain whether:
\begin{itemize}
\item the iterated derivatives of a curve, other than a circle, fill the sphere densely;
\item if the number of intersections of the iterated derivatives of a curve with a great circle is uniformly bounded, then the curve is a circle;
\item the inverse derivatives of an area-bisecting curve, sufficiently close to a great circle, converge to a great circle.
\end{itemize}
The latter question has a distinctive curve-shortening flavor, in a discrete setting.

In the hyperbolic plane, for the evolute to be bounded, the curve must be horocyclically convex. We are interested in the following questions:
\begin{itemize}
\item  if all the iterated evolutes of a curve are  horocyclically convex, does it imply that the curve is a circle?
\item do the iterated involutes of any curve converge to a point?
\end{itemize}

The discrete (polygonal) versions of these questions are also wide open and present a substantial challenge.

\end{document}